\documentclass[a4paper,reqno]{amsart}

\usepackage{graphicx}

\usepackage{mathtools}
\mathtoolsset{showonlyrefs=true}  
\usepackage{tikz}
\usetikzlibrary{patterns,shapes,decorations.pathmorphing,decorations.pathreplacing,calc,arrows}
\usepackage{mathdots}
\usepackage{amssymb}
\usepackage{amstext}
\usepackage{amsmath}
\usepackage{amscd}
\usepackage{amsthm}
\usepackage{amsfonts}
\usepackage{enumerate}
\usepackage{latexsym}
\usepackage{mathrsfs}
\usepackage{hyperref}
\usepackage[all,color]{xy}
\usepackage{pgfplots}
\pgfplotsset{compat=1.18}
\usepackage{booktabs}
\usepackage{multirow}
\usepackage{colortbl,array,xcolor}
\usepackage{longtable}
\usepackage{transparent}
\xyoption{all}
\usepackage{pstricks}
\usepackage{lscape}
\usepackage{bm}
\usepackage{todonotes}
\usepackage{url}
\usepackage[ruled,vlined]{algorithm2e}
\usepackage{algorithmic}
\usepackage{mathtools}
\usepackage{longtable}
\usepackage{hhline}

\usetikzlibrary{shadows.blur}
\newtheorem{theorem}{Theorem}[section]

\newtheorem{notation}[theorem]{Notation}
\newtheorem{lemma}[theorem]{Lemma}
\newtheorem{proposition}[theorem]{Proposition}
\newtheorem{definition-theorem}[theorem]{Definition-Theorem}
\newtheorem{definition-proposition}[theorem]{Definition-Proposition}
\newtheorem{problem}[theorem]{Problem}

\theoremstyle{definition}
\newtheorem{definition}[theorem]{Definition}

\newtheorem{remark}[theorem]{Remark}
\newtheorem{example}[theorem]{Example}

\renewcommand{\P}{\mathrm{P}}
\newcommand{\Z}{\mathbb{Z}}
\newcommand{\R}{\mathbb{R}}

\newcommand{\Hom}{\operatorname{Hom}\nolimits}
\newcommand{\rad}{\operatorname{rad}\nolimits}
\newcommand{\End}{\operatorname{End}\nolimits}

\newcommand{\RHom}{\mathbf{R}\strut\kern-.2em\operatorname{Hom}\nolimits}

\newcommand{\Conv}{{\rm Conv}}

\DeclareMathOperator{\proj}{\mathsf{proj}}

\DeclareMathOperator{\add}{\mathsf{add}}

\DeclareRobustCommand{\svdots}{
  \vbox{%
    \baselineskip=0.33333\normalbaselineskip
    \lineskiplimit=0pt
    \hbox{.}\hbox{.}\hbox{.}%
    \kern-0.2\baselineskip
  }%
}

\newcommand{\SigmaAe}{{\Sigma(A,e)}}
\newcommand{\SigmaeAe}{{\Sigma_{+-+}(A,e)}}

\newcommand{\cut}{\ar@{-}@[|(5)]}

\newcommand{\twosilt}{\mathsf{2\mbox{-}silt}}

\newcommand{\twopsilt}{\mathsf{2\mbox{-}presilt}}

\newcommand{\Kb}{\mathsf{K}^{\rm b}}

\newcommand{\twoips}{\mathsf{2\mbox{-}ips}}

\newcommand{\emax}{\sigma_{+-+,\max}}
\newcommand{\emin}{\sigma_{+-+,\min}}

\newcommand{\PAe}{{\rm P}(A,e)}
\newcommand{\Pe}{{\rm P}_{+-+}}
\newcommand{\PeAe}{{\rm P}_{+-+}(A,e)}
\newcommand{\PeX}{{\rm P}_{+-+}(X)}
\newcommand{\PeXZ}{{\rm P}_{+-+}(X)_{\mathbb{Z}}}

\newcommand{\conv}{\operatorname{conv}\nolimits}
\newcommand{\cone}{\operatorname{cone}\nolimits}
\numberwithin{equation}{section}

\usepackage{comment}
\DeclareRobustCommand{\svdots}{
  \vcenter{%
    \offinterlineskip
    \hbox{.}
    \vskip0.25\normalbaselineskip
    \hbox{.}
    \vskip0.25\normalbaselineskip
    \hbox{.}%
  }%
}

\begin{document}
	
\title[Fans and polytopes in tilting theory III]{Fans and polytopes in tilting theory III:\\ Classification of convex $g$-fans of rank 3}

\author{Toshitaka Aoki}
\address{Graduate School of Human Development and Environment, Kobe University, 3-11 Tsurukabuto, Nada-ku, Kobe 657-8501, Japan}
\email{toshitaka.aoki@people.kobe-u.ac.jp}

\author{Akihiro Higashitani}
\address{Department of Pure and Applied Mathematics, Graduate School of Information Science and Technology, Osaka University, 1-5 Yamadaoka, Suita, Osaka 565-0871, Japan}
\email{higashitani@ist.osaka-u.ac.jp}

\author{Osamu Iyama}
\address{Graduate School of Mathematical Sciences, University of Tokyo, 3-8-1 Komaba Meguro-ku Tokyo 153-8914, Japan}
\email{iyama@ms.u-tokyo.ac.jp}

\author{Ryoichi Kase}
\address{Faculty of Informatics, Okayama University of Science, 1-1 Ridaicho, Kita-ku, Okayama 700-0005, Japan}
\email{r-kase@ous.ac.jp}

\author{Yuya Mizuno}
\address{Faculty of Liberal Arts, Sciences and Global Education, Osaka Metropolitan University, 1-1 Gakuen-cho, Naka-ku, Sakai, Osaka 599-8531, Japan}
\email{yuya.mizuno@omu.ac.jp}

\begin{abstract}
The $g$-fan $\Sigma(A)$ of a finite dimensional algebra $A$ is a non-singular fan in its real Grothendieck group, defined by tilting theory. If the union $\P(A)$ of the simplices associated with the cones of $\Sigma(A)$ is convex, we call $A$ \emph{$g$-convex}. 
In this case, the \emph{$g$-polytope} $\P(A)$ of $A$ is a reflexive polytope. Thus, in each dimension, there are only finitely many isomorphism classes of fans that can be realized as $g$-fans of $g$-convex algebras. An important problem is to classify such fans for a fixed dimension $d$. In this paper, we give a complete answer for the case $d=3$: we prove that there are precisely $61$ convex $g$-fans of dimension $3$ up to isomorphism. Our method is based on the decomposition of fans into the $2^3$ orthants in the real Grothendieck group of $A$, together with a detailed analysis of possible sequences of $g$-vectors arising from iterated mutations.
\end{abstract}

\maketitle

\section{Introduction}
Tilting theory is basic to control equivalences of derived categories.
The class of silting complexes \cite{KV} complements the class of
tilting complexes from a point of view of mutation, which is a
categorical operation to replace a direct summand of a given silting
complex to construct a new silting complex \cite{AiI}.
The subclass of 2-term silting complexes gives rise to a fan
$\Sigma(A)$ (called the \emph{$g$-fan}) in the real Grothendieck group
$K_0(\proj A)_{\mathbb{R}}$ of a finite dimensional algebra $A$ such that the cones of $\Sigma(A)$ correspond bijectively with the isomorphism classes of
basic 2-term presilting complexes, see e.g. \cite{H1, H2, DF, DIJ, BST, AY, AsI, M, PY, P, B, As}. 
In \cite{AHIKM1}, we developed a basic theory
of $g$-fans, which form a special class of non-singular fans called
\emph{sign-coherent}. One of the basic results claims that $A$ is
\emph{$g$-finite} (that is, $A$ has only finitely many basic $2$-term silting complexes up to isomorphism) 
if and only if $\Sigma(A)$ is complete \cite[Theorem 4.7]{As}. 
The following is one of the basic problems
posed in \cite[Problem 1.7]{AHIKM1}.

\begin{problem}\label{characterization}
Characterize complete sign-coherent fans in $\R^d$ which can be
realized as a $g$-fan of some finite dimensional algebra.
\end{problem}

In the case $d=2$, a complete answer was given in \cite[Theorem 1.3]{AHIKM2}: 
For an arbitrary field $k$, any complete sign-coherent
fan in $\R^2$ can be realized as the $g$-fan of some finite dimensional $k$-algebra which is \emph{elementary} (that is, $A/\rad A$ is isomorphic to a product of $k$ as a $k$-algebra).

The aim of this article is to study a special case of Problem \ref{characterization} in the next case $d=3$. Precisely speaking, we consider the \emph{$g$-polytope} $\P(A)$, which is the union of the simplices associated to the cones of $\Sigma(A)$, and we call $A$ \emph{$g$-convex} if $\P(A)$ is convex. In this case, $A$ has to be $g$-finite, and the following result holds.

\begin{theorem}[{\cite[Theorem 1.10]{AHIKM1}}]
For each $g$-convex algebra $A$, the $g$-polytope $\P(A)$ is a reflexive polytope.
\end{theorem}

Motivated by this result, the following special case of Problem \ref{characterization} was posed in \cite[Problem 1.12]{AHIKM1}.

\begin{problem}\label{characterization 2}
Characterize reflexive polytopes in $\R^d$ which can be realized as the $g$-polytope of some finite dimensional algebra.
\end{problem}

Let $\Sigma=(\Sigma,\sigma_+)$ and $\Sigma'= (\Sigma',\sigma'_+)$ be sign-coherent fans in $\R^d$. Recall that an \emph{isomorphism $f\colon \Sigma\simeq\Sigma'$ of sign-coherent fans} is an isomorphism $f\colon \Z^d\to \Z^d$ of abelian groups such that the induced $\R$-linear isomorphism $\R^d\to\R^d$ gives a bijection $\Sigma\simeq\Sigma'$ between cones and $\{f(\sigma_+),-f(\sigma_+)\}=\{\sigma'_+,-\sigma'_+\}$ holds.
In the case $d=2$, the following answer to Problem \ref{characterization 2} is known.

\begin{theorem}[{\cite[Corollary 1.11]{AHIKM1} (cf. \cite[Theorem 1.6]{AHIKM2})}]
There are precisely 7 convex $g$-fans of rank 2 up to isomorphism of sign-coherent fans, and its complete set of representatives is given in Table \ref{tab:7convex_rank2}. 
\end{theorem}

\begin{table}[h]
    \centering
    \begin{tabular}{ccccccccc}
    {\begin{xy}
			0;<3.5pt,0pt>:<0pt,3.5pt>::
			(0,-5)="0",
			(-5,0)="1",
			(0,0)*{\bullet},
			(0,0)="2",
			(5,0)="3",
			(0,5)="4",
			(1.5,1.5)*{{\scriptstyle +}},
			(-1.5,-1.5)*{{\scriptstyle -}},
			\ar@{-}"0";"1",
			\ar@{-}"1";"4",
			\ar@{-}"4";"3",
			\ar@{-}"3";"0",
			\ar@{-}"2";"0",
			\ar@{-}"2";"1",
			\ar@{-}"2";"3",
			\ar@{-}"2";"4",
	\end{xy}}&
	{\begin{xy}
			0;<3.5pt,0pt>:<0pt,3.5pt>::
			(0,-5)="0",
			(-5,0)="1",
			(0,0)*{\bullet},
			(0,0)="2",
			(5,0)="3",
			(-5,5)="4",
			(0,5)="5",
			(1.5,1.5)*{{\scriptstyle +}},
			(-1.5,-1.5)*{{\scriptstyle -}},
			\ar@{-}"0";"1",
			\ar@{-}"1";"4",
			\ar@{-}"4";"5",
			\ar@{-}"5";"3",
			\ar@{-}"3";"0",
			\ar@{-}"2";"0",
			\ar@{-}"2";"1",
			\ar@{-}"2";"3",
			\ar@{-}"2";"4",
			\ar@{-}"2";"5",
	\end{xy}}&
	{\begin{xy}
			0;<3.5pt,0pt>:<0pt,3.5pt>::
			(0,-5)="0",
			(5,-5)="1",
			(-5,0)="2",
			(0,0)*{\bullet},
			(0,0)="3",
			(5,0)="4",
			(-5,5)="5",
			(0,5)="6",
			(1.5,1.5)*{{\scriptstyle +}},
			(-1.5,-1.5)*{{\scriptstyle -}},
			\ar@{-}"0";"2",
			\ar@{-}"2";"5",
			\ar@{-}"5";"6",
			\ar@{-}"6";"4",
			\ar@{-}"4";"1",
			\ar@{-}"1";"0",
			\ar@{-}"3";"0",
			\ar@{-}"3";"1",
			\ar@{-}"3";"2",
			\ar@{-}"3";"4",
			\ar@{-}"3";"5",
			\ar@{-}"3";"6",
	\end{xy}}&
	{\begin{xy}
			0;<3.5pt,0pt>:<0pt,3.5pt>::
			(0,-5)="0",
			(-5,0)="1",
			(0,0)*{\bullet},
			(0,0)="2",
			(5,0)="3",
			(-10,5)="4",
			(-5,5)="5",
			(0,5)="6",
			(1.5,1.5)*{{\scriptstyle +}},
			(-1.5,-1.5)*{{\scriptstyle -}},
			\ar@{-}"0";"3",
			\ar@{-}"3";"6",
			\ar@{-}"6";"4",
			\ar@{-}"4";"0",
			\ar@{-}"2";"0",
			\ar@{-}"2";"1",
			\ar@{-}"2";"3",
			\ar@{-}"2";"4",
			\ar@{-}"2";"5",
			\ar@{-}"2";"6",
	\end{xy}}&
	{\begin{xy}
			0;<3.5pt,0pt>:<0pt,3.5pt>::
			(0,-5)="0",
			(5,-5)="1",
			(-5,0)="2",
			(0,0)*{\bullet},
			(0,0)="3",
			(5,0)="4",
			(-10,5)="5",
			(-5,5)="6",
			(0,5)="7",
			(1.5,1.5)*{{\scriptstyle +}},
			(-1.5,-1.5)*{{\scriptstyle -}},
			\ar@{-}"0";"1",
			\ar@{-}"1";"4",
			\ar@{-}"4";"7",
			\ar@{-}"7";"5",
			\ar@{-}"5";"0",
			\ar@{-}"3";"0",
			\ar@{-}"3";"1",
			\ar@{-}"3";"2",
			\ar@{-}"3";"4",
			\ar@{-}"3";"5",
			\ar@{-}"3";"6",
			\ar@{-}"3";"7",
	\end{xy}}&
	{\begin{xy}
			0;<3.5pt,0pt>:<0pt,3.5pt>::
			(0,-5)="0",
			(5,-5)="1",
			(10,-5)="2",
			(-5,0)="3",
			(0,0)*{\bullet},
			(0,0)="4",
			(5,0)="5",
			(-10,5)="6",
			(-5,5)="7",
			(0,5)="8",
			(1.5,1.5)*{{\scriptstyle +}},
			(-1.5,-1.5)*{{\scriptstyle -}},
			\ar@{-}"0";"2",
			\ar@{-}"2";"8",
			\ar@{-}"8";"6",
			\ar@{-}"6";"0",
			\ar@{-}"4";"0",
			\ar@{-}"4";"1",
			\ar@{-}"4";"2",
			\ar@{-}"4";"3",
			\ar@{-}"4";"5",
			\ar@{-}"4";"6",
			\ar@{-}"4";"7",
			\ar@{-}"4";"8",
	\end{xy}}&
	{\begin{xy}
			0;<3.5pt,0pt>:<0pt,3.5pt>::
			(0,-2.5)="0",
			(5,-7.5)="1",
			(5,-2.5)="2",
			(-5,2.5)="3",
			(0,2.5)*{\bullet},
			(0,2.5)="4",
			(5,2.5)="5",
			(-10,7.5)="6",
			(0,7.5)="7",
			(-5,7.5)="8",
			(1.5,4)*{{\scriptstyle +}},
			(-1.5,1)*{{\scriptstyle -}},
			\ar@{-}"6";"1",
			\ar@{-}"1";"5",
			\ar@{-}"5";"7",
			\ar@{-}"7";"6",
			\ar@{-}"4";"0",
			\ar@{-}"4";"1",
			\ar@{-}"4";"2",
			\ar@{-}"4";"3",
			\ar@{-}"4";"5",
			\ar@{-}"4";"6",
			\ar@{-}"4";"7",
			\ar@{-}"4";"8",
	\end{xy}}
    \end{tabular}
    \caption{A complete set of representatives of convex $g$-fans of rank 2.}
    \label{tab:7convex_rank2}
\end{table}

In this paper, we give a complete answer to Problem \ref{characterization 2} in the
case $d=3$. 
To state our results, we need the following preparations.

For an $(A,B)$-bimodule $X$, let $t_A(X)$ (resp., $t(X)_B$) be the minimal number of generators of $X$ as a left $A$-module (resp., right $B$-module).
Let $A$ be a finite dimensional $k$-algebra of rank 3, 
and $e=(e_1,e_2,e_3)$ a complete set of pairwise orthogonal primitive idempotents of $A$. 
Then, we simply call the pair $(A,e)$ \emph{a finite dimensional $k$-algebra of rank 3}. 
In this case, we define a sign-coherent fan $\Sigma(A,e)$ in $\R^3$ as the image of the $g$-fan $\Sigma(A)$ via the isomorphism $K_0(\proj A)_{\mathbb{R}}\simeq\R^3$ given by $[e_iA]\mapsto \mathbf{e}_i$, where $\mathbf{e}_i$ denotes the $i$-th coordinate vector of $\mathbb{R}^3$. 
Moreover, we consider the following invariants. 
\begin{gather}
  l_{ij}^e:=t_{e_iAe_i}(e_iAe_j), \ r_{ij}^e := t(e_iAe_j)_{e_jAe_j}, \ h_{ij}^e := \begin{cases} 0 & \text{if $e_iAe_j \subseteq \rad^2 A$,} \\ 1 & \text{otherwise},\end{cases}\\
 d_{ij}^e := (l_{ij}^e,r_{ij}^e,h_{ij}^e) 
 \quad \text{and} \quad 
 d(A,e) := (d_{ij}^e)_{1\le i\ne j \le 3}.
\end{gather}

We omit the superscript $e$ when it is clear from the context.
We may visualize $d(A,e)$ as a matrix-like form as follows. 
\[
\begin{bmatrix}
    - & d_{12} & d_{13} \\ 
    d_{21} & - & d_{23} \\ 
    d_{31} & d_{32} & -
\end{bmatrix} 
= 
\begin{bmatrix}
        - & (l_{12},r_{12},h_{12}) & (l_{13},r_{13},h_{13}) \\
        (l_{21},r_{21},h_{21}) & - & (l_{23},r_{23},h_{23}) \\ 
        (l_{31},r_{31},h_{31}) & (l_{32},r_{32},h_{32})& - 
    \end{bmatrix}
\]
where $-$ stands for undefined entries.

\begin{theorem}\label{main theorem}
The following assertions hold.
\begin{enumerate}[\rm (1)]
\item Let $(A,e)$ be a finite dimensional $k$-algebra of rank 3. If $A$ is $g$-convex, then $\Sigma(A,e)$ is completely determined by $d(A,e)$. 
\item There are precisely 61 convex $g$-fans of rank 3 up to isomorphism of sign-coherent fans, and its complete set of representatives is given in Table \ref{fig:61alg&poly}.
\end{enumerate}
\end{theorem}


\bigskip \noindent 
{\bf Strategy.} 
Now, we explain our strategy of the proof of Theorem \ref{main theorem}. 
Let $(A,e=(e_1,e_2,e_3))$ be a finite dimensional $k$-algebra of rank 3, and $d(A,e)=(d_{ij}^e)_{1\leq i\neq j \leq 3}$ with $d_{ij}^e = (l_{ij}^e,r_{ij}^e,h_{ij}^e)$. 

When we say that $(\Sigma,\sigma_+)$ is a sign-coherent fan in $\R^3$, we assume that $\sigma_+$ is the positive cone 
$\cone \{\mathbf{e}_1,\mathbf{e}_2,\mathbf{e}_3\}$ of $\mathbb{R}^3$. 
Then, each isomorphism of sign-coherent fans in $\R^3$ is just the group action $G:=\mathfrak{S}_3\times \{\pm1\}$ on $\R^3$ given by
\[
G\ni g=(s,z)\mapsto zf_s \in \mathrm{GL}_3(\R),
\]
where $f_s$ is a linear transformation of $\mathbb{R}^3$ mapping $\bm{e}_i\mapsto \bm{e}_{s(i)}$ for each $1\leq i \leq 3$. 
Clearly, each element $g=(s,z)\in G$ gives an isomorphism
\begin{equation}\label{g iso}
\Sigma(A,e)\xrightarrow{\sim} \Sigma(A^g,e^g)
\end{equation}
of sign-coherent fans, where $A^g:=A$ if $z=1$ and $A^g:=A^{\rm op}$ if $z=-1$, and $e^g := (e_{s^{-1}(1)},e_{s^{-1}(2)},e_{s^{-1}(3)})$.

\bigskip
{\bf (1) Sign-decomposition.}
To each $\epsilon=(\epsilon_1,\ldots, \epsilon_d) \in \{\pm\}^d$, we associate the orthant $\R^d_\epsilon:=\cone\{\epsilon_i\mathbf{e}_i \mid 1\le i\le d\}$ of $\R^d$. For a sign-coherent fan $\Sigma=(\Sigma,\sigma_+)$ in $\R^d$, one can consider the subfan $\Sigma_\epsilon:=\{\sigma\in\Sigma\mid \sigma\in\R^d_\epsilon \}$.

For a finite dimensional $k$-algebra $(A,e)$ of rank $3$, let
\[
    d_{+-+}(A,e):=(d_{12}^e,d_{32}^e).
\]
A crucial step to prove Theorem \ref{main theorem} is the following, which claims that the datum $d_{+-+}(A,e)$ mostly determines $\Sigma_{+-+}(A,e)$.
Now, for $m\in \{0,\ldots,13\}$, let $d(m)=(d_{12},d_{32})$ be the datum given by Table \ref{tab:15list}. 
In addition, we define $d'(m)$ as the datum obtained from $d(m)$ by permuting a role of $d_{12}$ and $d_{32}$.

\begin{table}[ht]
\begin{tabular}{c|ccccccccc} \hline
$d_{12} \backslash d_{32}$ & $(0,0,0)$ & $(1,2,0)$ & $(1,1,0)$ & $(2,1,0)$ & $(1,2,1)$ & $(1,1,1)$ & $(2,1,1)$ \\
\hline 
$(0,0,0)$ &  $d(0)$ &&&&&&\\
$(1,2,0)$ & $-$ & $-$ & &&&&\\
$(1,1,0)$ & $-$ & $-$ &$-$ &&&&\\
$(2,1,0)$ & $-$ & $-$ & $-$ & $-$ &&&\\
$(1,2,1)$ & $d(2)$ & $d(11)$ & $d(10)$ &$-$ &$-$ & &\\ 
$(1,1,1)$ & $d(1)$ & $d(5)$ & $d(4)$ & $-$ & $d(12)$ & $d(7)$ &\\ 
$(2,1,1)$ & $d(3)$ & $-$ & $-$ & $d(6)$ &$d(13)$ & $d(8)$ &$d(9)$ \\ \hline 
\end{tabular}
\caption{ The datum $d(m)$ for $m\in \{0,\ldots,13\}$. }
\label{tab:15list}
\end{table}

\begin{theorem}\label{thm:orthant}
Let $(A,e)$ be a finite dimensional $k$-algebra of rank 3.
If $A$ is $g$-convex, then the following statements hold.
    \begin{enumerate}[\rm (1)]
    \item The possible values of $d_{+-+}(A,e)$ are precisely $d(m)$ and $d'(m)$ with $m\in \{0,\ldots, 13\}$.
    \item  If $d_{+-+}(A,e)=d(m)$ with $m\in \{0,\ldots,13\}$, 
    then the fan $\Sigma_{+-+}(A,e)$ is given by 
    \begin{equation}\label{eq:sdfans}
        \Sigma_{+-+}(A,e) = 
        \begin{cases}
            \Sigma_{d(m)} & \text{if $m\neq 10$}, \\ 
            \Sigma_{d(10),h_{13}^e} & \text{if $m=10$,}
        \end{cases}
    \end{equation}
        where the right-hand side is a non-singular fan in Table \ref{fig:15poly}. 
    \end{enumerate}
\end{theorem}

By symmetry, if $d_{+-+}(A,e) = d'(m)$ with $m\in \{0,\ldots,13\}$, then we have a similar statement for the fan $\Sigma_{+-+}(A,e)$ by multiplying the element $g=((13),1) \in G$. 

As an immediate consequence of Theorem \ref{thm:orthant}, it follows that $d(A,e)$ completely determines $\Sigma(A,e)$.
In fact, since $\Sigma_{+++}(A,e)$ and $\Sigma_{---}(A,e)$ are trivial, it suffices to determine $\Sigma_\epsilon(A,e)$ for each $\epsilon=(\epsilon_1,\epsilon_2,\epsilon_3)\in\{\pm\}^3\setminus\{(+++),(---)\}$.
For $g=(s,z)\in G$, the isomorphism \eqref{g iso} gives an isomorphism
\[
\Sigma_\epsilon(A,e) \simeq \Sigma_{g\cdot \epsilon}(A^g,e^g),
\]
where we naturally regard $\epsilon$ as an integer vector of $\mathbb{R}^3$. 
In particular, if we take $g=(s,z)\in G$ such that $g\cdot \epsilon = (+-+)$, then 
we have 
\[\Sigma_\epsilon(A,e)\simeq\Sigma_{+-+}(A^g,e^g).\]
Since $\Sigma_{+-+}(A^g,e^g)$ is determined by $d(A,e)$, so is $\Sigma_\epsilon(A,e)$.

\begin{table}[b]
\renewcommand{\arraystretch}{1.06}
\begin{tabular}{ccccccc}
    $\Sigma_{d(0)}$ & $\Sigma_{d(1)}$ & $\Sigma_{d(2)}$ \\ 
    \input{Tables/Figures_15poly/Figure_A0} & \input{Tables/Figures_15poly/Figure_A1} & \input{Tables/Figures_15poly/Figure_A2} \\       
     $\Sigma_{d(3)}$ & $\Sigma_{d(4)}$ & $\Sigma_{d(5)}$ \\ 
    \input{Tables/Figures_15poly/Figure_A3} & \input{Tables/Figures_15poly/Figure_A4} & \input{Tables/Figures_15poly/Figure_A5} \\  
     $\Sigma_{d(6)}$ & $\Sigma_{d(7)}$ & $\Sigma_{d(8)}$ \\ 
    \input{Tables/Figures_15poly/Figure_A6}& \input{Tables/Figures_15poly/Figure_A7} & \input{Tables/Figures_15poly/Figure_A8} \\ 
     $\Sigma_{d(9)}$ & $\Sigma_{d(10),0}$ &  $\Sigma_{d(10),1}$ \\
    \input{Tables/Figures_15poly/Figure_A9}& \input{Tables/Figures_15poly/Figure_A10} & \input{Tables/Figures_15poly/Figure_A11} \\ 
     $\Sigma_{d(11)}$& $\Sigma_{d(12)}$ & $\Sigma_{d(13)}$ \\ 
    \input{Tables/Figures_15poly/Figure_A12} & \input{Tables/Figures_15poly/Figure_A13} & \input{Tables/Figures_15poly/Figure_A14}  
\end{tabular}
\caption{Non-singular fans in the orthant $\mathbb{R}_{+-+}^3$. Each facet is given by triangle enclosed by blue lines.}
\label{fig:15poly}
\end{table}

\bigskip
{\bf (2) Narrowing down types of $g$-convex algebras.}
Our proof of Theorem 1.5 proceeds in the following three steps:
\begin{enumerate}[\rm (i)]
    \item By means of a computer program, we provide a list of 66 convex sign-coherent fans $\Sigma$ up to isomorphism such that, for each $\epsilon\in\{\pm\}^3\setminus\{(+++),(---)\}$, $\Sigma_\epsilon$ is one of the fans in Table \ref{fig:15poly}.
    Proposition \ref{Gen A} shows that the $g$-fan of any finite dimensional algebra of rank 3 appears in this list up to isomorphism. 
    \item We show that 61 of them, listed in Table \ref{fig:61alg&poly}, are in fact the $g$-fans of some finite dimensional algebras.
    \item In Proposition \ref{Gen X}, we show that the remaining 5, listed in Table \ref{tab:non-gfans}, cannot be realized as the $g$-fans of any finite dimensional algebra.
\end{enumerate}
The existence of these 5 polytopes contrasts with the result \cite[Theorem 1.3]{AHIKM2} for the case $d=2$, which asserts that the second and fourth quadrants of $g$-fans are completely independent.

\input{Tables/table_non_g_fan}


\begin{longtable}{c|ccccc}
\renewcommand{\arraystretch}{1.2}
{\bf No.} & {\rm Polytope} & {\rm Data} & {\rm Algebra} & {\rm $\#\twosilt$} \\ \hline \hline 
\endfirsthead

    \input{Tables/Figures_61alg_poly/poly1v8} \\ \hline
    \input{Tables/Figures_61alg_poly/poly2v10} \\ \hline 
    \input{Tables/Figures_61alg_poly/poly3v12-1} \\ \hline 
    \input{Tables/Figures_61alg_poly/poly4v12-2} \\ \hline 
    \input{Tables/Figures_61alg_poly/poly5v12-3} \\ \hline
    \input{Tables/Figures_61alg_poly/poly6v14-1} \\ \hline 
    \input{Tables/Figures_61alg_poly/poly7v14-2} \\ \hline 

    \input{Tables/Figures_61alg_poly/poly8v14-3} \\ \hline 
    \input{Tables/Figures_61alg_poly/poly9v14-4} \\ \hline
    \input{Tables/Figures_61alg_poly/poly10v14-5} \\ \hline
    \input{Tables/Figures_61alg_poly/poly11v16-1} \\ \hline 
    \input{Tables/Figures_61alg_poly/poly12v16-2} \\ \hline 
    \input{Tables/Figures_61alg_poly/poly13v16-3} \\ \hline 
    \input{Tables/Figures_61alg_poly/poly14v16-4} \\ \hline

    \input{Tables/Figures_61alg_poly/poly15v16-5} \\ \hline
    \input{Tables/Figures_61alg_poly/poly16v18-1} \\ \hline 
    \input{Tables/Figures_61alg_poly/poly17v18-2} \\ \hline 
    \input{Tables/Figures_61alg_poly/poly18v18-3} \\ \hline 
    \input{Tables/Figures_61alg_poly/poly19v18-4} \\ \hline
    \input{Tables/Figures_61alg_poly/poly20v18-5} \\ \hline
    \input{Tables/Figures_61alg_poly/poly21v20-1} \\ \hline 

    \input{Tables/Figures_61alg_poly/poly22v20-2} \\ \hline 
    \input{Tables/Figures_61alg_poly/poly23v20-3} \\ \hline 
    \input{Tables/Figures_61alg_poly/poly24v20-4} \\ \hline
    \input{Tables/Figures_61alg_poly/poly25v20-5} \\ \hline
    \input{Tables/Figures_61alg_poly/poly26v20-6} \\ \hline 
    \input{Tables/Figures_61alg_poly/poly27v20-7} \\ \hline 
    \input{Tables/Figures_61alg_poly/poly28v22-1} \\ \hline 

    \input{Tables/Figures_61alg_poly/poly29v22-2} \\ \hline
    \input{Tables/Figures_61alg_poly/poly30v22-3} \\ \hline
    \input{Tables/Figures_61alg_poly/poly31v22-4} \\ \hline 
    \input{Tables/Figures_61alg_poly/poly32v22-5} \\ \hline 
    \input{Tables/Figures_61alg_poly/poly33v24-1} \\ \hline 
    \input{Tables/Figures_61alg_poly/poly34v24-2} \\ \hline
    \input{Tables/Figures_61alg_poly/poly35v24-3} \\ \hline

    \input{Tables/Figures_61alg_poly/poly36v24-4} \\ \hline 
    \input{Tables/Figures_61alg_poly/poly37v24-5} \\ \hline 
    \input{Tables/Figures_61alg_poly/poly38v24-6} \\ \hline 
    \input{Tables/Figures_61alg_poly/poly39v26-1} \\ \hline
    \input{Tables/Figures_61alg_poly/poly40v26-2} \\ \hline
    \input{Tables/Figures_61alg_poly/poly41v26-3} \\ \hline 
    \input{Tables/Figures_61alg_poly/poly42v26-4} \\ \hline 

    \input{Tables/Figures_61alg_poly/poly43v26-5} \\ \hline 
    \input{Tables/Figures_61alg_poly/poly44v26-6} \\ \hline
    \input{Tables/Figures_61alg_poly/poly45v28-1} \\ \hline
    \input{Tables/Figures_61alg_poly/poly46v28-2} \\ \hline
    \input{Tables/Figures_61alg_poly/poly47v28-3} \\ \hline 
    \input{Tables/Figures_61alg_poly/poly48v28-4} \\ \hline 
    \input{Tables/Figures_61alg_poly/poly49v28-5} \\ \hline 

    \input{Tables/Figures_61alg_poly/poly50v28-6} \\ \hline
    \input{Tables/Figures_61alg_poly/poly51v30-1} \\ \hline
    \input{Tables/Figures_61alg_poly/poly52v30-2} \\ \hline
    \input{Tables/Figures_61alg_poly/poly53v32-1} \\ \hline 
    \input{Tables/Figures_61alg_poly/poly54v32-2} \\ \hline 
    \input{Tables/Figures_61alg_poly/poly55v32-3} \\ \hline 
    \input{Tables/Figures_61alg_poly/poly56v32-4} \\ \hline

    \input{Tables/Figures_61alg_poly/poly57v32-5} \\ \hline
    \input{Tables/Figures_61alg_poly/poly58v32-6} \\ \hline 
    \input{Tables/Figures_61alg_poly/poly59v32-7} \\ \hline 
    \input{Tables/Figures_61alg_poly/poly60v40} \\ \hline 
    \input{Tables/Figures_61alg_poly/poly61v48} \\ \hline

\caption{The list of $61$ convex $g$-fans of rank $3$ that forms a complete set of representatives up to isomorphism of sign-coherent fans} 
\label{fig:61alg&poly}
\end{longtable}

\section{Preliminaries}
In this section, we will explain the definitions of $g$-fans and $g$-polytopes, together with their fundamental properties. 
Throughout this section, let $A$ be a finite dimensional algebra over a field $k$. 
We fix a complete set $e = (e_1,\ldots,e_n)$ of pairwise orthogonal primitive idempotents of $A$. 
We have a canonical isomorphism $\phi_e\colon K_0(\proj A) \to \mathbb{Z}^n$ mapping $[e_iA] \mapsto \bm{e}_i$, where $\bm{e}_i$ denotes the $i$-th coordinate vector of $\mathbb{Z}^n$. 
In addition, it gives an isomorphism $K_0(\proj A)_{\mathbb{R}} := K_0(\proj A) \otimes_{\mathbb{Z}} \mathbb{R} \to \mathbb{R}^n$.

\subsection{The definition of $g$-fans} 
In this subsection, we recall the definition of $g$-fans, which arise from the $g$-vectors of $2$-term presilting complexes for finite dimensional algebras. We refer to \cite{AIR,DIJ,AHIKM1}. 
For the basics of silting theory \cite{AiI}, we deal with 2-term presilting complexes in this paper, so for simplicity, we will discuss the statement in the case of 2-term.

\begin{definition}\label{define silting}
Let $T=(T^i,d^i)\in\Kb(\proj A)$.
\begin{enumerate}[\rm(a)]
    \item $T$ is called a \emph{$2$-term complex} if $T^i = 0$ for all $i\not= 0,-1$.
    \item $T$ is called a \emph{$2$-term presilting complex} if $T$ is a $2$-term complex and $\Hom_{\Kb(\proj A)}(T,T[1])=0$.
    \item $T$ is called a \emph{$2$-term silting complex} if it is a $2$-term presilting complex and $|T|=|A|$.
\end{enumerate}
\end{definition}

We denote by $\twosilt A$ (respectively, $\twopsilt A$, $\twoips A$) the set of isomorphism classes of basic $2$-term silting (respectively, basic $2$-term presilting, indecomposable $2$-term presilting) complexes of $\Kb(\proj A)$. 
We recall that a finite dimensional algebra $A$ is called \emph{$g$-finite} if the set $\twosilt A$ is finite. 

The partial order defined below is important in silting mutation theory. 

\begin{definition-theorem}[\cite{AiI}] \label{defthm:poset}
For $T,U\in\twosilt A$, we write $U\leq T$ if $\Hom_{\Kb(\proj A)}(T,U[1])=0$. 
Then, $(\twosilt A,\leq)$ is a partially ordered set.
\end{definition-theorem}

We recall that the $g$-vector of a $2$-term complex $T=(T^{-1}\to T^{0})\in \Kb(\proj A)$ is defined to be the class $[T]=[T^{0}] - [T^{-1}]\in K_0(\proj A)$ in the Grothendieck group (\cite{AIR}). 
For our purpose, we will take the $g$-vector as a vector of Euclidean space $\mathbb{R}^n$ as follows: 
The \emph{$g$-vector} of $T$ is defined by 
$g(T) := \phi_e([T])$, where $\phi_e\colon K_0(\proj A) \to \mathbb{Z}^n$ is the isomorphism defined above. 
Notice that it depends on a choice of idempotents $e=(e_1,\ldots,e_n)$. 
By definition, if we set $T^0 \cong \bigoplus_{i=1}^n (e_iA)^{a_i}$ and $T^{-1} \cong \bigoplus_{i=1}^n (e_iA)^{b_i}$, then  
\begin{equation}
    g(T) = \left[
\begin{smallmatrix}a_1 -b_1 \\ \svdots \\ a_n - b_n \end{smallmatrix} 
\right]
\in \mathbb{Z}^n. 
\end{equation}

\begin{proposition}[{\cite[Proposition 2.5]{AIR}, \cite[Theorem 2.27]{AiI}}]\label{no common summand}
Let $T = (T^{-1}\to T^{0})$ be a basic $2$-term presilting complex. 
    \begin{enumerate}[\rm (1)]
        \item $T^0$ and $T^{-1}$ have no non-zero direct summands in common.
        \item Suppose that $T = T_1\oplus \cdots \oplus T_{\ell}$ with indecomposable $T_i$. Then, $g(T_1),\ldots,g(T_{\ell})$ forms a $\mathbb{Z}$-linearly independent set in $\mathbb{Z}^n$. 
    \end{enumerate}
\end{proposition}

Next, we recall the definition of $g$-fans. 

\begin{definition}
\begin{enumerate}[\rm (1)]
    \item For $T=T_1\oplus\cdots\oplus T_\ell\in\twopsilt A$ with indecomposable $T_i$, the \emph{$g$-vector cone} $C(T)$ of $T$ is a positive cone in $\mathbb{R}^n$ generated by $g(T_1),\ldots,g(T_{\ell})$: 
\begin{equation}
C(T) := \cone\{g(T_1), \ldots, g(T_\ell)\} = 
\left\{\sum_{i=1}^{\ell} a_ig(T_{i}) \mid a_{i} \geq 0 \right\}
\subset \mathbb{R}^n. 
\end{equation}
In this case, we also write it as a matrix form (depending on the ordering of summands): 
\[
C(T) = [g(T_1)| \cdots | g(T_{\ell})].
\]
For example, the $g$-vector cone $C(A)$ can be displayed as 
\[
C(A) = [\bm{e}_1 | \cdots | \bm{e}_n] = 
\left[
\begin{smallmatrix}1 \\ 0 \\ \svdots \\ 0\end{smallmatrix} 
\middle|
\cdots
\middle|
\begin{smallmatrix}0 \\ \svdots \\ 0 \\ 1\end{smallmatrix} 
\right]. 
\] 

\item The \emph{$g$-fan} of $A$ (with respect to $e$) is the set of cones:
\[\SigmaAe:=\{C(T)\mid T\in\twopsilt A\}.\]
\end{enumerate}
\end{definition}

Now, we give fundamental properties of a $g$-fan, that are needed in this paper. 
\begin{proposition}[\cite{DIJ}] \label{prop:basics_gfan}
We have the following. 
\begin{enumerate}[\rm (1)]
\item $T\mapsto C(T)$ gives a bijection $\twopsilt A \xrightarrow{\sim} \SigmaAe$.
\item $\SigmaAe$ is a nonsingular fan in $\mathbb{R}^n$ (i.e., a fan each of whose maximal cone is generated by a $\mathbb{Z}$-basis for $\mathbb{Z}^n$).
\item Any cone in $\SigmaAe$ of co-dimension $1$ is a face of precisely two maximal cones. 
\item Two maximal cones $\sigma,\sigma'$ of $\SigmaAe$ are adjacent to each other if and only if their corresponding $2$-term silting complexes $T$ and $T'$ differ in exactly one indecomposable direct summand. Thus, $T$ and $T'$ are related by mutation to each other. 
\item In (4), we suppose that $T = X\oplus \bigoplus_{i=1}^{n-1} T_i$ and $T = Y \oplus \bigoplus_{i=1}^{n-1} T_i$, in which case we have   
\[
\sigma=\cone\{g(T_1),\ldots,g(T_{n-1}), g(X)\} 
\quad \text{and} \quad 
\sigma' = \cone\{g(T_1),\ldots,g(T_{n-1}), g(Y)\}. 
\] 
Then, there exist non-negative integers $a_1,\ldots,a_{n-1}\geq 0$ satisfying 
\begin{equation} \label{eq:pw_nonnegativity}
    g(X)+g(Y) = \sum_{i=1}^{n-1} a_ig(T_i). 
\end{equation}
\end{enumerate}
\end{proposition}
\begin{remark}
\label{remark:basic_g-fan}
\begin{enumerate}[\rm (1)]
\item By Proposition \ref{prop:basics_gfan} (1) and (2), rays (i.e., $1$-dimensional cones) of $\SigmaAe$ are exactly $g$-vector cones of indecomposable $2$-term presilting complexes for $A$. 
In addition, the maximal cones of $\SigmaAe$ are exactly $g$-vector cones of $2$-term silting complexes for $A$. 
\item 
If we have another pairwise orthogonal primitive idempotents $e'=(e_1',\ldots,e_n')$ of $A$, 
then there is an isomorphism of fans $\SigmaAe\cong \Sigma(A,e')$.
\end{enumerate}
\end{remark}
We regard the set $\SigmaAe_n$ of maximal cones of $\Sigma(A,e)$ as a partially ordered set equipped with the induced partial order $\leq$ from that of $\twosilt A$ (see Definition-Theorem \ref{defthm:poset}).
By definition, it has the maximum element $C(A)$ and the minimum element $C(A[1]) = -C(A)$. 
By abuse of notation, we will write $(\SigmaAe,\leq)$ for this partially ordered set. 
More generally, for a given subset $\Sigma' \subseteq \SigmaAe$,  
we will write $(\Sigma',\leq)$ for the full subposet given by maximal cones in $\Sigma'$. 

Furthermore, the partial order of $\twosilt A$ can be 
interpreted as geometric information of the $g$-fan $\Sigma(A)$ as follows.
This shows that $g$-fans form a class of ordered fans \cite[Definition 4.15]{AHIKM1}.

\begin{theorem}[{\cite[Theorem 6.11]{DIJ}}]
\label{theorem:g-fan is ordered} 
The fan $\Sigma(A,e)$ satisfies both of the following conditions:
\begin{enumerate}[\rm(a)] 
\item For each pair of maximal cones $\sigma,\sigma'$ of $\SigmaAe$ adjacent to each other, 
a normal vector of $\sigma\cap\sigma'$ belongs to $C(A)$. 
\item For any maximal cones $\sigma,\sigma'$ of $\SigmaAe$, the following conditions are equivalent.
\begin{itemize}
\item There is an arrow $\sigma\to\sigma'$ in the Hasse quiver of $(\SigmaAe, \leq)$.
\item $\sigma$ and $\sigma'$ are adjacent to each other. 
Moreover, consider the hyperplane $\mathbb{H}:=\mathbb{R}(\sigma\cap\sigma')\subset \mathbb{R}^n$. 
Then $\sigma$ and $C(A)$ belong to the same connected component of $\mathbb{R}^n\setminus \mathbb{H}$.
\end{itemize}
\end{enumerate}
\end{theorem}

With notations in Proposition \ref{prop:basics_gfan}(5), 
we have either $\sigma \to \sigma'$ or $\sigma'\to \sigma$ by the above result. 
If this is the former case, then $T'$ is a left irreducible mutation of $T$, and $T$ is a right mutation of $T'$ simultaneously \cite{AiI}. 
In this case, we write $T'=\mu_X^-(T)$ and $T=\mu_Y^+(T')$. 
We will also use similar notations for not necessarily irreducible left/right mutations (We refer to \cite{AIR,AiI} for basics of silting mutation theory).

\subsection{Reduction method} 
\label{subsect:reduction}
In this subsection, we will describe the basic properties of the reduction of a 2-term presilting complex, along with the relationship between this reduction and the associated $g$-fan.

For a given $2$-term presilting complex $W\in \twopsilt A$ with $|W|=m$, 
let 
\[
    \twopsilt_W A := \{T\in \twopsilt A \mid W\in \add T\}. 
\]
Letting $T = T_1\oplus \cdots \oplus T_{n-m}\oplus W$ be the Bongartz completion of $W$ (see \cite{AIR}), 
we define the algebra $B_W := \End_{\Kb(\proj A)}(T)/[W]$ where $[W]$ denotes the ideal consisting of all endomorphisms factoring through $\add W$. 

We can summarize the relationship between $\Sigma(A)$ and $\Sigma(B_W)$ as follows. 

\begin{theorem}[{\cite[Theorem 4.11]{AHIKM1}}]
\label{theorem:jassoreduction}
\begin{enumerate}[\rm (1)]
  \item There exists a triangle functor $F_W:\Kb(\proj A) \to \Kb(\proj B_W)$ which sends $T$ to $B_W$
  and satisfies the following conditions:
  \begin{itemize}
  \item $F_W$ gives an isomorphism $K_0(\proj A)_\mathbb{R}/K_0(\add W)_\mathbb{R}\simeq K_0(\proj B_W)_\mathbb{R}$.
  \item $F_W$ gives a bijection 
  \[
    \{T\in \twopsilt_W A\mid |T|=d\} \simeq \{T'\in \twopsilt B_W\mid |T'|=d-m\}
    \] 
    for each $d\in \{m,m+1,\dots,n\}$.   
  \end{itemize}
  \item Take orthogonal primitive idempotents $e_1',\dots,e_{n-m}'$ of $B_W$ satisfying $F_W(T_i)\simeq e_i' B_W$
  and set $e'=(e_1',\dots,e_{n-m}')$. Define $\pi_{e,e'}:\mathbb{R}^n\to \mathbb{R}^{n-m}$ by
  \[
  \mathbb{R}^n \xrightarrow{\phi_e^{-1}} K_0(\proj A)_\mathbb{R} \twoheadrightarrow  
  K_0(\proj A)_\mathbb{R}/K_0(\add W)_\mathbb{R} \xrightarrow{F_W} K_0(\proj B_W)_\mathbb{R} \xrightarrow{\phi_{e'}} \mathbb{R}^{n-m},
  \]
  where $\twoheadrightarrow$ means the canonical surjection.
  Then we have 
  \[
  g(F_W(T)) = \pi_{e,e'} (g(T)) \quad \text{and} \quad 
  C(F_W(T)) = \pi_{e,e'} (C(T))
  \] 
  for any $T\in \twopsilt A$, and
  \[
  \Sigma(B_W,e')=\{C(F_W(T))\mid T\in \twopsilt_W A\}=\{\pi_{e,e'} (C(T))\mid T\in \twopsilt_W A\}.
  \]
  In particular, $F_W$ gives a bijection
  \begin{equation} \label{eq:reduction_fan}
    \Sigma_{W}(A,e) := \{C(T)\mid T\in \twopsilt_W A \} \xrightarrow{\sim} \Sigma(B_W,e')  
\end{equation}
\end{enumerate}
\end{theorem}
We call $\Sigma_W(A,e)$ the \emph{reduction} at $W$. 
In addition, if $W$ is indecomposable and $\bm{w}=g(W)$ is the $g$-vector of $W$, 
we call it the \emph{reduction} at $\bm{w}$ and write $\Sigma_{\bm{w}}(A,e)=\Sigma_W(A,e)$. 
\begin{remark}
 The bijection \eqref{eq:reduction_fan} is induced by $\pi_{e,e'}$.
Although the left-hand side is no longer a fan, it induces the following poset isomorphism (This isomorphism has been obtained in \cite{J}):
\begin{equation} \label{eq:reduction_poset}
(\Sigma_{W}(A,e),\leq) \xrightarrow{\sim} (\Sigma(B_W,e'),\leq). 
\end{equation}
\end{remark}

\subsection{$g$-polytopes and the convexity}
\label{subsec:g-convex}
In this subsection, we discuss the convexity of fans and polytopes. 
For $T=T_1\oplus \cdots \oplus T_{\ell}\in \twopsilt A$ with indecomposable $T_i$, we denote by $\Conv_0(T)$ the convex hull of $g(T_1),\ldots, g(T_{\ell})$ with the origin $\bm{0}$ in $\mathbb{R}^n$:  
\begin{equation}
\Conv_{0}(T) := \conv\{\bm{0}, g(T_1), \ldots, g(T_\ell)\} = 
\left\{\sum_{i=1}^{\ell} a_ig(T_{i}) \mid a_{i} \geq 0, \sum_{i=1}^{\ell} a_i \leq 1\right\}
\subset \mathbb{R}^n. 
\end{equation}
More generally, for a given finite collection $\mathcal{S}\subseteq\twopsilt A$, we set 
\begin{equation}
    \Conv_{0}(\mathcal{S}) := \conv\left(\{g(X)\mid \text{$X$ is an indecomposable direct summand of some $T\in \mathcal{S}$}\}\cup\{\bm{0}\}\right).
\end{equation}
Indeed, it is straightforward that $\Conv_0(T) = \Conv_0(\{T\})$ holds.

\begin{definition}
We define the \emph{$g$-polytope} of $A$ (with respect to $e$) by 
\[
\P(A,e) := \bigcup_{T\in \twosilt A} \Conv_0(T) \subset \mathbb{R}^n. 
\]     
\end{definition}

Thus, $\P(A,e)$ is the union of of the simplices associated to the maximal cones of $\Sigma(A,e)$. 
We note that it is not necessarily a convex set of $\mathbb{R}^n$ in general (see \cite[Example 5.3]{AHIKM1}). 
We say that $A$ is \emph{$g$-convex} (or $\SigmaAe$ is convex) if the $g$-polytope $\P(A,e)$ is convex. 
Notice that this property does not depend on the choice of idempotents. 
If $A$ is $g$-convex, then $\Conv_0(\mathcal{S})\subseteq \PAe$ holds for any $\mathcal{S}\subseteq \twopsilt A$ by the convexity. 

The following result is fundamental for $g$-convex algebras.
We recall that $A$ is said to be $g$-finite if 
$\twosilt A$ is finite, or equivalently $\Sigma(A,e)$ is finite. 

\begin{theorem}[{\cite[Theorem 5.10]{AHIKM1}}]
\label{theorem:g-convex=pairwise_g-convex}
\begin{enumerate}[{\rm (1)}]
\item $A$ is g-finite if and only if $\P(A,e)$ contains the origin in its interior. 
In this case, the origin is a unique lattice point in the interior of $\P(A,e)$.
\item The following statements are equivalent.
\begin{enumerate}[{\rm (a)}]
\item $A$ is $g$-convex.
\item The $g$-polytope $\P(A,e)$ is given as the convex hull 
\begin{equation}\label{eq:convex-hull_description}
\PAe = \Conv_{0}(\twoips A). 
\end{equation}
\item $A$ is $g$-finite and the following condition holds:
\begin{itemize}
\item For each pair of maximal cones $\sigma,\sigma'$ of $\SigmaAe$ adjacent to each other and given as in Proposition \ref{prop:basics_gfan}(5), we have the inequality $\sum_{i=1}^{n-1} a_i\leq 2$ for $a_i$'s in \eqref{eq:pw_nonnegativity}. 
\end{itemize}
\end{enumerate}
\end{enumerate}
\end{theorem} 

Next, we explain operations that preserve the $g$-convexity of given algebras. 
One is the reduction defined in Section \ref{subsect:reduction}. 
\begin{proposition}\label{prop:red_conv} 
Let $W = W_1\oplus \cdots \oplus W_m \in \twopsilt A$ with indecomposable $W_j$. 
\begin{enumerate}[\rm (1)]

\item Suppose that $T=X\oplus T_1\oplus \cdots \oplus T_{n-m-1} \oplus W$ and $T'=Y\oplus T_1\oplus \cdots \oplus T_{n-m-1} \oplus W$ are basic $2$-term silting complexes with indecomposable $X$, $Y$ and $T_i$'s,
and they are related by irreducible mutations to each other with satisfying 
\[
g(X)+g(Y)= \sum_{i=1}^{n-m-1} a_ig(T_i)+ \sum_{j=1}^m b_j g(W_j) 
\]
for non-negative integers $a_i$ and $b_j$. 
Then, 
\[
F_W(T) = F_W(X) \oplus \bigoplus_{i=1}^{n-m-1} F_W(T_i)
\quad \text{and} \quad 
F_W(T') = F_W(Y) \oplus \bigoplus_{i=1}^{n-m-1} F_W(T_i)
\]
are $2$-term silting complexes for $B_W$ that are related by mutation to each other with satisfying 
the equation 
\[
g(F_W(X)) + g(F_W(Y)) = \sum_{i=1}^{n-m-1} a_i g(F_W(T_i)). 
\]
\item If $A$ is $g$-convex, then the algebra $B_W$ is also $g$-convex.  
\end{enumerate}
\end{proposition}

\begin{proof}
(1) follows from Theorem \ref{theorem:jassoreduction}. (2) follows from (1) and Theorem \ref{theorem:g-convex=pairwise_g-convex}.
\end{proof}

The other is the idempotent truncation (see \cite[Section 4.2]{AHIKM1} for the detail). 

\begin{proposition}[{\cite[Theorem 4.18]{AHIKM1}}] \label{prop:idemp_fan}
Let $I\subseteq \{1,\dots,n\}$ and $e_I:=\sum_{i\in I} e_i$.
\begin{enumerate}[{\rm (1)}]
\item $-\otimes_{e_I A e_I} e_IA$ gives a bijection 
\[
    \twopsilt e_I A e_I \simeq \{T\in \twopsilt A\mid \text{the $j$-th entry of $g(T)$ is $0$ for each $j\notin I$}\}.
\]
\item The $g$-fan $\Sigma(e_I A e_I,(e_i)_{i\in I})$ is naturally identified with the subfan
of $\Sigma(A,e)$ consisting of those cones contained in the hyper-plane spanned by $(\mathbf{e}_i)_{i\in I}$. 

In particular, if $A$ is $g$-convex, then $e_I A e_I$ is also $g$-convex. 
\end{enumerate}
\end{proposition}

Finally, we give a complete list of all convex $g$-polygons (i.e., convex $g$-fans of rank $2$) in Table \ref{tab:convex_gfan_rank2} due to \cite[Theorem 6.3]{AHIKM1}. 
In addition, their isomorphism classes can be found in Table \ref{tab:7convex_rank2}. 

\begin{table}[h]
    \centering
\begin{tabular}{cccccccc}
{\begin{xy}
0;<3pt,0pt>:<0pt,3pt>::
(0,-5)="0-5",
(-5,0)="-50",
(0,0)*{\bullet},
(0,0)="00",
(5,0)="50",
(0,5)="05",
(-5,10)="-510",
(-5,5)="-55",
(-10,5)="-105",
\ar@{-}"00";"05",
\ar@{-}"00";"50",
\ar@{-}"00";"0-5",
\ar@{-}"00";"-50",
\ar@{-}"05";"50",
\ar@{-}"0-5";"-50",
\ar@{-}"50";"0-5",
\ar@{-}"-50";"05",
\end{xy}}&
{\begin{xy}
0;<3pt,0pt>:<0pt,3pt>::
(0,-5)="0-5",
(-5,0)="-50",
(0,0)*{\bullet},
(0,0)="00",
(5,0)="50",
(0,5)="05",
(5,-5)="5-5",
\ar@{-}"00";"05",
\ar@{-}"00";"50",
\ar@{-}"00";"0-5",
\ar@{-}"00";"-50",
\ar@{-}"05";"50",
\ar@{-}"-50";"0-5",
\ar@{-}"5-5";"00",
\ar@{-}"5-5";"50",
\ar@{-}"5-5";"0-5",
\ar@{-}"-50";"05",
\end{xy}}&
{\begin{xy}
0;<3pt,0pt>:<0pt,3pt>::
(0,-5)="0-5",
(-5,0)="-50",
(0,0)*{\bullet},
(0,0)="00",
(5,0)="50",
(0,5)="05",
(5,-5)="5-5",
(10,-5)="10-5",
\ar@{-}"00";"05",
\ar@{-}"00";"50",
\ar@{-}"00";"0-5",
\ar@{-}"00";"-50",
\ar@{-}"05";"50",
\ar@{-}"0-5";"-50",
\ar@{-}"5-5";"00",
\ar@{-}"50";"10-5",
\ar@{-}"00";"10-5",
\ar@{-}"5-5";"10-5",
\ar@{-}"5-5";"0-5",
\ar@{-}"-50";"05",
\end{xy}}&
{\begin{xy}
0;<3pt,0pt>:<0pt,3pt>::
(0,-5)="0-5",
(-5,0)="-50",
(0,0)*{\bullet},
(0,0)="00",
(5,0)="50",
(0,5)="05",
(5,-5)="5-5",
(5,-10)="5-10",
(10,-5)="10-5",
(-5,10)="-510",
(-5,5)="-55",
(-10,5)="-105",
\ar@{-}"00";"05",
\ar@{-}"00";"50",
\ar@{-}"00";"0-5",
\ar@{-}"00";"-50",
\ar@{-}"05";"50",
\ar@{-}"0-5";"-50",
\ar@{-}"5-5";"00",
\ar@{-}"00";"5-10",
\ar@{-}"5-5";"50",
\ar@{-}"5-5";"5-10",
\ar@{-}"5-10";"0-5",
\ar@{-}"-50";"05",
\end{xy}}&
{\begin{xy}
0;<3pt,0pt>:<0pt,3pt>::
(0,-5)="0-5",
(-5,0)="-50",
(0,0)*{\bullet},
(0,0)="00",
(5,0)="50",
(0,5)="05",
(5,-5)="5-5",
(5,-10)="5-10",
(10,-5)="10-5",
(-5,10)="-510",
(-5,5)="-55",
(-10,5)="-105",
\ar@{-}"00";"05",
\ar@{-}"00";"50",
\ar@{-}"00";"0-5",
\ar@{-}"00";"-50",
\ar@{-}"05";"50",
\ar@{-}"0-5";"-50",
\ar@{-}"50";"0-5",
\ar@{-}"-55";"00",
\ar@{-}"-55";"-50",
\ar@{-}"-55";"05",
\end{xy}}&
{\begin{xy}
0;<3pt,0pt>:<0pt,3pt>::
(0,-5)="0-5",
(-5,0)="-50",
(0,0)*{\bullet},
(0,0)="00",
(5,0)="50",
(0,5)="05",
(5,-5)="5-5",
(5,-10)="5-10",
(10,-5)="10-5",
(-5,10)="-510",
(-5,5)="-55",
(-10,5)="-105",
\ar@{-}"00";"05",
\ar@{-}"00";"50",
\ar@{-}"00";"0-5",
\ar@{-}"00";"-50",
\ar@{-}"05";"50",
\ar@{-}"0-5";"-50",
\ar@{-}"5-5";"00",
\ar@{-}"5-5";"50",
\ar@{-}"5-5";"0-5",
\ar@{-}"-55";"00",
\ar@{-}"-55";"-50",
\ar@{-}"-55";"05",
\end{xy}}&
{\begin{xy}
0;<3pt,0pt>:<0pt,3pt>::
(0,-5)="0-5",
(-5,0)="-50",
(0,0)*{\bullet},
(0,0)="00",
(5,0)="50",
(0,5)="05",
(5,-5)="5-5",
(5,-10)="5-10",
(10,-5)="10-5",
(-5,10)="-510",
(-5,5)="-55",
(-10,5)="-105",
\ar@{-}"00";"05",
\ar@{-}"00";"50",
\ar@{-}"00";"0-5",
\ar@{-}"00";"-50",
\ar@{-}"05";"50",
\ar@{-}"0-5";"-50",
\ar@{-}"5-5";"00",
\ar@{-}"50";"10-5",
\ar@{-}"00";"10-5",
\ar@{-}"5-5";"10-5",
\ar@{-}"5-5";"0-5",
\ar@{-}"-55";"00",
\ar@{-}"-55";"-50",
\ar@{-}"-55";"05",
\end{xy}}&
{\begin{xy}
0;<3pt,0pt>:<0pt,3pt>::
(0,-5)="0-5",
(-5,0)="-50",
(0,0)*{\bullet},
(0,0)="00",
(5,0)="50",
(0,5)="05",
(5,-5)="5-5",
(5,-10)="5-10",
(10,-5)="10-5",
(-5,10)="-510",
(-5,5)="-55",
(-10,5)="-105",
\ar@{-}"00";"05",
\ar@{-}"00";"50",
\ar@{-}"00";"0-5",
\ar@{-}"00";"-50",
\ar@{-}"05";"50",
\ar@{-}"0-5";"-50",
\ar@{-}"5-5";"00",
\ar@{-}"00";"5-10",
\ar@{-}"5-5";"50",
\ar@{-}"5-5";"5-10",
\ar@{-}"5-10";"0-5",
\ar@{-}"-55";"00",
\ar@{-}"-55";"-50",
\ar@{-}"-55";"05",
\end{xy}}\\ & & & & & & & \\
{\begin{xy}
0;<3pt,0pt>:<0pt,3pt>::
(0,-5)="0-5",
(-5,0)="-50",
(0,0)*{\bullet},
(0,0)="00",
(5,0)="50",
(0,5)="05",
(5,-5)="5-5",
(5,-10)="5-10",
(10,-5)="10-5",
(-5,10)="-510",
(-5,5)="-55",
(-10,5)="-105",
\ar@{-}"00";"05",
\ar@{-}"00";"50",
\ar@{-}"00";"0-5",
\ar@{-}"00";"-50",
\ar@{-}"05";"50",
\ar@{-}"0-5";"-50",
\ar@{-}"50";"0-5",
\ar@{-}"-55";"00",
\ar@{-}"-510";"00",
\ar@{-}"05";"-510",
\ar@{-}"-510";"-55",
\ar@{-}"-55";"-50",
\end{xy}}&
{\begin{xy}
0;<3pt,0pt>:<0pt,3pt>::
(0,-5)="0-5",
(-5,0)="-50",
(0,0)*{\bullet},
(0,0)="00",
(5,0)="50",
(0,5)="05",
(5,-5)="5-5",
(5,-10)="5-10",
(10,-5)="10-5",
(-5,10)="-510",
(-5,5)="-55",
(-10,5)="-105",
\ar@{-}"00";"05",
\ar@{-}"00";"50",
\ar@{-}"00";"0-5",
\ar@{-}"00";"-50",
\ar@{-}"05";"50",
\ar@{-}"0-5";"-50",
\ar@{-}"5-5";"00",
\ar@{-}"5-5";"50",
\ar@{-}"5-5";"0-5",
\ar@{-}"-55";"00",
\ar@{-}"-510";"00",
\ar@{-}"05";"-510",
\ar@{-}"-510";"-55",
\ar@{-}"-55";"-50",
\end{xy}}&
{\begin{xy}
0;<3pt,0pt>:<0pt,3pt>::
(0,-5)="0-5",
(-5,0)="-50",
(0,0)*{\bullet},
(0,0)="00",
(5,0)="50",
(0,5)="05",
(5,-5)="5-5",
(5,-10)="5-10",
(10,-5)="10-5",
(-5,10)="-510",
(-5,5)="-55",
(-10,5)="-105",
\ar@{-}"00";"05",
\ar@{-}"00";"50",
\ar@{-}"00";"0-5",
\ar@{-}"00";"-50",
\ar@{-}"05";"50",
\ar@{-}"0-5";"-50",
\ar@{-}"5-5";"00",
\ar@{-}"50";"10-5",
\ar@{-}"00";"10-5",
\ar@{-}"5-5";"10-5",
\ar@{-}"5-5";"0-5",
\ar@{-}"-55";"00",
\ar@{-}"-510";"00",
\ar@{-}"05";"-510",
\ar@{-}"-510";"-55",
\ar@{-}"-55";"-50",
\end{xy}}&
{\begin{xy}
0;<3pt,0pt>:<0pt,3pt>::
(0,-5)="0-5",
(-5,0)="-50",
(0,0)*{\bullet},
(0,0)="00",
(5,0)="50",
(0,5)="05",
(5,-5)="5-5",
(5,-10)="5-10",
(10,-5)="10-5",
(-5,10)="-510",
(-5,5)="-55",
(-10,5)="-105",
\ar@{-}"00";"05",
\ar@{-}"00";"50",
\ar@{-}"00";"0-5",
\ar@{-}"00";"-50",
\ar@{-}"05";"50",
\ar@{-}"0-5";"-50",
\ar@{-}"5-5";"00",
\ar@{-}"00";"5-10",
\ar@{-}"5-5";"50",
\ar@{-}"5-5";"5-10",
\ar@{-}"5-10";"0-5",
\ar@{-}"-55";"00",
\ar@{-}"-510";"00",
\ar@{-}"05";"-510",
\ar@{-}"-510";"-55",
\ar@{-}"-55";"-50",
\end{xy}}&
{\begin{xy}
0;<3pt,0pt>:<0pt,3pt>::
(0,-5)="0-5",
(-5,0)="-50",
(0,0)*{\bullet},
(0,0)="00",
(5,0)="50",
(0,5)="05",
(5,-5)="5-5",
(5,-10)="5-10",
(10,-5)="10-5",
(-5,10)="-510",
(-5,5)="-55",
(-10,5)="-105",
\ar@{-}"00";"05",
\ar@{-}"00";"50",
\ar@{-}"00";"0-5",
\ar@{-}"00";"-50",
\ar@{-}"05";"50",
\ar@{-}"0-5";"-50",
\ar@{-}"50";"0-5",
\ar@{-}"-55";"00",
\ar@{-}"-105";"00",
\ar@{-}"05";"-55",
\ar@{-}"-55";"-105",
\ar@{-}"-105";"-50",
\end{xy}}&
{\begin{xy}
0;<3pt,0pt>:<0pt,3pt>::
(0,-5)="0-5",
(-5,0)="-50",
(0,0)*{\bullet},
(0,0)="00",
(5,0)="50",
(0,5)="05",
(5,-5)="5-5",
(5,-10)="5-10",
(10,-5)="10-5",
(-5,10)="-510",
(-5,5)="-55",
(-10,5)="-105",
\ar@{-}"00";"05",
\ar@{-}"00";"50",
\ar@{-}"00";"0-5",
\ar@{-}"00";"-50",
\ar@{-}"05";"50",
\ar@{-}"0-5";"-50",
\ar@{-}"5-5";"00",
\ar@{-}"5-5";"50",
\ar@{-}"5-5";"0-5",
\ar@{-}"-55";"00",
\ar@{-}"-105";"00",
\ar@{-}"05";"-55",
\ar@{-}"-55";"-105",
\ar@{-}"-105";"-50",
\end{xy}}&
{\begin{xy}
0;<3pt,0pt>:<0pt,3pt>::
(0,-5)="0-5",
(-5,0)="-50",
(0,0)*{\bullet},
(0,0)="00",
(5,0)="50",
(0,5)="05",
(5,-5)="5-5",
(5,-10)="5-10",
(10,-5)="10-5",
(-5,10)="-510",
(-5,5)="-55",
(-10,5)="-105",
\ar@{-}"00";"05",
\ar@{-}"00";"50",
\ar@{-}"00";"0-5",
\ar@{-}"00";"-50",
\ar@{-}"05";"50",
\ar@{-}"0-5";"-50",
\ar@{-}"5-5";"00",
\ar@{-}"50";"10-5",
\ar@{-}"00";"10-5",
\ar@{-}"5-5";"10-5",
\ar@{-}"5-5";"0-5",
\ar@{-}"-55";"00",
\ar@{-}"-105";"00",
\ar@{-}"05";"-55",
\ar@{-}"-55";"-105",
\ar@{-}"-105";"-50",
\end{xy}}&
{\begin{xy}
0;<3pt,0pt>:<0pt,3pt>::
(0,-5)="0-5",
(-5,0)="-50",
(0,0)*{\bullet},
(0,0)="00",
(5,0)="50",
(0,5)="05",
(5,-5)="5-5",
(5,-10)="5-10",
(10,-5)="10-5",
(-5,10)="-510",
(-5,5)="-55",
(-10,5)="-105",
\ar@{-}"00";"05",
\ar@{-}"00";"50",
\ar@{-}"00";"0-5",
\ar@{-}"00";"-50",
\ar@{-}"05";"50",
\ar@{-}"0-5";"-50",
\ar@{-}"5-5";"00",
\ar@{-}"00";"5-10",
\ar@{-}"5-5";"50",
\ar@{-}"5-5";"5-10",
\ar@{-}"5-10";"0-5",
\ar@{-}"-55";"00",
\ar@{-}"-105";"00",
\ar@{-}"05";"-55",
\ar@{-}"-55";"-105",
\ar@{-}"-105";"-50",
\end{xy}}
\end{tabular}
    \caption{Convex $g$-fans of rank $2$}
    \label{tab:convex_gfan_rank2}
\end{table}

\subsection{Sign-coherent property and sign-decomposition}
In this section, we explain sign-coherent property and sign-decomposition of $g$-fans.

To each $\epsilon = (\epsilon_1,\ldots,\epsilon_n) \in \{\pm\}^n$, we associate the orthant $\mathbb{R}^n_{\epsilon} := \cone\{\epsilon_i\bm{e}_i\mid 1 \leq i \leq n\} \subseteq \mathbb{R}^n$. 
The following result is known as the \emph{sign-coherent property} of $g$-vectors.

\begin{proposition}[{\cite{DIJ}}] \label{prop:sign-coherent}
The $g$-vector cone of each $2$-term presilting complex $T$ is contained in $\mathbb{R}^n_{\epsilon}$ for some $\epsilon \in \{\pm\}^n$. 
\end{proposition}

Let  
\[
e_{\epsilon}^+ = \sum_{\epsilon_i = +} e_{i}  \quad \text{and} \quad 
e_{\epsilon}^- = \sum_{\epsilon_i = -} e_{i} 
\]
be idempotents of $A$. 
We denote by $\twopsilt_{\epsilon} A$ the subset of $2$-term presilting complexes whose $g$-vector cones are contained in $\mathbb{R}^n_{\epsilon}$. 
In other words, it consists of those complexes whose $0$-th (resp., $(-1)$-th) term lies in $\add e_{\epsilon}^+A$ (resp., $\add e_{\epsilon}^-A$). 
Similarly, the sets $\twoips_{\epsilon} A$ and $\twosilt_{\epsilon} A$ are defined. 
\begin{lemma}
\label{lemma:orthant_completion}
Let $\epsilon\in \{\pm\}^n$. 
For any $W\in \twopsilt_{\epsilon} A$, there exists $T\in \twosilt_{\epsilon} A$ such that $W\in \add T$.
\end{lemma}
\begin{proof}
We show the assertion by induction on $n:=|A|$.
If $n=1$, then the assertion is clear. 

Assume that $n\ge 2$. If each entry of $g(W)$ is non-zero, then the assertion follows from Proposition \ref{prop:sign-coherent}.
Therefore, we may assume that the $j$-th entry of $g(W)$ is $0$ for some $j$.
By Proposition \ref{prop:idemp_fan} and the induction hypothesis, there is $U\in \twopsilt_\epsilon A$ satisfying the following conditions:
\begin{itemize}
\item $|U|=n-1$ and $W\in \add U$.
\item $g(U)\in \sum_{i\ne j} \mathbb{R}_{>0}\epsilon_i \bm{e}_i$.
\end{itemize}
Therefore, if $\epsilon_j=+$ (respectively, $\epsilon_j =-$), 
then the Bongartz completion (respectively, co-Bongartz completion) of $U$ has the desired property. 
\end{proof}
Under these notations, we define 
\begin{equation} \label{eq:sp_epsilon}
    \Sigma_{\epsilon}(A,e) 
    := \{C(T) \mid T\in \twopsilt_{\epsilon} A\} 
    \quad \text{and} \quad 
    \P_{\epsilon}(A,e) 
    := \bigcup_{T\in \twosilt_{\epsilon} A} \Conv_0(T).  
\end{equation}

By Lemma\;\ref{lemma:orthant_completion}, 
these objects are the restrictions of $\SigmaAe$ and $\P(A,e)$ to the region $\mathbb{R}_{\epsilon}^n$ respectively. 
Conversely, by the sign-coherent property (Proposition \ref{prop:sign-coherent}), one can recover the entire of $\twopsilt A$, $\SigmaAe$ and $\P(A,e)$, respectively, as 
\begin{equation}
    \twopsilt A = \bigcup_{\epsilon\in \{\pm\}^n}
    \twopsilt_{\epsilon} A, \quad 
    \SigmaAe = \bigcup_{\epsilon\in \{\pm\}^n} \Sigma_{\epsilon}(A,e)
    \quad \text{and} \quad 
    \PAe = \bigcup_{\epsilon\in \{\pm\}^n} \P_{\epsilon}(A,e). 
\end{equation}

Regarding the (sub-)poset structure, we will use the following result. 

\begin{proposition}[{\cite[Proposition~3.2]{Ao}}] \label{epsilon_interval}
For each $\epsilon\in \{\pm\}^n$, the full subposet $(\twosilt_{\epsilon} A ,\leq)$ is an interval and given by 
\begin{equation} \label{eq:ep_int}
(\twosilt_{\epsilon} A ,\leq)=  [\mu^+_{e_{\epsilon}^+ A[1]}(A[1]), \ \mu^-_{e_{\epsilon}^-A}(A)].
\end{equation}
\end{proposition}

A similar statement holds for the full subposet $(\Sigma_{\epsilon}(A,e),\leq)$. 
Indeed, it is isomorphic to the interval \eqref{eq:ep_int} via the bijection given in Proposition \ref{prop:basics_gfan}(1). Thus, it has the minimum and maximum elements given by 
\begin{equation} \label{emin_emax}
\sigma_{\epsilon,\min} := C(\mu^+_{e_{\epsilon}^+ A[1]}(A[1])) \quad \text{and}\quad 
\sigma_{\epsilon,\max} := C(\mu^-_{e_{\epsilon}^-A}(A))
\end{equation}
respectively. 

\subsection{$\mathcal{H}$-description for convex $g$-polytopes}
When $A$ is $g$-convex, then the equation \eqref{eq:convex-hull_description} gives a convex-hull description of the $g$-polytope $\PAe$. 
In this section, we consider to describe it as a certain intersection of half-spaces. 

We assume that $A$ is $g$-convex. 
In this case, $\SigmaAe$ has only finitely many cones. 
For $\epsilon \in \{\pm\}^n$, we recall that the restriction $\P_{\epsilon}(A,e)$ of $\PAe$ is defined in \eqref{eq:sp_epsilon}. 
By Theorem \ref{theorem:g-convex=pairwise_g-convex}(2), we have 
\[
\P_{\epsilon}(A,e) = \Conv_{0}(\twoips_{\epsilon} A).   
\]
In particular, $\P_{\epsilon}(A,e)$ is convex. 
By Theorem \ref{theorem:g-convex=pairwise_g-convex}(1), we have
\[
\P_{\epsilon}(A,e)_{\mathbb{Z}}:=\P_{\epsilon}(A,e)\cap \mathbb{Z}^n\setminus \{\bm{0}\}=\{g(W)\mid W\in \twoips_{\epsilon} A\}.
\]

Now, we take an arbitrary maximal cone $\tau = [\bm{w}_1|\ldots|\bm{w}_n] \in \SigmaAe$. 
Since it is $n$-dimensional, there is a unique integer vector, say $\bm{v}_{\tau}\in \mathbb{Z}^n$, satisfying 
\begin{equation}
    \langle \bm{w}_{j}, \bm{v}_{\tau} \rangle = 1 \quad \text{for all} \ 1\leq j \leq n.
\end{equation}
We set   
\begin{eqnarray}
    \mathbb{H}_\tau := \{\bm{x}\in \mathbb{R}^n \mid \langle \bm{x},\bm{v}_{\tau}  \rangle = 1\} \quad \text{and} \quad 
    \mathbb{H}^{\leq 1}_\tau := \{\bm{x}\in \mathbb{R}^n \mid \langle \bm{x},\bm{v}_{\tau} \rangle \leq 1\}. 
\end{eqnarray}
Notice that $\mathbb{H}_{\tau}$ is precisely the hyperplane of co-dimension $1$ which contains all $\bm{w}_i$'s. 
In addition, for a given collection $X$ of maximal cones of $\SigmaAe$, we define the region 
\begin{eqnarray} 
    \P(X) &:=&  
        \displaystyle 
        \bigcap_{\tau\in X} \mathbb{H}^{\leq 1}_{\tau}.
\end{eqnarray} 
Since $\P(A,e)$ is convex, we have that $\P(A,e) \subseteq \P(X)$ for an arbitrary $X$, 
and the equality holds if we take $X$ as the set of all maximal cones of $\SigmaAe$. 
In particular, for any integer vector $\bm{x}\in \mathbb{Z}^n$, we have $\bm{x}\in \P(A,e)$ if and only if $
\langle\bm{x}, \bm{v}_{\tau}\rangle\leq 1$ for all maximal cones $\tau$ of $\SigmaAe$.

In addition, for $\epsilon\in \{\pm\}^n$, let 
\[
\P_{\epsilon}(X) := \P(X) \cap \mathbb{R}_{\epsilon}^n \quad \text{and} \quad \P_{\epsilon}(X)_{\mathbb{Z}} := \P_{\epsilon}(X) \cap \mathbb{Z}^n \setminus \{\bm{0}\}.
\] 
From the convexity, we have the following inclusions: 
\[
\P_{\epsilon}(A,e) \subseteq \P_{\epsilon}(X) \quad \text{and} \quad 
\P_{\epsilon}(A,e)_{\mathbb{Z}} \subseteq \P_{\epsilon}(X)_{\mathbb{Z}}. 
\]

Now, we give a sufficient condition for $X$ such that $\P_{\epsilon}(A,e)$ coincides with $\P_{\epsilon}(X)$, that will be helpful in our proof of Theorem \ref{thm:orthant} in Section \ref{sec:proof_Thm2}. 

\begin{proposition}\label{intermediate}
    Let $\epsilon\in \{\pm\}^n$. 
    For a set $X$ of maximal cones of $\Sigma_{\epsilon}(A,e)$, we assume that it satisfies the following condition {\rm (H)}.
    \begin{itemize}
        \item[\rm (H)] $\det [\bm{u}_1\,\cdots\,\bm{u}_n]\in \{0,\pm 1\}$ holds for any elements $\bm{u}_1,\ldots,\bm{u}_n\in \{\bm{v}_{\tau}\mid \tau \in X\}\cup \{-\epsilon_i\bm{e}_i\}_{i=1}^n$. 
    \end{itemize}
    If moreover $\P_{\epsilon}(A,e) \supseteq \P_{\epsilon}(X)_{\mathbb{Z}}$ holds, then we have
    \begin{eqnarray}  
        \P_{\epsilon}(A,e) = \P_{\epsilon}(X) 
        \quad \text{and} \quad 
        \P_{\epsilon}(A,e)_{\mathbb{Z}} = \P_{\epsilon}(X)_{\mathbb{Z}}. 
    \end{eqnarray} 
\end{proposition}
\begin{proof}
Let $X=\{\tau_1,\ldots, \tau_{\ell}\}$ and $\bm{v}_{j} := \bm{v}_{\tau_j}$ for all $j$. 
By definition, $\P_{\epsilon}(X)$ is the region bounded by the hyperplanes
\begin{equation} \label{eq_hyperplanes} 
\{\bm{x}\mid \langle \bm{x}, \bm{v}_j\rangle= 1\} \quad \text{and}\quad 
\{\bm{x}\mid \langle \bm{x}, -\epsilon_{i} \bm{e}_i \rangle = 0\}. 
\end{equation}

To prove the claim, since $\P_{\epsilon}(A,e) \supseteq \P_{\epsilon}(X)_{\mathbb{Z}}$ holds by our assumption, 
it is enough to show that every vertex $\bm{v}$ of $\P_{\epsilon}(X)$ lies in $\P_{\epsilon}(X)_{\mathbb{Z}}$. 
For such a vertex $\bm{v}$, we can choose $n$ hyperplanes
$\mathbb{H}_1,\ldots,\mathbb{H}_n$ from \eqref{eq_hyperplanes} satisfying
\[
\mathbb{H}_1\cap \cdots \cap \mathbb{H}_{n} =\{\bm{v}\}.
\]
Let $\bm{u}_1,\ldots,\bm{u}_n$ be normal vectors of these hyperplanes, then we have 
$\det [\bm{u}_1\,\cdots\,\bm{u}_n]\in \{\pm 1\}$ by the condition ({\rm H}). 
This shows that $\bm{v}$ is an integer vector and hence $\bm{v}\in \P_{\epsilon}(X)_{\mathbb{Z}}$. 
This completes the proof.
\end{proof}

\begin{example} \label{ex:sufficient sets}
We consider the case of $n=3$. 
In Proposition \ref{intermediate}, let $\mathcal{L}(X) := \{{\bm v}_{\tau} \mid \tau \in X\} \setminus \left\{\pm \bm{e}_1, \pm \bm{e}_2,\pm \bm{e}_3\right\}$. 
For $\epsilon =(+-+)$, one can easily check that, $X$ satisfies the condition {\rm (H)} if $\mathcal{L}(X)$ 
is contained in one of the following sets $\mathcal{E}_1$, $\mathcal{E}_2$ and $\mathcal{E}_3$. 
    \[
     \mathcal{E}_1 = \left\{\pm\begin{bsmallmatrix}
        1 \\ 1 \\ 1
    \end{bsmallmatrix}, \pm\begin{bsmallmatrix}
        1 \\ 0 \\ 1
    \end{bsmallmatrix}, \pm\begin{bsmallmatrix}
        1 \\ 1 \\ 0 
    \end{bsmallmatrix}\right\},\ 
    \mathcal{E}_2 = \left\{\pm\begin{bsmallmatrix}
        1 \\ 1 \\ 1
    \end{bsmallmatrix}, \pm\begin{bsmallmatrix}
        1 \\ 0 \\ 1
    \end{bsmallmatrix}, \pm\begin{bsmallmatrix}
        0 \\ 1 \\ 1 
    \end{bsmallmatrix}\right\},\ 
    \mathcal{E}_3 = \left\{\pm\begin{bsmallmatrix}
        1 \\ 1 \\ 1
    \end{bsmallmatrix}, \pm\begin{bsmallmatrix}
        1 \\ 0 \\ -1
    \end{bsmallmatrix}, \pm\begin{bsmallmatrix}
        0 \\ 1 \\ 1 
    \end{bsmallmatrix}\right\}. 
    \]
\end{example}

\section{The data $d(A,e)$
on the minimal number of generators}

In this section, we devote to study $g$-convex algebras of rank $3$. 
We will see that the structures of the $g$-fan and $g$-polytope are closely related to numerical invariant which we call the data on the minimal number of generators.

Throughout this section, let $A$ be a finite dimensional $k$-algebra of rank $3$ and $e=(e_1,e_2,e_3)$ a complete set of pairwise orthogonal primitive idempotents. 
We simply call the pair $(A,e)$ a finite dimensional $k$-algebra of rank $3$. 
Moreover, we assume that $A$ is $g$-convex (Section \ref{subsec:g-convex}). 

Now, we introduce the following numerical data for $(A,e)$. 

\begin{definition} \label{def:dAe}
For each pair $(i,j)$ with $1\le i \ne j \le 3$, we define integers $l_{ij}^{e}$, $r_{ij}^{e}$ and $h_{ij}^{e}$ as follows.   
\begin{equation}
  l_{ij}^e :=t_{e_iAe_i}(e_iAe_j), \quad r_{ij}^e := t(e_iAe_j)_{e_jAe_j} \quad \text{and} \quad h_{ij}^e := \begin{cases} 0 & \text{if $e_iAe_j \subseteq \rad^2 A$.} \\ 1 & \text{otherwise}.\end{cases}
\end{equation} 
In addition, we set 
\[
d(A,e) := (d_{ij}^e)_{1\le i\ne j \le 3} 
\quad \text{where}\quad  
d_{ij}^e:=(l_{ij}^e,r_{ij}^e,h_{ij}^e). 
\]
We call $d(A,e)$ \emph{the 
data on the minimal number of generators} of $(A,e)$. 
We omit the superscript $e$ when it is clear from the context. 
\end{definition}

We may visualize the data $d(A,e)$ as a matrix-like form as follows. 
\[
\begin{bmatrix}
    - & d_{12} & d_{13} \\ 
    d_{21} & - & d_{23} \\ 
    d_{31} & d_{32} & -
\end{bmatrix} 
= 
\begin{bmatrix}
        - & (l_{12},r_{12},h_{12}) & (l_{13},r_{13},h_{13}) \\
        (l_{21},r_{21},h_{21}) & - & (l_{23},r_{23},h_{23}) \\ 
        (l_{31},r_{31},h_{31}) & (l_{32},r_{32},h_{32})& - 
    \end{bmatrix}
\]
where $-$ stands for undefined entries. 

\begin{lemma}
In the above, for $1\leq i \neq j \leq 3$, we have 
\begin{equation*}
    d_{ij} \in \{(0,0,0), (1,1,0), (1,2,0),(2,1,0), (1,1,1), (1,2,1),(2,1,1)\}. 
\end{equation*}    
\end{lemma}

\begin{proof}
Since $A$ is $g$-convex, $(e_i+e_j)A(e_i+e_j)$ is also $g$-convex by Proposition \ref{prop:idemp_fan}.
Then the assertion follows from \cite[Theorem 5.1]{AHIKM2}.
\end{proof}

\begin{example}
Let $A$ be a finite dimensional $k$-algebra given by the following quiver and relations (This is the algebra appearing in No.\,60 of Table \ref{fig:61alg&poly}): 
\begin{equation}
\begin{xy}
    (0,-5)*+{1} = "1", (14,-5)*+{3} = "3", (7,5)*+{2} = "2", 
    { "1" \ar@<2.5mm>^{a} "2"}, 
    { "2" \ar@<-1mm>^{a^*} "1"}, 
    { "1" \ar@<-1mm>^{c^*} "3"}, 
    { "3" \ar@<2.5mm>^{c} "1"}, 
    { "2" \ar@(lu,ru)^(0.8){v} "2"}, 
\end{xy}
\quad \text{and} \quad
aa^* = a^*a = ca =  cc^* = c^*c = v^2 = 0.  
\end{equation}
Let $e_i$ be the idempotent corresponding to a vertex $i\in \{1,2,3
\}$. 
For each pair $(i,j)$ with $1\leq i \neq j \leq 3$, we display in the following table minimal sets of generators of $e_iAe_j$ as a left $e_iAe_i$-module and a right $e_jAe_j$-module respectively, together with the data $h_{ij}\in \{0,1\}$. 
Here, we notice that $h_{ij}=0$ if and only if $e_i A e_j\subseteq \rad^2 A$, if and only if there is no arrow from $i$ to $j$. 

\begin{center}    
\begin{tabular}{c|c|c|c|c|cccccc}
    & ${}_{e_iAe_i}(e_iAe_j)$ & $(e_iAe_j)_{e_jAe_j}$ & $h_{ij}$\\ \hline \hline 
    $(1,2)$ & $a,av$& $a$ & $1$\\
    $(1,3)$ & $c^*$ & $c^*, ava^*c^*$ & $1$\\
    $(2,1)$ & $a^*$ & $a^*, va^*$ & $1$\\
    $(2,3)$ & $a^*c^*$ & $a^*c^*, va^*c^*$ & $0$\\ 
    $(3,1)$ & $c$ & $c$ & $1$\\
    $(3,2)$ & $0$ & $0$ & $0$
\end{tabular}
\end{center}    

For example, we have equations 
\[
e_1 A e_2 = k a \oplus k av  = (e_1 A e_1) a+ (e_1 A e_1) av = a(e_2 A e_2).  
\]
Therefore, the datum $d(A,e)$ for $(A,e=(e_1,e_2,e_3))$ is given by 
\[
    \begin{bmatrix}
        - & (2,1,1) & (1,2,1) \\
        (1,2,1) & - & (1,2,0) \\ 
        (1,1,1) & (0,0,0)& - 
    \end{bmatrix}.
\]
\end{example}

\subsection{The subfan $\Sigma_{\epsilon}(A,e)$}
In this section, we study the subfan $\Sigma_{\epsilon}(A,e)$ in terms of the datum $d(A,e)$ for a $g$-convex algebra $(A,e)$ of rank $3$. 

We begin with the trivial case. 
Now, we denote by 
$\sigma_{+} := \mathbb{R}_{+++}^{3}$ and 
$\sigma_{-} := - \sigma_{+}$ 
the positive and negative cones of $\mathbb{R}^3$ respectively. 

\begin{proposition}\label{prop:positive_cone}
    The subfans $\Sigma_{+++}(A,e)$ and $\Sigma_{---}(A,e)$ have the unique maximal cones $\sigma_{+}$ and $\sigma_{-}$ respectively.  
\end{proposition}

\begin{proof}
    It is immediate from the fact that $\sigma_+=C(A)$ and $\sigma_- = C(A[1])$. 
\end{proof}

Thus, we consider signatures in $\{\pm\}^3\setminus\{(+++),(---)\}$. 
By the group action of $G=\mathfrak{S}_3\times \{\pm1\}$ (see \eqref{g iso}), 
we mainly study the signature $(+-+)\in \{\pm\}^3$ in this section.

Recall that the full poset $(\SigmaeAe,\leq)$ is an interval with the maximum $\emax$ and minimum $\emin$ given in \eqref{emin_emax}. 
We first give a description of them in terms of $d(A,e)$. 

\begin{proposition}\label{min_max}
We have 
\begin{equation}
    \emin = \left[\begin{smallmatrix}
    1 \\ -r_{12} \\ 0 
\end{smallmatrix}\middle| \begin{smallmatrix}
    0 \\ -1 \\ 0
\end{smallmatrix}\middle| \begin{smallmatrix}
    0 \\ -r_{32} \\ 1
\end{smallmatrix}\right] 
\quad \text{and} \quad 
\emax = \left[\begin{smallmatrix}
    1 \\ 0 \\ 0 
\end{smallmatrix}\middle| \begin{smallmatrix}
    l_{12}' \\ -1 \\ l_{32}'
\end{smallmatrix}\middle| \begin{smallmatrix}
    0 \\ 0 \\ 1
\end{smallmatrix}\right],
\end{equation}
where integers $l_{12}'$ and $l_{32}'$ are defined as follows.  
    \begin{equation}\label{lij_prime}
        (l_{12}', l_{32}') := \begin{cases}
        (0,0) & \text{if $(h_{12},h_{32}) = (0,0)$,}\\
        (l_{12},0) & \text{if $(h_{12},h_{32}) = (1,0)$,}\\
        (0,l_{32}) & \text{if $(h_{12},h_{32}) = (0,1)$,}\\
        (1,1) & \text{if $(h_{12}, h_{32}) = (1,1)$.}
        \end{cases}
    \end{equation}
Moreover, their normal vectors are given by  
\begin{equation}
    \bm{v}_{\emin} = \begin{bsmallmatrix}
        1-r_{12} \\ -1 \\ 1-r_{32}
    \end{bsmallmatrix} \quad \text{and} \quad 
    \bm{v}_{\emax} = \begin{bsmallmatrix}
        1 \\ l_{12}' + l_{32}' -1 \\1
    \end{bsmallmatrix} 
    .
\end{equation}
\end{proposition}
\begin{proof}
From \eqref{emin_emax}, the minimum and maximum elements of $\SigmaeAe$ are given by 
$$\emin = C(\mu^+_{e_1A[1]\oplus e_3A[1]}(A[1])) \quad \text{and}\quad 
\emax= C(\mu^-_{e_2A}(A))$$
 respectively.

We need to compute these left and right mutations to get the claim. 
Firstly, by the definitions of $l_{ij}$ and $l_{ij}'$, 
a minimal left $(\add (e_1A\oplus e_3A))$-approximation of $e_2A$ is of the form 
\[
h\colon e_2A \longrightarrow e_1A^{\oplus l'_{12}}\oplus e_3A^{\oplus l'_{32}}. 
\]
Then, we have a left mutation 
$\mu^-_{e_2A}(A) = e_1A\oplus V\oplus e_3A$, 
where 
\[ 
V = [e_2A \xrightarrow{h} e_1A^{\oplus l'_{12}}\oplus e_3A^{\oplus l'_{32}}]. 
\]

On the other hand, by the definition of $r_{ij}$, 
a minimal right $(\add e_2A[1])$-approximation of $e_1A[1]$ (resp., $e_3A[1]$) has the following form.
\[
f_{12}\colon e_2A[1]^{\oplus r_{12}} \longrightarrow e_1A[1] \quad 
(\text{resp.,} \ f_{32}\colon e_2A[1]^{\oplus r_{32}} \longrightarrow e_3A[1]). 
\]
It implies that the right mutation is given by $\mu^+_{e_1A[1]\oplus e_3A[1]}A[1] = U_1\oplus e_2A[1]\oplus U_3$, where 
\[
    U_1 = [e_2A^{\oplus r_{12}} \xrightarrow{f_{12}} e_1A] 
    \quad \text{and} \quad 
    U_3 = [e_2A^{\oplus r_{32}} \xrightarrow{f_{32}} e_3A]
\]
Thus, we get the claim. 
\end{proof}

On the other hand, we have the following result. 

\begin{proposition}
\label{coord_plane}
Keep the above notations, the following hold. 
\begin{enumerate}[\upshape (1)]
\item For each case of $(l_{12},r_{12})\in \{(0,0), (1,1),(1,2),(2,1)\}$, 
the set of rays of $\SigmaeAe$ lying in the region $\cone\{\bm{e}_1, - \bm{e}_2\}$ 
is given by the following table (See the top row of Table \ref{tab:plane_vectors}).
\begin{center}
\scalebox{0.9}{\begin{tabular}{cccccc}
 \hline
$(0,0)$ & $(1,1)$ & $(1,2)$ & $(2,1)$ \\ 
\hline 
\vspace{5pt}
\raisebox{-5pt}{$\left\{\left[\begin{smallmatrix}
        1 \\ 0 \\ 0 
    \end{smallmatrix}\right],  \left[\begin{smallmatrix}
        0 \\ -1 \\ 0 
    \end{smallmatrix}\right] \right\}$}  
& \raisebox{-5pt}{$ \left\{\left[\begin{smallmatrix}
    1 \\ 0 \\ 0 
\end{smallmatrix}\right],  \left[\begin{smallmatrix}
    1 \\ -1 \\ 0 
\end{smallmatrix}\right],  \left[\begin{smallmatrix}
    0 \\ -1 \\ 0 
\end{smallmatrix}\right] \right\} $}
 & \raisebox{-5pt}{$\left\{ \left[\begin{smallmatrix}
    1 \\ 0 \\ 0 
\end{smallmatrix}\right], \left[\begin{smallmatrix}
    1 \\ -1 \\ 0 
\end{smallmatrix}\right], \left[\begin{smallmatrix}
    1 \\ -2 \\ 0 
\end{smallmatrix}\right], \left[\begin{smallmatrix}
    0 \\ -1 \\ 0 
\end{smallmatrix}\right] \right\} $} 
&\raisebox{-5pt}{$\left\{\left[\begin{smallmatrix}
    1 \\ 0 \\ 0 
\end{smallmatrix}\right], \left[\begin{smallmatrix}
    2 \\ -1 \\ 0 
\end{smallmatrix}\right], \left[\begin{smallmatrix}
    1 \\ -1 \\ 0 
\end{smallmatrix}\right],
 \left[\begin{smallmatrix}
    0 \\ -1 \\ 0 
\end{smallmatrix}\right] \right\} $}   
    \\
 \hline
 \end{tabular}}
\end{center}
\item For each case of $(l_{32},r_{32})\in \{(0,0), (1,1),(1,2),(2,1)\}$, 
the set of rays of $\SigmaeAe$ lying in the region $\cone\{-\bm{e}_2,\bm{e}_3\}$
is given by the following table (See the bottom row of Table \ref{tab:plane_vectors}).
\begin{center}
\scalebox{0.9}{\begin{tabular}{cccc}
 \hline
 $(0,0)$ & $(1,1)$ & $(1,2)$ & $(2,1)$ \\ 
\hline \vspace{5pt}
\raisebox{-5pt}{$\left\{\left[\begin{smallmatrix}
        0 \\ 0 \\ 1 
    \end{smallmatrix}\right],  \left[\begin{smallmatrix}
        0 \\ -1 \\ 0 
    \end{smallmatrix}\right] \right\}$}  
& \raisebox{-5pt}{$\left\{\left[\begin{smallmatrix}
    0 \\ 0 \\ 1 
\end{smallmatrix}\right],  \left[\begin{smallmatrix}
    0 \\ -1 \\ 1 
\end{smallmatrix}\right],  \left[\begin{smallmatrix}
    0 \\ -1 \\ 0 
\end{smallmatrix}\right] \right\} $}
 & \raisebox{-5pt}{$\left\{ \left[\begin{smallmatrix}
    0 \\ 0 \\ 1 
\end{smallmatrix}\right], \left[\begin{smallmatrix}
    0 \\ -1 \\ 1 
\end{smallmatrix}\right], \left[\begin{smallmatrix}
    0 \\ -2 \\ 1 
\end{smallmatrix}\right], \left[\begin{smallmatrix}
    0 \\ -1 \\ 0 
\end{smallmatrix}\right] \right\} $} 
&\raisebox{-5pt}{$\left\{\left[\begin{smallmatrix}
    0 \\ 0 \\ 1 
\end{smallmatrix}\right], \left[\begin{smallmatrix}
    0 \\ -1 \\ 2 
\end{smallmatrix}\right], \left[\begin{smallmatrix}
    0 \\ -1 \\ 1 
\end{smallmatrix}\right],
 \left[\begin{smallmatrix}
    0 \\ -1 \\ 0 
\end{smallmatrix}\right] \right\} $}   
    \\
 \hline
 \end{tabular}}
\end{center}
\end{enumerate}
\end{proposition}

\input{Tables/table_planevect}

\begin{proof}
Let $f := e_1 + e_2$. 
Then, the idempotent subalgebra $B := fAf$ of $A$ is an algebra of rank $2$. 
In addition, it is $g$-convex by Proposition \ref{prop:idemp_fan},
and its $g$-fan is identified with the subfan of $\SigmaAe$ consisting of all $g$-vector cones $C(T)$ which are contained in the plane $\{\bm{x}\in \mathbb{R}^3\mid x_3 =0\}$. 
Under this identification, the region $\cone\{\bm{e}_1,-\bm{e}_2\}$ corresponds to the $(+,-)$-quadrant of $\mathbb{R}^2$. 
Since integers $l_{12}$ and $r_{12}$ are given by $l_{12}:=t_{e_1Ae_1}(e_1Ae_2) = t_{e_1Be_1}(e_1Be_2)$ and $r_{12}:=t(e_1Ae_2)_{e_2Ae_2} = t(e_1Be_2)_{e_2Be_2}$, 
 the result of \cite[Theorem\;5.1]{AHIKM2} for $B$ gives that 
the set of rays of the $g$-fan of $B$ lying in this quadrant is 
determined by $(l_{12},r_{12}) \in \{(0,0),(1,1),(1,2),(2,1)\}$ and given by 
\[
\left\{\left[\begin{smallmatrix}
    1 \\ 0 
\end{smallmatrix}\right],  \left[\begin{smallmatrix}
    0 \\ -1 
\end{smallmatrix}\right] \right\}, 
\left\{\left[\begin{smallmatrix}
    1 \\ 0  
\end{smallmatrix}\right],  \left[\begin{smallmatrix}
    1 \\ -1 
\end{smallmatrix}\right],  \left[\begin{smallmatrix}
    0 \\ -1 
\end{smallmatrix}\right] \right\},
\left\{ \left[\begin{smallmatrix}
    1 \\ 0  
\end{smallmatrix}\right], \left[\begin{smallmatrix}
    1 \\ -1 
\end{smallmatrix}\right], \left[\begin{smallmatrix}
    1 \\ -2 
\end{smallmatrix}\right], \left[\begin{smallmatrix}
    0 \\ -1 
\end{smallmatrix}\right] \right\},
\left\{ \left[\begin{smallmatrix}
    1 \\ 0 
\end{smallmatrix}\right], \left[\begin{smallmatrix}
    2 \\ -1  
\end{smallmatrix}\right], \left[\begin{smallmatrix}
    1 \\ -1  
\end{smallmatrix}\right], \left[\begin{smallmatrix}
    0 \\ -1 
\end{smallmatrix}\right] \right\}
\]
respectively. 
Thus, we get the assertion (1). 
Similarly, we get (2). 
\end{proof}

\subsection{Reduction and maximal paths}
In this section, we use the reduction at rays (Section \ref{subsect:reduction}) to analyze the local structure of $\SigmaAe$. 

Let $W$ be an indecomposable 2-term presilting complex and $\bm{w}:=g(W)$ the $g$-vector of $W$. 
Considering the reduction at $\bm{w}$, we obtain a bijection 
$\Sigma_{\bm{w}}(A,e)\simeq \Sigma(B_W,e')$ as in \eqref{eq:reduction_fan}. 
In this case, $B_W$ is an algebra of rank $2$, which is $g$-convex by Proposition \ref{prop:red_conv}.
Due to a classification of all convex $g$-polygons (Table \ref{tab:convex_gfan_rank2}), we find that its $g$-fan $\Sigma(B_W,e')$ is one of fans in that table. 
In particular, the poset $(\Sigma(B_W,e'),\leq)$ has exactly two maximal paths from the maximum to the minimum in its Hasse quiver. 
Hence, so does the full subposet $(\Sigma_{\bm{w}}(A,e),\leq)$ via the poset isomorphism \eqref{eq:reduction_poset}. 

In Subsection \ref{subsect:u_reduction}, we provide a computation for the maximal paths of $(\Sigma_{\bm{w}}(A,e),\leq)$ by using exchange sequences inherited from silting mutation. 
Then, we apply in the next Subsection \ref{subsect:e2_reduction} the previous observation to a particular integer vector as $\bm{w} = -\bm{e}_2 \in\mathbb{R}^3$. 
The results in both sections will be frequently used in Section \ref{sec:proof_Thm2}.

\subsubsection{Maximal paths of $\Sigma_{\bm{w}}(A,e)$} \label{subsect:u_reduction}

To understand the specific shape of the $g$-fan, we will describe the detailed information about the Hasse quiver  $(\Sigma_{\bm{w}}(A,e),\leq)$. We introduce the following notation. 

\begin{notation}\rm 
    Let $p \colon \sigma \to \tau$ be an arrow in the Hasse quiver of $(\Sigma_{\bm w}(A,e),\leq)$. By Theorem \ref{theorem:g-fan is ordered}, two maximal cones $\sigma$ and $\tau$ are adjacent to each other. 
    Now, we suppose that 
    $\sigma$ and $\tau$ are given by matrix form as 
    $\sigma = [\bm{u}|\bm{v}|\bm{w}]$ and 
    $\tau = [\bm{u}'|\bm{v}|\bm{w}]$. 
    By Theorem \ref{theorem:g-convex=pairwise_g-convex}, 
    there exist non-negative integers $a,b \geq 0$ satisfying 
    $\bm{u} + \bm{u}' = a\bm{v} + b \bm{w}$ with 
    $a+b\leq 2$. Thus, 
    \[(a,b)\in \{(0,0),(0,1),(0,2),(1,0),(1,1),(2,0)\}.\] 
    In this case, we assign 
    \begin{equation}
        p \colon \
        \sigma = [\bm{u}|\bm{v}|\bm{w}] \xrightarrow{(\bullet,a,b)} 
        \tau = [\bm{u}'|\bm{v}|\bm{w}] 
    \end{equation}
    to this arrow, where the symbol $\bullet$ shows the position of ray which is exchanged. 
    When we emphasize that $p$ is an arrow of $(\Sigma_{\bm{w}}(A,e),\leq)$, we will also describe it as 
    \begin{equation} \label{eq:path_at_w}
        \underline{p}_{\bm{w}} \colon \ 
        \underline{\sigma}_{\bm{w}} = [\bm{u}|\bm{v}] \underset{b}{\xrightarrow{(\bullet,a)}} 
        \underline{\tau}_{\bm{w}} = [\bm{u}'|\bm{v}]. 
    \end{equation}
    If $\bm{w}$ is clear, then we simply write $\underline{\sigma}$ and $\underline{\tau}$ instead of $\underline{\sigma}_{\bm{w}}$ and $\underline{\tau}_{\bm{w}}$ respectively. 
    By permuting a role of $\bm{v}$ and $\bm{w}$, 
    we can also describe this arrow by 
    \begin{equation}
        \underline{p}_{\bm{v}} \colon \ 
        \underline{\sigma}_{\bm{v}} = [\bm{u}|\bm{w}] \underset{a}{\xrightarrow{(\bullet,b)}} 
        \underline{\tau}_{\bm{v}} = [\bm{u}'|\bm{w}]. 
    \end{equation}
    
    We will use these notation too for those given by permuting the ordering of rays. 
    For example, the above arrow can be displayed by  
    \begin{equation}
        [\bm{w}|\bm{u}|\bm{v}] \xrightarrow{(b,\bullet,a)} 
        [\bm{w}|\bm{u}'|\bm{v}] 
    \quad \text{or} \quad 
        [\bm{v}|\bm{u}] \underset{b}{\xrightarrow{(a,\bullet)}} 
        [\bm{v}|\bm{u}'], 
    \end{equation}
    and so on. 
\end{notation}

\begin{remark}
    Our notation \eqref{eq:path_at_w} could be justified in the following sense: 
    By Proposition \ref{prop:red_conv} (1), 
    the arrow in \eqref{eq:path_at_w} corresponds to the arrow 
    \begin{equation}
        F_{\bm{w}}(p) \colon F_{\bm{w}}(\sigma) = [F_{\bm{w}}(\bm{u})|F_{\bm{w}}(\bm{v})] \xrightarrow{(\bullet,a)} 
        F_{\bm{w}}(\tau) = [F_{\bm{w}}(\bm{u}')|F_{\bm{w}}(\bm{v})]
    \end{equation}
    in the poset $(\Sigma(B_W,e'),\leq)$. 
    Thus, the number under the arrow shows the remaining information. 
\end{remark}

As we mentioned before, the full subposet $(\Sigma_{\bm{w}}(A,e),\leq)$ has precisely two maximal paths, say $p$ and $q$, that connect the maximum and the minimum in its Hasse quiver. 
We refer to subpaths $s$ of $p,q$ as \emph{paths around $\bm{w}$} and write $\underline{s}_{\bm{w}}$.

\begin{proposition}\label{rank2} 
    Let $\bm{w}$ be as above. 
    Suppose that $\sigma=[\bm{u}|\bm{v}|\bm{w}]$ and $\sigma':=[\bm{u}'|\bm{v}'|\bm{w}]$ are the maximum and minimum of the poset $(\Sigma_{\bm{w}}(A,e),\leq)$ respectively. 
    
    Consider the maximal path $\underline{p}_{\bm{w}} = (\sigma \to \tau_1\to \cdots \to \tau_{\ell} = \sigma')$ starting with an exchange at $\bm{u}$. 
    Then, it is one of the following forms:
    \begin{enumerate}[\rm (i)]
        \item 
        \begin{equation*}
            \underline{p}_{\bm{w}} = \left(\underline{\sigma} = 
            [\bm{u}|\bm{v}] \underset{a_1}{\xrightarrow{(\bullet,0)}} 
            \underline{\tau_1} \underset{a_2}{\xrightarrow{(0,\bullet)}} 
            \underline{\tau_2} = \underline{\sigma'} \right),
        \end{equation*}
        \item  
        \begin{equation*}
            \underline{p}_{\bm{w}} = 
            \left(\underline{\sigma} = 
            [\bm{u}|\bm{v}] \underset{b_1}{\xrightarrow{(\bullet,1)}} 
            \underline{\tau_1} \underset{b_2}{\xrightarrow{(1,\bullet)}} 
            \underline{\tau_2} \underset{b_3}{\xrightarrow{(\bullet,1)}}
            \underline{\tau_3}
            = \underline{\sigma'} \right),
        \end{equation*}
        \item 
        \begin{equation*}
            \underline{p}_{\bm{w}} = \left(\underline{\sigma} = 
            [\bm{u}|\bm{v}] \underset{0}{\xrightarrow{(\bullet,2)}} 
            \underline{\tau_1} \underset{c_1}{\xrightarrow{(1,\bullet)}} 
            \underline{\tau_2} \underset{0}{\xrightarrow{(\bullet,2)}}
            \underline{\tau_3} \underset{c_2}{\xrightarrow{(1,\bullet)}}
            \underline{\tau_4} = \underline{\sigma'} \right),
        \end{equation*}
        \item 
        \begin{equation*}
            \underline{p}_{\bm{w}} = \left(\underline{\sigma} = 
            [\bm{u}|\bm{v}] \underset{d_1}{\xrightarrow{(\bullet,1)}} 
            \underline{\tau_1} \underset{0}{\xrightarrow{(2,\bullet)}} 
            \underline{\tau_2} \underset{d_2}{\xrightarrow{(\bullet,1)}}
            \underline{\tau_3} \underset{0}{\xrightarrow{(2,\bullet)}}
            \underline{\tau_4} = \sigma' \right),
        \end{equation*}
    \end{enumerate}
    where $a_i\in \{0,1,2\}$ and $b_i,c_i,d_i\in \{0,1\}$. 
    Notice that each $\tau_j$ can be computed from $\sigma$ inductively. 

    Similarly, let $\underline{p}'_{\bm{w}} = (\sigma \to \tau_1' \to \cdots \to \tau_{\ell'}' = \sigma')$ 
    be the maximal path starting with an exchange of $\bm{v}$. 
    Then, it is one of the following forms: 
    \begin{enumerate}[\rm (i')]
        \item 
        \begin{equation*}
            \underline{p}'_{\bm{w}} = \left(\underline{\sigma} = 
            [\bm{u}|\bm{v}] \underset{a_1'}{\xrightarrow{(0,\bullet)}} 
            \underline{\tau_1'} \underset{a_2'}{\xrightarrow{(\bullet,0)}} 
            \underline{\tau_2'} = \underline{\sigma'} \right),
        \end{equation*}
        \item 
        \begin{equation*}
            \underline{p}'_{\bm{w}} = \left(\underline{\sigma} = 
            [\bm{u}|\bm{v}] \underset{b_1'}{\xrightarrow{(1,\bullet)}} 
            \underline{\tau_1'} \underset{b_2'}{\xrightarrow{(\bullet,1)}} 
            \underline{\tau_2'} \underset{b_3'}{\xrightarrow{(1,\bullet)}}
            \underline{\tau_3'}
            = \underline{\sigma'} \right),
        \end{equation*}
        \item 
        \begin{equation*}
            \underline{p}'_{\bm{w}} = \left(\underline{\sigma} = 
            [\bm{u}|\bm{v}] \underset{0}{\xrightarrow{(2,\bullet)}} 
            \underline{\tau_1'} \underset{c_1'}{\xrightarrow{(\bullet,1)}} 
            \underline{\tau_2'} \underset{0}{\xrightarrow{(2,\bullet)}}
            \underline{\tau_3'} \underset{c_2'}{\xrightarrow{(\bullet,1)}}
            \underline{\tau_4'} = \underline{\sigma'} \right),
        \end{equation*}
        \item 
        \begin{equation*}
            \underline{p}'_{\bm{w}} = \left(\underline{\sigma} = 
            [\bm{u}|\bm{v}] \underset{d_1'}{\xrightarrow{(1,\bullet)}} 
            \underline{\tau_1'} \underset{0}{\xrightarrow{(\bullet,2)}} 
            \underline{\tau_2'} \underset{d_2'}{\xrightarrow{(1,\bullet)}}
            \underline{\tau_3'} \underset{0}{\xrightarrow{(\bullet,2)}}
            \underline{\tau_4'} = \underline{\sigma'} \right),
        \end{equation*}
    \end{enumerate}
    where $a_i'\in \{0,1,2\}$ and $b_i',c_i',d_i'\in \{0,1\}$. Notice that each $\tau_j'$ can be computed from $\sigma$ inductively. 
\end{proposition}

\begin{proof}
    This follows from \cite[Theorem\;4.11]{AHIKM1} and \cite[Theorem\;5.1]{AHIKM2}. 
\end{proof}

\subsubsection{Maximal paths in $\Sigma_{-\mathbf{e}_2}(A,e)$} 
\label{subsect:e2_reduction}

In this section, we focus on the reduction at $e_2A[1]\in \twoips A$ with $g$-vector $g(e_2A[1])= - \bm{e}_2$.
Since we have an interval 
\begin{equation}
    (\twosilt_{e_2A[1]} A,\leq) = [A[1], \mu_{e_1A[1] \oplus e_3A[1]}^+(A[1])], 
\end{equation}
the full subposet $(\Sigma_{-\bm{e}_2}(A,e),\leq)$ is also an interval whose maximum and minimum are given by 
\[
\emin = \left[\begin{smallmatrix}
    1 \\ -r_{12} \\ 0 
\end{smallmatrix}\middle| \begin{smallmatrix}
    0 \\ -1 \\ 0
\end{smallmatrix}\middle| \begin{smallmatrix}
    0 \\ -r_{32} \\ 1
\end{smallmatrix}\right] 
\quad \text{and} \quad
\sigma_- = [-\bm{e}_1|-\bm{e}_2|-\bm{e}_3]     
\]
respectively (cf.\,the proof of Proposition \ref{min_max}).
Thus, it has two maximal paths, say $\alpha$ and $\beta$, starting at exchanges of $\left[\begin{smallmatrix}
            0 \\  -r_{32} \\ 1
\end{smallmatrix}\right]$
and $\left[\begin{smallmatrix}
            1 \\  -r_{12} \\ 0
        \end{smallmatrix}\right]$ respectively:
\begin{equation}\label{path pq}
\begin{split}
        \underline{\alpha}_{(-\bm{e}_2)} &=  
        \left( 
        \underline{\emin} = \left[\begin{smallmatrix}
            1 \\  -r_{12} \\ 0
        \end{smallmatrix}\middle| 
        \begin{smallmatrix}
            0 \\  -r_{32} \\ 1
        \end{smallmatrix}\right] 
        \underset{a_2}{\xrightarrow{(a_1,\bullet)}} 
        \underline{\tau_1} \to \cdots \to 
        \underline{\tau_{\ell}} = 
        \underline{\sigma_-} 
        \right), \\
        \underline{\beta}_{(-\bm{e}_2)} &= 
        \left(\underline{\emin} = 
        \left[\begin{smallmatrix}
            1 \\  -r_{12} \\ 0
        \end{smallmatrix}\middle| 
        \begin{smallmatrix}
            0 \\  -r_{32} \\ 1
        \end{smallmatrix}\right] 
        \underset{a_2'}{\xrightarrow{(\bullet,a_1')}} 
        \underline{\tau_1'} \to \cdots \to \underline{\tau'_{\ell'}} = \underline{\sigma_-} 
        \right).
\end{split}
\end{equation}
Then, we have a detailed version of the assertions made in Proposition\;\ref{rank2}.
\begin{proposition}\label{path p1}
The path $\alpha$ in \eqref{path pq} is determined by $(r_{12},r_{32})$ and given by one of 
{\upshape {\bf (i)}}, {\upshape {\bf (ii)}}, {\upshape {\bf (iii)}}, {\upshape {\bf (iv)}}, and {\upshape {\bf (v)}} as follows. 
\begin{enumerate}
    \item[\bf (i)] 
    \begin{equation*}
        \underline{\alpha}_{(-\bm{e}_2)} = \left( \underline{\emin} = \left[\begin{smallmatrix}
            1 \\  -r_{12} \\ 0
        \end{smallmatrix}\middle| 
        \begin{smallmatrix}
            0 \\  -r_{32} \\ 1
        \end{smallmatrix}\right] 
        \underset{r_{32}}{\xrightarrow{(0,\bullet)}} 
        \underline{\tau_1} 
        \underset{r_{12}}{\xrightarrow{(\bullet,0)}} 
        \underline{\tau_2} = \underline{\sigma_-} 
        \right).
    \end{equation*}
    \item[\bf (ii)] 
    \begin{equation*}
        \underline{\alpha}_{(-\bm{e}_2)} = \left( \underline{\emin} = \left[\begin{smallmatrix}
            1 \\  -r_{12} \\ 0
        \end{smallmatrix}\middle| 
        \begin{smallmatrix}
            0 \\  -r_{32} \\ 1
        \end{smallmatrix}\right] 
        \underset{r_{32}}{\xrightarrow{(1,\bullet)}} 
        \underline{\tau_1} 
        \underset{0}{\xrightarrow{(\bullet,1)}} 
        \underline{\tau_2} 
        \underset{r_{12}}{\xrightarrow{(1,\bullet)}}
        \underline{\tau_3} = \underline{\sigma_-}
        \right).
    \end{equation*}
    \item[\bf (iii)] 
    \begin{equation*}
        \underline{\alpha}_{(-\bm{e}_2)} = \left(\underline{\emin} = \left[\begin{smallmatrix}
            1 \\  -r_{12} \\ 0
        \end{smallmatrix}\middle| 
        \begin{smallmatrix}
            0 \\  -r_{32} \\ 1
        \end{smallmatrix}\right] 
        \underset{r_{32}-1}{\xrightarrow{(1,\bullet)}} 
        \underline{\tau_1} 
        \underset{1}{\xrightarrow{(\bullet,1)}} 
        \underline{\tau_2} 
        \underset{r_{12}-1}{\xrightarrow{(1,\bullet)}}
        \underline{\tau_3} = \underline{\sigma_-}
        \right).
    \end{equation*}
    \item[\bf (iv)] 
    \begin{equation*}
        \underline{\alpha}_{(-\bm{e}_2)} = \left(\underline{\emin} = \left[\begin{smallmatrix}
            1 \\  -r_{12} \\ 0
        \end{smallmatrix}\middle| 
        \begin{smallmatrix}
            0 \\  -r_{32} \\ 1
        \end{smallmatrix}\right] 
        \underset{0}{\xrightarrow{(2,\bullet)}} 
        \underline{\tau_1} 
        \underset{\frac{r_{32}}{2}}{\xrightarrow{(\bullet,1)}} 
        \underline{\tau_2} 
        \underset{0}{\xrightarrow{(2,\bullet)}}
        \underline{\tau_3} 
        \underset{r_{12}-\frac{r_{32}}{2}}{\xrightarrow{(\bullet,1)}} 
        \underline{\tau_4} = \underline{\sigma_-}
        \right).
    \end{equation*}
    \item[\bf (v)] 
    \begin{equation*}
        \underline{\alpha}_{(-\bm{e}_2)} = \left(\underline{\emin} = \left[\begin{smallmatrix}
            1 \\  -r_{12} \\ 0
        \end{smallmatrix}\middle| 
        \begin{smallmatrix}
            0 \\  -r_{32} \\ 1
        \end{smallmatrix}\right] 
        \underset{r_{32}-\frac{r_{12}}{2}}{\xrightarrow{(1,\bullet)}} 
        \underline{\tau_1} 
        \underset{0}{\xrightarrow{(\bullet,2)}} 
        \underline{\tau_2} 
        \underset{\frac{r_{12}}{2}}{\xrightarrow{(1,\bullet)}}
        \underline{\tau_3} 
        \underset{0}{\xrightarrow{(\bullet,2)}} 
        \underline{\tau_4} = \underline{\sigma_-}
        \right)
    \end{equation*}
\end{enumerate}
    Here, each $\tau_j$ can be computed from $\emin$ inductively. Moreover, the following statements hold. 
    \begin{enumerate}[\rm (1)]
        \item The case {\upshape {\bf (i)}} appears if and only if $h_{13}=0$. If this is the case, $h_{12}=0$ if and only if $r_{12}=0$. 
        \item If {\upshape {\bf (ii)}} appears, then $r_{12},r_{32}\in \{0,1\}$. 
        Moreover, $h_{12}=0$ if and only if $r_{12}=0$.
        \item If {\upshape {\bf (iii)}} appears, then $r_{12},r_{32}\in \{1,2\}$. Moreover, $h_{12}=0$ if and only if $r_{12}=1$. 
        \item If {\upshape {\bf (iv)}} appears, then $(r_{12},r_{32})\in \{(0,0),(1,0),(1,2),(2,2)\}$. 
        Moreover, $h_{12}=0$ if and only if $(r_{12},r_{32})\in \{(0,0),(1,2)\}$. 
        \item If {\upshape {\bf (v)}} appears, then $(r_{12},r_{32})\in \{(0,0),(0,1),(2,1),(2,2)\}$. In this case, we have $h_{12}=0$. 
    \end{enumerate}
\end{proposition}

Dually, we have the following statement. 

\begin{proposition}\label{path p3}
The path $\beta$ in \eqref{path pq} is determined by $(r_{12},r_{32})$ and given by one of
{\upshape {\bf (i')}}, {\upshape {\bf (ii')}}, {\upshape {\bf (iii')}}, {\upshape {\bf (iv')}}, and {\upshape {\bf (v')}} as follows.
\begin{enumerate}
    \item[\bf (i')]  
    \begin{equation*}
        \underline{\beta}_{(-\bm{e}_2)} = \left(\underline{\emin} = \left[\begin{smallmatrix}
            1 \\  -r_{12} \\ 0
        \end{smallmatrix}\middle| 
        \begin{smallmatrix}
            0 \\  -r_{32} \\ 1
        \end{smallmatrix}\right] 
        \underset{r_{12}}{\xrightarrow{(\bullet,0)}} 
        \underline{\tau_1'} 
        \underset{r_{32}}{\xrightarrow{(0,\bullet)}} 
        \underline{\tau_2'} = \underline{\sigma_-}
        \right).
    \end{equation*}
    \item[\bf (ii')] 
    \begin{equation*}
        \underline{\beta}_{(-\bm{e}_2)} = \left(\underline{\emin} = \left[\begin{smallmatrix}
            1 \\  -r_{12} \\ 0
        \end{smallmatrix}\middle| 
        \begin{smallmatrix}
            0 \\  -r_{32} \\ 1
        \end{smallmatrix}\right] 
        \underset{r_{12}}{\xrightarrow{(\bullet,1)}} 
        \underline{\tau_1'} 
        \underset{0}{\xrightarrow{(1,\bullet)}} 
        \underline{\tau_2'} 
        \underset{r_{32}}{\xrightarrow{(\bullet,1)}}
        \underline{\tau_3'} = \underline{\sigma_-}
        \right).
    \end{equation*}
    \item[\bf (iii')] 
    \begin{equation*}
        \underline{\beta}_{(-\bm{e}_2)} = \left(\underline{\emin} = \left[\begin{smallmatrix}
            1 \\  -r_{12} \\ 0
        \end{smallmatrix}\middle| 
        \begin{smallmatrix}
            0 \\  -r_{32} \\ 1
        \end{smallmatrix}\right] 
        \underset{r_{12}-1}{\xrightarrow{(\bullet,1)}} 
        \underline{\tau_1'} 
        \underset{1}{\xrightarrow{(1,\bullet)}} 
        \underline{\tau_2'} 
        \underset{r_{32}-1}{\xrightarrow{(\bullet,1)}}
        \underline{\tau_3'} = \underline{\sigma_-}
        \right).
    \end{equation*}
    \item[\bf (iv')] 
    \begin{equation*}
        \underline{\beta}_{(-\bm{e}_2)} = \left(\underline{\emin} = \left[\begin{smallmatrix}
            1 \\  -r_{12} \\ 0
        \end{smallmatrix}\middle| 
        \begin{smallmatrix}
            0 \\  -r_{32} \\ 1
        \end{smallmatrix}\right] 
        \underset{0}{\xrightarrow{(\bullet,2)}} 
        \underline{\tau_1'}
        \underset{\frac{r_{12}}{2}}{\xrightarrow{(1,\bullet)}} 
        \underline{\tau_2'} 
        \underset{0}{\xrightarrow{(2,\bullet)}}
        \underline{\tau_3'} 
        \underset{r_{32}-\frac{r_{12}}{2}}{\xrightarrow{(1,\bullet)}} 
        \underline{\tau_4'} = \underline{\sigma_-}
        \right).
    \end{equation*}
    \item[\bf (v')] 
    \begin{equation*}
        \underline{\beta}_{(-\bm{e}_2)} = \left(\underline{\emin} = \left[\begin{smallmatrix}
            1 \\  -r_{12} \\ 0
        \end{smallmatrix}\middle| 
        \begin{smallmatrix}
            0 \\  -r_{32} \\ 1
        \end{smallmatrix}\right] 
        \underset{r_{12}-\frac{r_{32}}{2}}{\xrightarrow{(\bullet,1)}} 
        \underline{\tau_1'} 
        \underset{0}{\xrightarrow{(\bullet,2)}} 
        \underline{\tau_2'} 
        \underset{\frac{r_{32}}{2}}{\xrightarrow{(1,\bullet)}}
        \underline{\tau_3'} 
        \underset{0}{\xrightarrow{(2,\bullet)}} 
        \underline{\tau_4'} = \underline{\sigma_-}
        \right).
    \end{equation*}
\end{enumerate}
    Here, each $\tau_j'$ can be computed from $\emin$ inductively. Moreover, the following statements hold. 
    \begin{enumerate}[\rm (1')]
        \item The case {\upshape {\bf (i')}} appears if and only if $h_{31}=0$. If this is the case, $h_{32}=0$ if and only if $r_{32}=0$. 
        \item If {\upshape {\bf (ii')}} appears, then $r_{12},r_{32}\in \{0,1\}$. 
        Moreover, $h_{32}=0$ if and only if $r_{32}=0$.
        \item If {\upshape {\bf (iii')}} appears, then $r_{12},r_{32}\in \{1,2\}$. Moreover, $h_{32}=0$ if and only if $r_{32}=1$. 
        \item If {\upshape {\bf (iv')}} appears, then $(r_{12},r_{32})\in \{(0,0),(0,1),(2,1),(2,2)\}$. 
        Moreover, $h_{32}=0$ if and only if $(r_{12},r_{32})\in \{(0,0),(2,1)\}$. 
        \item If {\upshape {\bf (v')}} appears, then $(r_{12},r_{32})\in \{(0,0),(1,0),(1,2),(2,2)\}$. 
        In this case, we have $h_{32}=0$. 
    \end{enumerate}
\end{proposition}

\begin{proof}[Proofs of Propositions \ref{path p1} and \ref{path p3}] 
We only prove Proposition \ref{path p1} since Proposition \ref{path p3} can be shown similarly. 

From now on, we compute each case of {\upshape {(i)}}-{\upshape {(v)}} of Proposition \ref{rank2} in this setting. 
Let $f\colon e_2A[1]^{\oplus b}\oplus e_3A[1]^{\oplus a}\to e_1A[1]$ be a a minimal right $(\add (e_2A[1]\oplus e_3A[1]))$-approximation of $e_1A[1]$. Then, we clearly have the following statements.
\begin{itemize}
\item $a=0$ if and only if $h_{13}=0$.
\item $b=0$ if and only if $h_{12}=0$.
\end{itemize}

Firstly, we consider the case {\upshape {(i)}} in Proposition \ref{rank2}. 
By the above discussion, 
the path $\alpha$ in \eqref{path pq} has the following form if and only if $h_{13}=0$. 
\[
    \underline{\emin} = 
    \left[\begin{smallmatrix}
        1 \\  -r_{12} \\ 0
    \end{smallmatrix}\middle| 
    \begin{smallmatrix}
        0 \\  -r_{32} \\ 1
    \end{smallmatrix}\right]
    \underset{b_1}{\xrightarrow{(0,\bullet)}}
    \left[\begin{smallmatrix}
        1 \\  -r_{12} \\ 0
    \end{smallmatrix}\middle| 
    \begin{smallmatrix}
        0 \\  r_{32}-b_1 \\ -1
    \end{smallmatrix}\right]
    \underset{b_2}{\xrightarrow{(\bullet,0)}}
    \left[\begin{smallmatrix}
        -1 \\ r_{12} -b_2 \\ 0
    \end{smallmatrix}\middle| 
    \begin{smallmatrix}
        0 \\  r_{32}-b_1 \\ -1
    \end{smallmatrix}\right]
    =
    \left[\begin{smallmatrix}
        -1 \\  0 \\ 0
    \end{smallmatrix}\middle| 
    \begin{smallmatrix}
        0 \\  0 \\ -1
    \end{smallmatrix}\right].        
\]
In this case, $r_{12}=b_2$ holds, and $r_{12}=0$ if and only if $h_{12}=0$. 
This shows the desired assertion for {\upshape {\bf (i)}} and (1) in the statement.

Secondly, we check the cases {\upshape {(ii)}} and {\upshape {(iii)}}. 
We assume that the path $\alpha$ in \eqref{path pq} is given by
\begin{eqnarray}
    \underline{\emin} = 
    \left[\begin{smallmatrix}
        1 \\  -r_{12} \\ 0
    \end{smallmatrix}\middle| 
    \begin{smallmatrix}
        0 \\  -r_{32} \\ 1
    \end{smallmatrix}\right]  
    &\underset{b_1}{\xrightarrow{(1,\bullet)}} &
    \underline{\tau_1} = 
    \left[\begin{smallmatrix}
        1 \\  -r_{12} \\ 0
    \end{smallmatrix}\middle| 
    \begin{smallmatrix}
        1 \\  r_{32}-r_{12}-b_1 \\ -1
    \end{smallmatrix}\right] \\ 
    &\underset{b_2}{\xrightarrow{(\bullet,1)}} &
    \underline{\tau_2} = 
    \left[\begin{smallmatrix}
        0 \\  r_{32}-b_1-b_2 \\ -1
    \end{smallmatrix}\middle| 
    \begin{smallmatrix}
        1 \\  r_{32}-r_{12}-b_1 \\ -1
    \end{smallmatrix}\right] \\
    &\underset{b_3}{\xrightarrow{(1,\bullet)}} &
    \underline{\tau_3} = 
    \left[\begin{smallmatrix}
        0 \\  r_{32}-b_1-b_2 \\ -1
    \end{smallmatrix}\middle| 
    \begin{smallmatrix}
        -1 \\  r_{12}-b_2-b_3 \\ 0
    \end{smallmatrix}\right] = 
    \left[\begin{smallmatrix}
        0 \\ 0 \\ -1
    \end{smallmatrix}\middle| 
    \begin{smallmatrix}
        -1 \\ 0 \\ 0
    \end{smallmatrix}\right]   
\end{eqnarray}
where $b_1,b_2,b_3\in \{0,1\}$. Then, $r_{32}-b_1-b_2= r_{12}-b_2-b_3 =0$ hold. Therefore, $(b_1,b_2,b_3)$ equals to $(r_{32},0,r_{12})$ or $(r_{32}-1,1,r_{12}-1)$. 
The former (resp., latter) one corresponds to the case {\upshape {\bf (ii)}} (resp., {\upshape {\bf (iii)}}), and
(2) (resp. (3)) in the statement follows from the fact that $h_{12}=0$ if and only if $b_3=0$.  

Our computation for the case {\upshape {(iv)}} is similar to the previous one. 
We assume that the path $\alpha$ in \eqref{path pq} is given by
\begin{eqnarray}
    \underline{\emin} = 
    \left[\begin{smallmatrix}
        1 \\  -r_{12} \\ 0
    \end{smallmatrix}\middle| 
    \begin{smallmatrix}
        0 \\  -r_{32} \\ 1
    \end{smallmatrix}\right]  
    &\underset{0}{\xrightarrow{(2,\bullet)}} &
    \underline{\tau_1} = 
    \left[\begin{smallmatrix}
        1 \\  -r_{12} \\ 0
    \end{smallmatrix}\middle| 
    \begin{smallmatrix}
        2 \\  r_{32}-2r_{12} \\ -1
    \end{smallmatrix}\right] \\ 
    &\underset{b_1}{\xrightarrow{(\bullet,1)}} &
    \underline{\tau_2} = 
    \left[\begin{smallmatrix}
        1 \\  r_{32}-r_{12}-b_1 \\ -1
    \end{smallmatrix}\middle| 
    \begin{smallmatrix}
        2 \\  r_{32}-2r_{12} \\ -1
    \end{smallmatrix}\right] \\
    &\underset{0}{\xrightarrow{(2,\bullet)}} &
    \underline{\tau_3} = 
    \left[\begin{smallmatrix}
        1 \\  r_{32}-r_{12}-b_1 \\ -1
    \end{smallmatrix}\middle| 
    \begin{smallmatrix}
        0 \\  r_{32}-2b_1 \\ -1
    \end{smallmatrix}\right]\\
    &\underset{b_2}{\xrightarrow{(\bullet,1)}} &
    \underline{\tau_4} = 
    \left[\begin{smallmatrix}
        -1 \\  r_{12}-b_1-b_2 \\ 0
    \end{smallmatrix}\middle| 
    \begin{smallmatrix}
        0 \\  r_{32}-2b_1 \\ -1
    \end{smallmatrix}\right]
    = 
    \left[\begin{smallmatrix}
        -1 \\ 0 \\ 0
    \end{smallmatrix}\middle| 
    \begin{smallmatrix}
        0 \\ 0 \\ -1
    \end{smallmatrix}\right]   
\end{eqnarray}
where $b_1,b_2\in \{0,1\}$. 
In this case, $b_1=\frac{r_{32}}{2}$ and $b_2=r_{12}-b_1=r_{12}-\frac{r_{32}}{2}$ hold.
Since $h_{12}=0$ if and only if $b_2=0$, we have the assertions {\upshape {(iv)}} and (4).

Finally, we check the case {\upshape {(v)}}. 
We assume that the path $\alpha$ in \eqref{path pq} is given by
\begin{eqnarray}
    \underline{\emin} = 
    \left[\begin{smallmatrix}
        1 \\  -r_{12} \\ 0
    \end{smallmatrix}\middle| 
    \begin{smallmatrix}
        0 \\  -r_{32} \\ 1
    \end{smallmatrix}\right]  
    &\underset{b_1}{\xrightarrow{(1,\bullet)}} &
    \underline{\tau_1} = 
    \left[\begin{smallmatrix}
        1 \\  -r_{12} \\ 0
    \end{smallmatrix}\middle| 
    \begin{smallmatrix}
        1 \\  r_{32}-r_{12}-b_1 \\ -1
    \end{smallmatrix}\right] \\ 
    &\underset{0}{\xrightarrow{(\bullet,2)}} &
    \underline{\tau_2} = 
    \left[\begin{smallmatrix}
        1 \\  2r_{32}-r_{12}-2b_1 \\ -2
    \end{smallmatrix}\middle| 
    \begin{smallmatrix}
        1 \\  r_{32}-r_{12}-b_1 \\ -1
    \end{smallmatrix}\right] \\
    &\underset{b_2}{\xrightarrow{(1,\bullet)}} &
    \underline{\tau_3} = 
    \left[\begin{smallmatrix}
        1 \\  2r_{32}-r_{12}-2b_1 \\ -2
    \end{smallmatrix}\middle| 
    \begin{smallmatrix}
        0 \\  r_{32}-b_1-b_2 \\ -1
    \end{smallmatrix}\right] \\
    &\underset{0}{\xrightarrow{(\bullet,2)}} &
    \underline{\tau_4} = 
    \left[\begin{smallmatrix}
        -1 \\  r_{12}-2b_2 \\ 0
    \end{smallmatrix}\middle| 
    \begin{smallmatrix}
        0 \\  r_{32}-b_1-b_2 \\ -1
    \end{smallmatrix}\right] 
    = 
    \left[\begin{smallmatrix}
        -1 \\ 0 \\ 0
    \end{smallmatrix}\middle| 
    \begin{smallmatrix}
        0 \\ 0 \\ -1
    \end{smallmatrix}\right]   
\end{eqnarray}
where $b_1,b_2\in \{0,1\}$.
Then, we have $b_2=\frac{r_{12}}{2}$ and $b_1=r_{32}-b_2=r_{32}-\frac{r_{12}}{2}$.
In particular, this shows the assertion {\upshape {(v)}} and (5).
It finishes the proof. 
\end{proof}

\section{Proof of Theorem \ref{thm:orthant}} \label{sec:proof_Thm2}
In this section, we prove Theorem \ref{thm:orthant}. 
Throughout this section, let $(A,e = (e_1,e_2,e_3))$ be a finite dimensional $k$-algebra of rank $3$ which is $g$-convex. 
Let $d(A,e)=(d_{ij})_{1\leq i\neq j \leq 3}$ be the datum for $(A,e)$ defined in Definition \ref{def:dAe} 
with $d_{ij}=(l_{ij}, r_{ij}, h_{ij})$. 
In addition, let $d_{+-+}(A,e) = (d_{12},d_{32})$. 
On the other hand, we will use the sets $\mathcal{E}_1$, $\mathcal{E}_2$ and $\mathcal{E}_3$ given in Example \ref{ex:sufficient sets}.

For the signature $(+-+)\in \{\pm\}^3$, 
our strategy of a proof of Theorem \ref{thm:orthant} is the following. 
\begin{itemize}
    \item We show the simplest case $d(0)$. 
    \item We show the cases $d(1)$-$d(5)$, in which case the subfan $\SigmaeAe$ is automatically determined by the maximum $\emax$ and the minimum $\emin$.
    \item We show the cases $d(6)$-$d(9)$ by using Proposition \ref{intermediate}.
    \item We show the cases $d(10)$-$d(13)$ by using Proposition \ref{intermediate} together with Propositions \ref{path p1} and \ref{path p3}. 
    \item Independently of the above discussion, we exclude all other cases (indicated by "$-$") in Table \ref{tab:15list} to complete our proof. 
\end{itemize}

\subsection{Case $d(0)$} \label{case:0}
We claim that $d_{12} = d_{32} = (0,0,0)$ if and only if $h_{12} = h_{32}=0$. 
In fact, if $h_{12}=h_{32}=0$, then we have $e_1Ae_2 = e_1Ae_3A_2$ and $e_3Ae_2 = e_3Ae_1A_2$. It implies that $e_1Ae_2 = e_3Ae_2 =0$.
In this case, we have $\emin = \emax = \cone\{\bm{e}_1, -\bm{e}_2, \bm{e}_3\}$ by Proposition \ref{min_max}. 
Thus, $\SigmaeAe = \Sigma_{d(0)}$ holds as desired. 

\subsection{Cases $d(1)$-$d(5)$} \label{case:1-5}
We only check the case $d(5)$ since others can be shown similarly. 
Assume that $d_{+-+}(A,e) = d(5) = ((1,1,1),(1,2,0))$. 
By Proposition \ref{min_max}, we have 
    \begin{equation}
    \emax = 
    \left[
     \begin{smallmatrix}
        1 \\ 0 \\ 0 
    \end{smallmatrix}\middle|
    \begin{smallmatrix}
        1 \\ -1 \\ 0 
    \end{smallmatrix}\middle| 
    \begin{smallmatrix}
        0 \\ 0 \\ 1 
    \end{smallmatrix}
    \right],\ 
    \emin = \left[
     \begin{smallmatrix}
        1 \\ -1 \\ 0 
    \end{smallmatrix}\middle|
    \begin{smallmatrix}
        0 \\ -1 \\ 0 
    \end{smallmatrix}\middle| 
    \begin{smallmatrix}
        0 \\ -2 \\ 1 
    \end{smallmatrix}
    \right], \ 
    \bm{v}_{\emax} = 
    \begin{bsmallmatrix} 1 \\ 0 \\ 1 \end{bsmallmatrix}\quad \text{and} \quad  
    \bm{v}_{\emin} = 
    \begin{bsmallmatrix} 0 \\ -1 \\ -1 \end{bsmallmatrix}. 
\end{equation}
For $X := \{\emax, \emin\}$, 
the region $\PeX$ is given by the following system of inequalities in $\mathbb{R}^3_{+-+}$: 
\begin{equation} 
x_1 + x_3 \leq 1,\  -x_2-x_3 \leq 1,\ -x_1\leq 0,\ x_2\leq 0,\ -x_3\leq 0.
\end{equation}
Thus, we have 
\begin{eqnarray} 
\PeXZ = \left\{\begin{bsmallmatrix}
        1 \\ 0 \\ 0
    \end{bsmallmatrix}, 
    \begin{bsmallmatrix}
        1 \\ -1 \\ 0
    \end{bsmallmatrix}, 
    \begin{bsmallmatrix}
        0 \\ -1 \\ 0
    \end{bsmallmatrix}, 
    \begin{bsmallmatrix}
        0 \\ -2 \\ 1
    \end{bsmallmatrix},
    \begin{bsmallmatrix}
        0 \\ -1 \\ 1
    \end{bsmallmatrix},
    \begin{bsmallmatrix}
        0 \\ 0 \\ 1
    \end{bsmallmatrix}\right\}.   
\end{eqnarray}
From this description, we find that 
\[
\PeXZ \subseteq \Conv_0(\{\emax,\emin\}) \subseteq \PeAe.
\] 
On the other hand, we have $\mathcal{L}(X) = \{
\begin{bsmallmatrix}
    1 \\ 0 \\ 1
\end{bsmallmatrix}, 
\begin{bsmallmatrix}
    0 \\ -1 \\ -1
\end{bsmallmatrix}
\} \subseteq \mathcal{E}_2$. 
Thus, we can apply Proposition \ref{intermediate} to get $\PeAe = \PeX$.
In this situation, since $\SigmaeAe$ is a convex non-singular fan, it must contain the following maximal cones 
\begin{equation}
    \tau = \left[\begin{smallmatrix}
        1 \\ -1 \\ 0 
    \end{smallmatrix}\middle| 
    \begin{smallmatrix}
        0\\ -1 \\ 1
    \end{smallmatrix}\middle| 
    \begin{smallmatrix}
        0 \\ 0 \\ 1
    \end{smallmatrix}
    \right], \ 
    \tau' = \left[\begin{smallmatrix}
        1 \\ -1 \\ 0 
    \end{smallmatrix}\middle| 
    \begin{smallmatrix}
        0\\ -1 \\ 1
    \end{smallmatrix}\middle| 
    \begin{smallmatrix}
        0 \\ -2 \\ 1
    \end{smallmatrix}
    \right].
\end{equation}
This shows that $\SigmaeAe = \Sigma_{d(5)}$ as desired. 
The corresponding figure is given by Figure \ref{fig:A5}.

\begin{figure}[h]
\begin{tabular}{cccc}
    $\Sigma_{d(5)}$ \\ 
    \begin{tikzpicture}[baseline=0mm, scale=1]
    \coordinate(0) at(0:0); 
    \node(x) at(215:1) {}; 
    \node(y) at(0:1) {}; 
    \node(z) at(90:1) {}; 
    \draw[gray, <-] ($1*(x)$)--($-1*(x)$); 
    \draw[gray, <-] ($1*(y)$)--($-1*(y)$); 
    \draw[gray, <-] ($1*(z)$)--($-1*(z)$); 
    \node at(x) [below]{$\bm{e}_1$}; 
    \node at(y) [below]{$\bm{e}_2$}; 
    \node at(z) [right]{$\bm{e}_3$}; 
    \def\gridwidth{0.1}
    \begin{scope}[gray, line width = \gridwidth]
    \foreach \c in {0,1,2} {
    \draw ($-\c*(y)$)--($2*(x)+ -\c*(y)$); 
    \draw ($\c*(x)$)--($-2*(y)+ \c*(x)$); 
    \draw ($-\c*(y)$)--($2*(z)+ -\c*(y)$); 
    \draw ($\c*(z)$)--($-2*(y)+ \c*(z)$); 
    \draw ($-2*(y)+\c*(z)$)--($-2*(y)+ 2*(x)+ \c*(z)$); 
    \draw ($-2*(y)+\c*(x)$)--($-2*(y)+ 2*(z)+ \c*(x)$); }
    \end{scope}
    \fill[black, opacity = 0.8] (0)circle(0.5mm);
    \coordinate(1) at($1*(x) + 0*(y) + 0*(z)$); 
    \coordinate(2) at($1*(x) + -1*(y) + 0*(z)$); 
    \coordinate(3) at($0*(x) + 0*(y) + 1*(z)$); 
    \coordinate(4) at($0*(x) + -1*(y) + 1*(z)$); 
    \coordinate(5) at($0*(x) + -2*(y) + 1*(z)$); 
    \coordinate(6) at($0*(x) + -1*(y) + 0*(z)$); 
\def\segwidth{0.5mm}
\def\positivecolor{green!30!red}
\def\negativecolor{green!30!orange}
\def\segcolor{green!30!blue}
    \foreach \s\t in {(2)/(6), (5)/(6)} 
    {\draw[line width = 0.4mm, \negativecolor] \s--\t;}
    \draw[line width = \segwidth, \negativecolor] (2)--(5);
    \foreach \s\t in {(1)/(2), (1)/(3), (2)/(3)} 
    {\draw[line width = \segwidth, \positivecolor] \s--\t;}
    \foreach \s\t in {(4)/(2), (4)/(3), (4)/(5)} 
        {\draw[line width = \segwidth, \segcolor] \s--\t;} 
    \foreach \v in {(1),(2),(3),(4),(5),(6)} 
        {\fill[gray, opacity = 0.8] \v circle(1.1mm);} 
\end{tikzpicture}
\end{tabular}
    \caption{}
    \label{fig:A5}
\end{figure}


\subsection{Case $d(6)$} \label{case:6}
We assume that $d_{+-+}(A,e) = d(6) = ((2,1,1),(2,1,0))$.
By Proposition \ref{min_max}, we have 
\begin{equation}
    \emax = 
    \left[
     \begin{smallmatrix}
        1 \\ 0 \\ 0 
    \end{smallmatrix}\middle|
    \begin{smallmatrix}
        2 \\ -1 \\ 0 
    \end{smallmatrix}\middle| 
    \begin{smallmatrix}
        0 \\ 0 \\ 1 
    \end{smallmatrix}
    \right], \ 
    \emin = \left[
     \begin{smallmatrix}
        1 \\ -1 \\ 0 
    \end{smallmatrix}\middle|
    \begin{smallmatrix}
        0 \\ -1 \\ 0 
    \end{smallmatrix}\middle| 
    \begin{smallmatrix}
        0 \\ -1 \\ 1 
    \end{smallmatrix}
    \right], \ 
   {\bm{v}}_{\emax} = 
    \begin{bsmallmatrix} 1 \\ 1 \\ 1 \end{bsmallmatrix}\ \ \text{and} \ \
   {\bm{v}}_{\emin} = 
    \begin{bsmallmatrix} 0 \\ -1 \\ 0 \end{bsmallmatrix}. 
\end{equation}
For $X:=\{\emax,\emin\}$, the set $\PeX$ is given by the following system of inequalities in $\mathbb{R}_{+-+}^3$: 
\begin{equation}
    x_1 + x_2 + x_3 \leq 1,\ -x_2\leq 1,\ -x_1\le 0,\ x_2\le 0,\ -x_3\le 0.  
\end{equation}
Then, we have 
\begin{equation}
    \PeXZ = \left\{ 
    \begin{bsmallmatrix}
    1 \\ 0 \\ 0    
    \end{bsmallmatrix}, 
\begin{bsmallmatrix}
    2 \\ -1 \\ 0    
    \end{bsmallmatrix},
    \begin{bsmallmatrix}
    1 \\ -1 \\ 0    
    \end{bsmallmatrix},
    \begin{bsmallmatrix}
    0 \\ -1 \\ 0    
    \end{bsmallmatrix},
    \begin{bsmallmatrix}
    0 \\ -1 \\ 1    
    \end{bsmallmatrix},
    \begin{bsmallmatrix}
    0 \\ -1 \\ 2    
    \end{bsmallmatrix},
    \begin{bsmallmatrix}
    0 \\ 0 \\ 1    
    \end{bsmallmatrix},
    \begin{bsmallmatrix}
    1 \\ -1 \\ 1    
    \end{bsmallmatrix}
    \right\}. 
\end{equation}
Since $(l_{32},r_{32}) = (2,1)$ by our assumption, we have 
$\bm{w} := \begin{bsmallmatrix}
        0 \\ -1 \\ 2    
    \end{bsmallmatrix}\in \SigmaeAe$ 
by Proposition \ref{coord_plane}. 
Thus, 
\[
\PeXZ \subseteq \Conv_0(\{\emax,\emin, \bm{w}\}) \subseteq \PeAe.
\]
Moreover, since 
$\mathcal{L}(X) = \{ 
\begin{bsmallmatrix}
    1 \\ 1 \\ 1    
\end{bsmallmatrix}
\} \subseteq \mathcal{E}_1$ holds, we have $\PeAe= \PeX$ by Proposition \ref{intermediate}. 
Then, there are only $3$ possibility of $\SigmaeAe$ described in Figure \ref{fig:B1}, 
where the left most one is the fan $\Sigma_{d(6)}$.

\begin{figure}[ht] 
    \begin{tabular}{ccccc}
    $\Sigma_{d(6)}$ & (a) & (b) \\ 
    \begin{tikzpicture}[baseline=0mm, scale=1]
    \coordinate(0) at(0:0); 
    \node(x) at(215:1) {}; 
    \node(y) at(0:1) {}; 
    \node(z) at(90:1) {}; 
    \draw[gray, <-] ($1*(x)$)--($-1*(x)$); 
    \draw[gray, <-] ($1*(y)$)--($-1*(y)$); 
    \draw[gray, <-] ($1*(z)$)--($-1*(z)$); 
    \node at(x) [below]{$\bm{e}_1$}; 
    \node at(y) [below]{$\bm{e}_2$}; 
    \node at(z) [right]{$\bm{e}_3$}; 
    \def\gridwidth{0.1}
    \begin{scope}[gray, line width = \gridwidth]
    \foreach \c in {0,1,2} {
    \draw ($-\c*(y)$)--($2*(x)+ -\c*(y)$); 
    \draw ($\c*(x)$)--($-2*(y)+ \c*(x)$); 
    \draw ($-\c*(y)$)--($2*(z)+ -\c*(y)$); 
    \draw ($\c*(z)$)--($-2*(y)+ \c*(z)$); 
    \draw ($-2*(y)+\c*(z)$)--($-2*(y)+ 2*(x)+ \c*(z)$); 
    \draw ($-2*(y)+\c*(x)$)--($-2*(y)+ 2*(z)+ \c*(x)$); }
    \end{scope}
    \fill[black, opacity = 0.8] (0)circle(0.5mm);
    \coordinate(1) at($1*(x) + 0*(y) + 0*(z)$); 
    \coordinate(2) at($2*(x) + -1*(y) + 0*(z)$); 
    \coordinate(3) at($0*(x) + 0*(y) + 1*(z)$); 
    \coordinate(4) at($1*(x) + -1*(y) + 1*(z)$); 
    \coordinate(5) at($0*(x) + -1*(y) + 2*(z)$); 
    \coordinate(6) at($1*(x) + -1*(y) + 0*(z)$); 
    \coordinate(7) at($0*(x) + -1*(y) + 1*(z)$); 
    \coordinate(8) at($0*(x) + -1*(y) + 0*(z)$); 
\def\segwidth{0.5mm}
\def\positivecolor{green!30!red}
\def\negativecolor{green!30!orange}
\def\segcolor{green!30!blue}
    \foreach \s\t in {(6)/(7), (7)/(8), (8)/(6)} 
    {\draw[line width = 0.4mm, \negativecolor] \s--\t;}
    \foreach \s\t in {(6)/(4), (6)/(5), (5)/(7), (2)/(6)} 
    {\draw[line width = 0.4mm, \segcolor] \s--\t;}
    \foreach \s\t in {(1)/(2), (1)/(3), (2)/(3)} 
    {\draw[line width = \segwidth, \positivecolor] \s--\t;}
    \foreach \s\t in {(4)/(2), (4)/(3), (4)/(5), (3)/(5)} 
        {\draw[line width = \segwidth, \segcolor] \s--\t;} 
    \foreach \v in {(1),(2),(3),(4),(5),(6),(7),(8)} 
        {\fill[gray, opacity = 0.8] \v circle(1.1mm);} 
\end{tikzpicture}
&
\begin{tikzpicture}[baseline=0mm, scale=1]
    \coordinate(0) at(0:0); 
    \node(x) at(215:1) {}; 
    \node(y) at(0:1) {}; 
    \node(z) at(90:1) {}; 
    \draw[gray, <-] ($1*(x)$)--($-1*(x)$); 
    \draw[gray, <-] ($1*(y)$)--($-1*(y)$); 
    \draw[gray, <-] ($1*(z)$)--($-1*(z)$); 
    \node at(x) [below]{$\bm{e}_1$}; 
    \node at(y) [below]{$\bm{e}_2$}; 
    \node at(z) [right]{$\bm{e}_3$}; 
    \def\gridwidth{0.1}
    \begin{scope}[gray, line width = \gridwidth]
    \foreach \c in {0,1,2} {
    \draw ($-\c*(y)$)--($2*(x)+ -\c*(y)$); 
    \draw ($\c*(x)$)--($-2*(y)+ \c*(x)$); 
    \draw ($-\c*(y)$)--($2*(z)+ -\c*(y)$); 
    \draw ($\c*(z)$)--($-2*(y)+ \c*(z)$); 
    \draw ($-2*(y)+\c*(z)$)--($-2*(y)+ 2*(x)+ \c*(z)$); 
    \draw ($-2*(y)+\c*(x)$)--($-2*(y)+ 2*(z)+ \c*(x)$); }
    \end{scope}
    \fill[black, opacity = 0.8] (0)circle(0.5mm);
    \coordinate(1) at($1*(x) + 0*(y) + 0*(z)$); 
    \coordinate(2) at($2*(x) + -1*(y) + 0*(z)$); 
    \coordinate(3) at($0*(x) + 0*(y) + 1*(z)$); 
    \coordinate(4) at($1*(x) + -1*(y) + 1*(z)$); 
    \coordinate(5) at($0*(x) + -1*(y) + 2*(z)$); 
    \coordinate(6) at($1*(x) + -1*(y) + 0*(z)$); 
    \coordinate(7) at($0*(x) + -1*(y) + 1*(z)$); 
    \coordinate(8) at($0*(x) + -1*(y) + 0*(z)$); 
\def\segwidth{0.5mm}
\def\positivecolor{green!30!red}
\def\negativecolor{green!30!orange}
\def\segcolor{green!30!blue}
    \foreach \s\t in {(6)/(7), (7)/(8), (8)/(6)} 
    {\draw[line width = 0.4mm, \negativecolor] \s--\t;}
    \foreach \s\t in {(6)/(4), (4)/(7), (5)/(7), (2)/(6)} 
    {\draw[line width = 0.4mm, \segcolor] \s--\t;}
    \foreach \s\t in {(1)/(2), (1)/(3), (2)/(3)} 
    {\draw[line width = \segwidth, \positivecolor] \s--\t;}
    \foreach \s\t in {(4)/(2), (4)/(3), (4)/(5), (3)/(5)} 
        {\draw[line width = \segwidth, \segcolor] \s--\t;} 
    \foreach \v in {(1),(2),(3),(4),(5),(6),(7),(8)} 
        {\fill[gray, opacity = 0.8] \v circle(1.1mm);} 
\end{tikzpicture}
&
\begin{tikzpicture}[baseline=0mm, scale=1]
    \coordinate(0) at(0:0); 
    \node(x) at(215:1) {}; 
    \node(y) at(0:1) {}; 
    \node(z) at(90:1) {}; 
    \draw[gray, <-] ($1*(x)$)--($-1*(x)$); 
    \draw[gray, <-] ($1*(y)$)--($-1*(y)$); 
    \draw[gray, <-] ($1*(z)$)--($-1*(z)$); 
    \node at(x) [below]{$\bm{e}_1$}; 
    \node at(y) [below]{$\bm{e}_2$}; 
    \node at(z) [right]{$\bm{e}_3$}; 
    \def\gridwidth{0.1}
    \begin{scope}[gray, line width = \gridwidth]
    \foreach \c in {0,1,2} {
    \draw ($-\c*(y)$)--($2*(x)+ -\c*(y)$); 
    \draw ($\c*(x)$)--($-2*(y)+ \c*(x)$); 
    \draw ($-\c*(y)$)--($2*(z)+ -\c*(y)$); 
    \draw ($\c*(z)$)--($-2*(y)+ \c*(z)$); 
    \draw ($-2*(y)+\c*(z)$)--($-2*(y)+ 2*(x)+ \c*(z)$); 
    \draw ($-2*(y)+\c*(x)$)--($-2*(y)+ 2*(z)+ \c*(x)$); }
    \end{scope}
    \fill[black, opacity = 0.8] (0)circle(0.5mm);
    \coordinate(1) at($1*(x) + 0*(y) + 0*(z)$); 
    \coordinate(2) at($2*(x) + -1*(y) + 0*(z)$); 
    \coordinate(3) at($0*(x) + 0*(y) + 1*(z)$); 
    \coordinate(4) at($1*(x) + -1*(y) + 1*(z)$); 
    \coordinate(5) at($0*(x) + -1*(y) + 2*(z)$); 
    \coordinate(6) at($1*(x) + -1*(y) + 0*(z)$); 
    \coordinate(7) at($0*(x) + -1*(y) + 1*(z)$); 
    \coordinate(8) at($0*(x) + -1*(y) + 0*(z)$); 
\def\segwidth{0.5mm}
\def\positivecolor{green!30!red}
\def\negativecolor{green!30!orange}
\def\segcolor{green!30!blue}
    \foreach \s\t in {(6)/(7), (7)/(8), (8)/(6)} 
    {\draw[line width = 0.4mm, \negativecolor] \s--\t;}
    \foreach \s\t in {(5)/(7), (4)/(7), (2)/(7), (2)/(6)} 
    {\draw[line width = 0.4mm, \segcolor] \s--\t;}
    \foreach \s\t in {(1)/(2), (1)/(3), (2)/(3)} 
    {\draw[line width = \segwidth, \positivecolor] \s--\t;}
    \foreach \s\t in {(4)/(2), (4)/(3), (4)/(5), (3)/(5)} 
        {\draw[line width = \segwidth, \segcolor] \s--\t;} 
    \foreach \v in {(1),(2),(3),(4),(5),(6),(7),(8)} 
        {\fill[gray, opacity = 0.8] \v circle(1.1mm);} 
\end{tikzpicture}
\end{tabular}
    \caption{}
    \label{fig:B1}
\end{figure}

Assume that this is the case (a) or (b). Since $\Sigma(A,e)$ is ordered (Theorem \ref{theorem:g-fan is ordered}), we have the following path $p$ around $\bm{u}:= \left[\begin{smallmatrix}
    1 \\ -1 \\ 1
\end{smallmatrix}\right]$. 
\begin{equation}
    \underline{p}_{\bm{u}}\colon\ 
    \left( 
    \left[
    \begin{smallmatrix}
        2 \\ -1 \\ 0 
    \end{smallmatrix}\middle|
    \begin{smallmatrix}
        0 \\ 0 \\ 1 
    \end{smallmatrix}
    \right]
    \underset{2}{\xrightarrow{(\bullet,0)}}
    \left[
     \begin{smallmatrix}
        0 \\ -1 \\ 2 
    \end{smallmatrix}\middle|
    \begin{smallmatrix}
        0 \\ 0 \\ 1 
    \end{smallmatrix}
    \right]
    \underset{0}{\xrightarrow{(1,\bullet)}}
    \left[
     \begin{smallmatrix}
        0 \\ -1 \\ 2 
    \end{smallmatrix}\middle|
    \begin{smallmatrix}
        0 \\ -1 \\ 1 
    \end{smallmatrix}
    \right]\right). 
\end{equation}
However, it is a contradiction by Proposition \ref{rank2}.
Thus, the fan $\Sigma_{d(6)}$ only appears.

\subsection{Case $d(7)$} \label{case:7}
Assume that $d_{+-+}(A,e) = d(7) = ((1,1,1),(1,1,1))$. 
By Proposition \ref{min_max}, we have 
\begin{equation}\label{B2 mm}
\hspace{-3mm}
    \emax = 
    \left[
     \begin{smallmatrix}
        1 \\ 0 \\ 0 
    \end{smallmatrix}\middle|
    \begin{smallmatrix}
        1 \\ -1 \\ 1 
    \end{smallmatrix}\middle| 
    \begin{smallmatrix}
        0 \\ 0 \\ 1 
    \end{smallmatrix}
    \right], \ 
    \emin = \left[
     \begin{smallmatrix}
        1 \\ -1 \\ 0 
    \end{smallmatrix}\middle|
    \begin{smallmatrix}
        0 \\ -1 \\ 0 
    \end{smallmatrix}\middle| 
    \begin{smallmatrix}
        0 \\ -1 \\ 1 
    \end{smallmatrix}
    \right],\ 
   {\bm{v}}_{\emax} = \begin{bsmallmatrix}
        1 \\ 1 \\ 1
    \end{bsmallmatrix} \ \ \text{and}\ \
   {\bm{v}}_{\emin} = \begin{bsmallmatrix}
        0 \\ -1 \\ 0
    \end{bsmallmatrix}.
\end{equation}
We claim that $\tau,\varphi\in \SigmaeAe$, where 
\begin{equation}\label{B2 tautau'}
\tau = 
\left[ 
     \begin{smallmatrix}
        0 \\ -1 \\ 1 
    \end{smallmatrix}\middle|
    \begin{smallmatrix}
        1 \\ -1 \\ 1 
    \end{smallmatrix}\middle| 
    \begin{smallmatrix}
        0 \\ 0 \\ 1 
    \end{smallmatrix}
    \right], \ 
    \varphi = 
\left[ 
     \begin{smallmatrix}
        1 \\ 0 \\ 0 
    \end{smallmatrix}\middle|
    \begin{smallmatrix}
        1 \\ -1 \\ 1 
    \end{smallmatrix}\middle| 
    \begin{smallmatrix}
        1 \\ -1 \\ 0 
    \end{smallmatrix}
    \right]\ \ \text{with}\ \ {\bm{v}}_{\tau}=\begin{bsmallmatrix}
        0 \\ 0 \\ 1
\end{bsmallmatrix}, \ 
\bm{v}_{\varphi} = \begin{bsmallmatrix}
        1 \\ 0 \\ 0
\end{bsmallmatrix}. 
\end{equation}
After that, we get the assertion by the following reason: 
For $X:=\{\emax,\emin,\tau,\varphi\}$, the set $\Pe(X)$ is given by the following system of inequalities in $\mathbb{R}_{+-+}^3$:
\begin{equation}
    x_1+x_2+x_3\leq 1,\ x_1\le 1,\ -x_2\le 1,\ x_3 \le 1,\ -x_1\le 0,\ x_2\le 0,\ -x_3 \le 0.  
\end{equation}
Then, we have    
\begin{equation}
    \PeXZ = \left\{
    \begin{bsmallmatrix}
        1 \\ 0 \\ 0 
    \end{bsmallmatrix}, 
    \begin{bsmallmatrix}
        1 \\ -1 \\ 0 
    \end{bsmallmatrix}, 
    \begin{bsmallmatrix}
        0 \\ -1 \\ 0 
    \end{bsmallmatrix}, 
    \begin{bsmallmatrix}
        0 \\ -1 \\ 1 
    \end{bsmallmatrix}, 
    \begin{bsmallmatrix}
        1 \\ -1 \\ 1 
    \end{bsmallmatrix},
       \begin{bsmallmatrix}
        0 \\ 0 \\ 1
    \end{bsmallmatrix}
    \right\} \subseteq \Conv_0(\{\emax,\emin\}) \subseteq \PeAe. 
\end{equation}
Since  
$\mathcal{L}(X) = \{
    \begin{bsmallmatrix}
        1 \\ 1 \\ 1 
    \end{bsmallmatrix}
    \} \subseteq \mathcal{E}_1$ holds, we obtain $\PeAe = \PeX$ by Proposition \ref{intermediate}. 
Then, these cones uniquely give rise to a convex non-singular fan $\Sigma_{d(7)}$, which coincides with $\SigmaeAe$ as desired (see Figure \ref{fig:B2}).

\begin{figure}[ht]
    \begin{tabular}{cccc}
    $\Sigma_{d(7)}$  \\ 
    \begin{tikzpicture}[baseline=0mm, scale=1]
    \coordinate(0) at(0:0); 
    \node(x) at(215:1) {}; 
    \node(y) at(0:1) {}; 
    \node(z) at(90:1) {}; 
    \draw[gray, <-] ($1*(x)$)--($-1*(x)$); 
    \draw[gray, <-] ($1*(y)$)--($-1*(y)$); 
    \draw[gray, <-] ($1*(z)$)--($-1*(z)$); 
    \node at(x) [below]{$\bm{e}_1$}; 
    \node at(y) [below]{$\bm{e}_2$}; 
    \node at(z) [right]{$\bm{e}_3$}; 
    \def\gridwidth{0.1}
    \begin{scope}[gray, line width = \gridwidth]
    \foreach \c in {0,1,2} {
    \draw ($-\c*(y)$)--($2*(x)+ -\c*(y)$); 
    \draw ($\c*(x)$)--($-2*(y)+ \c*(x)$); 
    \draw ($-\c*(y)$)--($2*(z)+ -\c*(y)$); 
    \draw ($\c*(z)$)--($-2*(y)+ \c*(z)$); 
    \draw ($-2*(y)+\c*(z)$)--($-2*(y)+ 2*(x)+ \c*(z)$); 
    \draw ($-2*(y)+\c*(x)$)--($-2*(y)+ 2*(z)+ \c*(x)$); }
    \end{scope}
    \fill[black, opacity = 0.8] (0)circle(0.5mm);
    \coordinate(1) at($1*(x) + 0*(y) + 0*(z)$); 
    \coordinate(2) at($1*(x) + -1*(y) + 1*(z)$); 
    \coordinate(3) at($0*(x) + 0*(y) + 1*(z)$); 
    \coordinate(4) at($0*(x) + -1*(y) + 1*(z)$); 
    \coordinate(5) at($1*(x) + -1*(y) + 0*(z)$); 
    \coordinate(6) at($0*(x) + -1*(y) + 0*(z)$); 
\def\segwidth{0.5mm}
\def\positivecolor{green!30!red}
\def\negativecolor{green!30!orange}
\def\segcolor{green!30!blue}
    \foreach \s\t in {(6)/(5), (4)/(6), (5)/(4)} 
    {\draw[line width = 0.4mm, \negativecolor] \s--\t;}
    \foreach \s\t in {(1)/(2), (1)/(3), (2)/(3)} 
    {\draw[line width = \segwidth, \positivecolor] \s--\t;}
    \foreach \s\t in {(1)/(5), (2)/(5), (2)/(4), (3)/(4)} 
        {\draw[line width = \segwidth, \segcolor] \s--\t;} 
    \foreach \v in {(1),(2),(3),(4),(5),(6)} 
        {\fill[gray, opacity = 0.8] \v circle(1.1mm);} 
\end{tikzpicture}
    \end{tabular}
    \caption{}
    \label{fig:B2}
\end{figure}

We first show that the above cone $\tau$ lies in $\SigmaeAe$. 
Consider an arrow  
\begin{equation}
    \emax = 
    \left[
     \begin{smallmatrix}
        1 \\ 0 \\ 0 
    \end{smallmatrix}\middle|
    \begin{smallmatrix}
        1 \\ -1 \\ 1 
    \end{smallmatrix}\middle| 
    \begin{smallmatrix}
        0 \\ 0 \\ 1 
    \end{smallmatrix}
    \right] \xrightarrow{(\bullet,a_1,a_2)} 
    \tau' := \left[
     \begin{smallmatrix}
        a_1-1 \\ -a_1 \\ a_1+a_2 
    \end{smallmatrix}\middle|
    \begin{smallmatrix}
        1 \\ -1 \\ 1 
    \end{smallmatrix}\middle| 
    \begin{smallmatrix}
        0 \\ 0 \\ 1 
    \end{smallmatrix}
    \right] 
\end{equation}
in the Hasse quiver of $(\SigmaeAe,\leq)$, where $(a_1,a_2)\in \{(0,0),(0,1),(0,2),(1,0),(1,1),(2,0)\}$. We show $(a_1,a_2)=(1,0)$ by excluding all other cases. 
Firstly, if $a_1=0$, then we have $\left[\begin{smallmatrix}
        -1 \\ 0 \\ a_2 
    \end{smallmatrix}\middle|
    \begin{smallmatrix}
        1 \\ -1 \\ 1 
    \end{smallmatrix}\right]\in \SigmaAe$, which is a contradiction to the sign-coherent property. 
Secondly, if $(a_1,a_2)=(1,1)$, then we have  $\left[\begin{smallmatrix}
        0 \\ -1 \\ 2 
\end{smallmatrix}\right]\in \SigmaAe$. 
It implies that $r_{32}=2$ by Proposition \ref{coord_plane}, which is a contradiction to our assumption that $r_{32}=1$. 
Finally, if $(a_1,a_2)=(2,0)$, then we have 
$\bm{w}:= \left[\begin{smallmatrix}
        1 \\ -2 \\ 2 
    \end{smallmatrix}\right]\in \SigmaAe$. 
However, we have $\langle{\bm{v}}_{\emin}, \bm{w} \rangle = 2 \not\leq 1$, which is a contradiction.
Consequently, we obtain $(a_1,a_2)=(1,0)$ and 
\begin{equation}
\tau' = \tau =
\left[ 
     \begin{smallmatrix}
        0 \\ -1 \\ 1 
    \end{smallmatrix}\middle|
    \begin{smallmatrix}
        1 \\ -1 \\ 1 
    \end{smallmatrix}\middle| 
    \begin{smallmatrix}
        0 \\ 0 \\ 1 
    \end{smallmatrix}
    \right] \in \SigmaeAe. 
\end{equation}
By permuting a role of $\bm{e}_1$ and $\bm{e}_3$, we also have  
\begin{equation}
\varphi = 
\left[ 
     \begin{smallmatrix}
        1 \\ 0 \\ 0 
    \end{smallmatrix}\middle|
    \begin{smallmatrix}
        1 \\ -1 \\ 1 
    \end{smallmatrix}\middle| 
    \begin{smallmatrix}
        1 \\ -1 \\ 0 
    \end{smallmatrix}
    \right] \in \SigmaeAe. 
\end{equation}
This completes the proof.

\subsection{Case $d(8)$} \label{case:8}
Assume that $d_{+-+}(A,e) = d(8) = ((2,1,1),(1,1,1))$. 
In this case, one can see that $\emax$ and $\emin$ are the same as \eqref{B2 mm} by Proposition \ref{min_max}. 
Furthermore, the maximal cone $\tau$ in \eqref{B2 tautau'} belongs to $\SigmaeAe$ by the same discussion. 
For $X=\{\emax,\emin,\tau\}$, the set $\PeX$ is given by the following system of inequalities in $\mathbb{R}_{+-+}^3$:
\begin{equation}
    x_1+x_2+x_3\leq 1,\ -x_2\le 1,\ x_3 \le 1,\ -x_1\le 0,\ x_2\le 0,\ -x_3 \le 0.  
\end{equation}
Then, we have 
\begin{equation} 
    \PeXZ = \left\{
    \begin{bsmallmatrix}
        1 \\ 0 \\ 0 
    \end{bsmallmatrix}, 
    \begin{bsmallmatrix}
        1 \\ -1 \\ 0 
    \end{bsmallmatrix}, 
    \begin{bsmallmatrix}
        0 \\ -1 \\ 0 
    \end{bsmallmatrix}, 
    \begin{bsmallmatrix}
        0 \\ -1 \\ 1 
    \end{bsmallmatrix}, 
    \begin{bsmallmatrix}
        1 \\ -1 \\ 1 
    \end{bsmallmatrix},
    \begin{bsmallmatrix}
        0 \\ 0 \\ 1 
    \end{bsmallmatrix},
    \begin{bsmallmatrix}
        2 \\ -1 \\ 0 
    \end{bsmallmatrix}
    \right\}. 
\end{equation}
Since $(l_{12},r_{12}) = (2,1)$, we have 
$\begin{bsmallmatrix}
        2 \\ -1 \\ 0 
    \end{bsmallmatrix}\in \SigmaeAe$ by Proposition\;\ref{coord_plane}. 
It implies that $\PeXZ \subseteq \PeAe$. 
On the other hand, we clearly have  
    $\mathcal{L}(X) = \{
    \begin{bsmallmatrix}
        1 \\ 1 \\ 1 
    \end{bsmallmatrix}
    \} \subseteq \mathcal{E}_1$. 
By Proposition \ref{intermediate}, we have $\PeAe = \PeX$. 
Then, $\SigmaeAe$ is one of fans described in Figure \ref{fig:B3}, where the left one is the fan $\Sigma_{d(8)}$.

\begin{figure}[ht]
    \begin{tabular}{ccccc}
    $\Sigma_{d(8)}$ & (a)  \\ 
    \begin{tikzpicture}[baseline=0mm, scale=1]
    \coordinate(0) at(0:0); 
    \node(x) at(215:1) {}; 
    \node(y) at(0:1) {}; 
    \node(z) at(90:1) {}; 
    \draw[gray, <-] ($1*(x)$)--($-1*(x)$); 
    \draw[gray, <-] ($1*(y)$)--($-1*(y)$); 
    \draw[gray, <-] ($1*(z)$)--($-1*(z)$); 
    \node at(x) [below]{$\bm{e}_1$}; 
    \node at(y) [below]{$\bm{e}_2$}; 
    \node at(z) [right]{$\bm{e}_3$}; 
    \def\gridwidth{0.1}
    \begin{scope}[gray, line width = \gridwidth]
    \foreach \c in {0,1,2} {
    \draw ($-\c*(y)$)--($2*(x)+ -\c*(y)$); 
    \draw ($\c*(x)$)--($-2*(y)+ \c*(x)$); 
    \draw ($-\c*(y)$)--($2*(z)+ -\c*(y)$); 
    \draw ($\c*(z)$)--($-2*(y)+ \c*(z)$); 
    \draw ($-2*(y)+\c*(z)$)--($-2*(y)+ 2*(x)+ \c*(z)$); 
    \draw ($-2*(y)+\c*(x)$)--($-2*(y)+ 2*(z)+ \c*(x)$); }
    \end{scope}
    \fill[black, opacity = 0.8] (0)circle(0.5mm);
    \coordinate(1) at($1*(x) + 0*(y) + 0*(z)$); 
    \coordinate(2) at($1*(x) + -1*(y) + 1*(z)$); 
    \coordinate(3) at($0*(x) + 0*(y) + 1*(z)$); 
    \coordinate(4) at($0*(x) + -1*(y) + 1*(z)$); 
    \coordinate(5) at($2*(x) + -1*(y) + 0*(z)$); 
    \coordinate(6) at($1*(x) + -1*(y) + 0*(z)$); 
    \coordinate(7) at($0*(x) + 0*(y) + 1*(z)$); 
\def\segwidth{0.5mm}
\def\positivecolor{green!30!red}
\def\negativecolor{green!30!orange}
\def\segcolor{green!30!blue}
    \foreach \s\t in {(4)/(6), (6)/(7), (7)/(4)} 
    {\draw[line width = 0.4mm, \negativecolor] \s--\t;}
    \foreach \s\t in {(2)/(6), (6)/(5)} 
    {\draw[line width = 0.4mm, \segcolor] \s--\t;}
    \foreach \s\t in {(1)/(2), (1)/(3), (2)/(3)} 
    {\draw[line width = \segwidth, \positivecolor] \s--\t;}
    \foreach \s\t in {(5)/(2), (4)/(2), (3)/(4), (1)/(5)} 
        {\draw[line width = \segwidth, \segcolor] \s--\t;} 
    \foreach \v in {(1),(2),(3),(4),(5),(6),(7)} 
        {\fill[gray, opacity = 0.8] \v circle(1.1mm);} 
\end{tikzpicture}
&
\begin{tikzpicture}[baseline=0mm, scale=1]
    \coordinate(0) at(0:0); 
    \node(x) at(215:1) {}; 
    \node(y) at(0:1) {}; 
    \node(z) at(90:1) {}; 
    \draw[gray, <-] ($1*(x)$)--($-1*(x)$); 
    \draw[gray, <-] ($1*(y)$)--($-1*(y)$); 
    \draw[gray, <-] ($1*(z)$)--($-1*(z)$); 
    \node at(x) [below]{$\bm{e}_1$}; 
    \node at(y) [below]{$\bm{e}_2$}; 
    \node at(z) [right]{$\bm{e}_3$}; 
    \def\gridwidth{0.1}
    \begin{scope}[gray, line width = \gridwidth]
    \foreach \c in {0,1,2} {
    \draw ($-\c*(y)$)--($2*(x)+ -\c*(y)$); 
    \draw ($\c*(x)$)--($-2*(y)+ \c*(x)$); 
    \draw ($-\c*(y)$)--($2*(z)+ -\c*(y)$); 
    \draw ($\c*(z)$)--($-2*(y)+ \c*(z)$); 
    \draw ($-2*(y)+\c*(z)$)--($-2*(y)+ 2*(x)+ \c*(z)$); 
    \draw ($-2*(y)+\c*(x)$)--($-2*(y)+ 2*(z)+ \c*(x)$); }
    \end{scope}
    \fill[black, opacity = 0.8] (0)circle(0.5mm);
    \coordinate(1) at($1*(x) + 0*(y) + 0*(z)$); 
    \coordinate(2) at($1*(x) + -1*(y) + 1*(z)$); 
    \coordinate(3) at($0*(x) + 0*(y) + 1*(z)$); 
    \coordinate(4) at($0*(x) + -1*(y) + 1*(z)$); 
    \coordinate(5) at($2*(x) + -1*(y) + 0*(z)$); 
    \coordinate(6) at($1*(x) + -1*(y) + 0*(z)$); 
    \coordinate(7) at($0*(x) + 0*(y) + 1*(z)$); 
\def\segwidth{0.5mm}
\def\positivecolor{green!30!red}
\def\negativecolor{green!30!orange}
\def\segcolor{green!30!blue}
    \foreach \s\t in {(4)/(6), (6)/(7), (7)/(4)} 
    {\draw[line width = 0.4mm, \negativecolor] \s--\t;}
    \foreach \s\t in {(5)/(4), (6)/(5)} 
    {\draw[line width = 0.4mm, \segcolor] \s--\t;}
    \foreach \s\t in {(1)/(2), (1)/(3), (2)/(3)} 
    {\draw[line width = \segwidth, \positivecolor] \s--\t;}
    \foreach \s\t in {(5)/(2), (4)/(2), (3)/(4), (1)/(5)} 
        {\draw[line width = \segwidth, \segcolor] \s--\t;} 
    \foreach \v in {(1),(2),(3),(4),(5),(6),(7)} 
        {\fill[gray, opacity = 0.8] \v circle(1.1mm);} 
\end{tikzpicture}
\end{tabular}
    \caption{}
    \label{fig:B3}
\end{figure}

If this is (a), then there exists an arrow 
\begin{equation}
    \tau = \left[
    \begin{smallmatrix}
        0 \\ -1 \\ 1
    \end{smallmatrix}\middle|
    \begin{smallmatrix}
        1 \\ -1 \\ 1
    \end{smallmatrix}\middle|
    \begin{smallmatrix}
        0 \\ 0 \\ 1
    \end{smallmatrix}
    \right] \xrightarrow{(-1,2,\bullet)}
    \left[
    \begin{smallmatrix}
        0 \\ -1 \\ 1
    \end{smallmatrix}\middle|
    \begin{smallmatrix}
        1 \\ -1 \\ 1
    \end{smallmatrix}\middle|
    \begin{smallmatrix}
        2 \\ -1 \\ 0
    \end{smallmatrix}
    \right] 
\end{equation}
in the Hasse quiver of $(\SigmaeAe,\leq)$. However, it contradicts to Proposition \ref{prop:basics_gfan}(5). 
Thus, the fan $\Sigma_{d(8)}$ only appears.

\subsection{Case $d(9)$} \label{case:9}
Assume that $d_{+-+}(A,e) = d(9) = ((2,1,1),(2,1,1))$. 
In this case, $\emax$ and $\emin$ are the same as \eqref{B2 mm} by Proposition \ref{min_max}. 
For $X:=\{\emax,\emin\}$, the set $\PeX$ is given by the following system of inequalities in $\mathbb{R}_{+-+}^3$:
\begin{equation}
    x_1+x_2+x_3\leq 1,\ -x_2\le 1,\ -x_1\le 0,\ x_2\le 0,\ -x_3 \le 0.  
\end{equation}
Then, we have
\begin{equation}
    \PeXZ = \left\{
    \begin{bsmallmatrix}
        1 \\ 0 \\ 0 
    \end{bsmallmatrix}, 
    \begin{bsmallmatrix}
        1 \\ -1 \\ 0 
    \end{bsmallmatrix}, 
    \begin{bsmallmatrix}
        0 \\ -1 \\ 0 
    \end{bsmallmatrix}, 
    \begin{bsmallmatrix}
        0 \\ -1 \\ 1 
    \end{bsmallmatrix}, 
    \begin{bsmallmatrix}
        1 \\ -1 \\ 1 
    \end{bsmallmatrix},
    \begin{bsmallmatrix}
        0 \\ 0 \\ 1 
    \end{bsmallmatrix},
    \begin{bsmallmatrix}
        2 \\ -1 \\ 0 
    \end{bsmallmatrix},
       \begin{bsmallmatrix}
        0 \\ -1 \\ 2 
    \end{bsmallmatrix}
    \right\}.
\end{equation}
We have 
$\begin{bsmallmatrix}
        2 \\ -1 \\ 0 
    \end{bsmallmatrix}, \begin{bsmallmatrix}
        0 \\ -1 \\ 2 
    \end{bsmallmatrix}\in \SigmaeAe$ 
by $(l_{12},r_{12}) = (l_{32},r_{32})= (2,1)$.
Thus, we have $\PeXZ \subseteq \PeAe$. 
Moreover, $\mathcal{L}(X) = \{
\begin{bsmallmatrix}
    1 \\ 1 \\ 1
\end{bsmallmatrix}
\}\subseteq \mathcal{E}_1$ holds by a direct calculation.
Consequently, we have $\PeAe = \PeX$ by Proposition \ref{intermediate}. 
Therefore, there are $3$ possibility of $\SigmaeAe$ described in Figure \ref{fig:B4}, where the left most one is the fan $\Sigma_{d(9)}$. 

\begin{figure}[ht]
    \begin{tabular}{cccc}
    $\Sigma_{d(9)}$ & (a) & (b) \\
    \begin{tikzpicture}[baseline=0mm, scale=1]
    \coordinate(0) at(0:0); 
    \node(x) at(215:1) {}; 
    \node(y) at(0:1) {}; 
    \node(z) at(90:1) {}; 
    \draw[gray, <-] ($1*(x)$)--($-1*(x)$); 
    \draw[gray, <-] ($1*(y)$)--($-1*(y)$); 
    \draw[gray, <-] ($1*(z)$)--($-1*(z)$); 
    \node at(x) [below]{$\bm{e}_1$}; 
    \node at(y) [below]{$\bm{e}_2$}; 
    \node at(z) [right]{$\bm{e}_3$}; 
    \def\gridwidth{0.1}
    \begin{scope}[gray, line width = \gridwidth]
    \foreach \c in {0,1,2} {
    \draw ($-\c*(y)$)--($2*(x)+ -\c*(y)$); 
    \draw ($\c*(x)$)--($-2*(y)+ \c*(x)$); 
    \draw ($-\c*(y)$)--($2*(z)+ -\c*(y)$); 
    \draw ($\c*(z)$)--($-2*(y)+ \c*(z)$); 
    \draw ($-2*(y)+\c*(z)$)--($-2*(y)+ 2*(x)+ \c*(z)$); 
    \draw ($-2*(y)+\c*(x)$)--($-2*(y)+ 2*(z)+ \c*(x)$); }
    \end{scope}
    \fill[black, opacity = 0.8] (0)circle(0.5mm);
    \coordinate(1) at($1*(x) + 0*(y) + 0*(z)$); 
    \coordinate(2) at($1*(x) + -1*(y) + 1*(z)$); 
    \coordinate(3) at($0*(x) + 0*(y) + 1*(z)$); 
    \coordinate(4) at($2*(x) + -1*(y) + 0*(z)$); 
    \coordinate(5) at($0*(x) + -1*(y) + 2*(z)$); 
    \coordinate(6) at($1*(x) + -1*(y) + 0*(z)$); 
    \coordinate(7) at($0*(x) + -1*(y) + 1*(z)$); 
    \coordinate(8) at($0*(x) + -1*(y) + 0*(z)$); 
\def\segwidth{0.5mm}
\def\positivecolor{green!30!red}
\def\negativecolor{green!30!orange}
\def\segcolor{green!30!blue}
    \foreach \s\t in {(6)/(7), (7)/(8), (8)/(6)} 
    {\draw[line width = 0.4mm, \negativecolor] \s--\t;}
    \foreach \s\t in {(2)/(6), (2)/(7), (5)/(7),(4)/(6)} 
    {\draw[line width = 0.4mm, \segcolor] \s--\t;}
    \foreach \s\t in {(1)/(2), (1)/(3), (2)/(3)} 
    {\draw[line width = \segwidth, \positivecolor] \s--\t;}
    \foreach \s\t in {(1)/(4), (5)/(3), (2)/(5), (2)/(4)} 
        {\draw[line width = \segwidth, \segcolor] \s--\t;} 
    \foreach \v in {(1),(2),(3),(4),(5),(6),(7),(8)} 
        {\fill[gray, opacity = 0.8] \v circle(1.1mm);} 
\end{tikzpicture}
&
\begin{tikzpicture}[baseline=0mm, scale=1]
    \coordinate(0) at(0:0); 
    \node(x) at(215:1) {}; 
    \node(y) at(0:1) {}; 
    \node(z) at(90:1) {}; 
    \draw[gray, <-] ($1*(x)$)--($-1*(x)$); 
    \draw[gray, <-] ($1*(y)$)--($-1*(y)$); 
    \draw[gray, <-] ($1*(z)$)--($-1*(z)$); 
    \node at(x) [below]{$\bm{e}_1$}; 
    \node at(y) [below]{$\bm{e}_2$}; 
    \node at(z) [right]{$\bm{e}_3$}; 
    \def\gridwidth{0.1}
    \begin{scope}[gray, line width = \gridwidth]
    \foreach \c in {0,1,2} {
    \draw ($-\c*(y)$)--($2*(x)+ -\c*(y)$); 
    \draw ($\c*(x)$)--($-2*(y)+ \c*(x)$); 
    \draw ($-\c*(y)$)--($2*(z)+ -\c*(y)$); 
    \draw ($\c*(z)$)--($-2*(y)+ \c*(z)$); 
    \draw ($-2*(y)+\c*(z)$)--($-2*(y)+ 2*(x)+ \c*(z)$); 
    \draw ($-2*(y)+\c*(x)$)--($-2*(y)+ 2*(z)+ \c*(x)$); }
    \end{scope}
    \fill[black, opacity = 0.8] (0)circle(0.5mm);
    \coordinate(1) at($1*(x) + 0*(y) + 0*(z)$); 
    \coordinate(2) at($1*(x) + -1*(y) + 1*(z)$); 
    \coordinate(3) at($0*(x) + 0*(y) + 1*(z)$); 
    \coordinate(4) at($2*(x) + -1*(y) + 0*(z)$); 
    \coordinate(5) at($0*(x) + -1*(y) + 2*(z)$); 
    \coordinate(6) at($1*(x) + -1*(y) + 0*(z)$); 
    \coordinate(7) at($0*(x) + -1*(y) + 1*(z)$); 
    \coordinate(8) at($0*(x) + -1*(y) + 0*(z)$); 
\def\segwidth{0.5mm}
\def\positivecolor{green!30!red}
\def\negativecolor{green!30!orange}
\def\segcolor{green!30!blue}
    \foreach \s\t in {(6)/(7), (7)/(8), (8)/(6)} 
    {\draw[line width = 0.4mm, \negativecolor] \s--\t;}
    \foreach \s\t in {(4)/(7), (2)/(7), (5)/(7),(4)/(6)} 
    {\draw[line width = 0.4mm, \segcolor] \s--\t;}
    \foreach \s\t in {(1)/(2), (1)/(3), (2)/(3)} 
    {\draw[line width = \segwidth, \positivecolor] \s--\t;}
    \foreach \s\t in {(1)/(4), (5)/(3), (2)/(5), (2)/(4)} 
        {\draw[line width = \segwidth, \segcolor] \s--\t;} 
    \foreach \v in {(1),(2),(3),(4),(5),(6),(7),(8)} 
        {\fill[gray, opacity = 0.8] \v circle(1.1mm);} 
\end{tikzpicture}
&
\begin{tikzpicture}[baseline=0mm, scale=1]
    \coordinate(0) at(0:0); 
    \node(x) at(215:1) {}; 
    \node(y) at(0:1) {}; 
    \node(z) at(90:1) {}; 
    \draw[gray, <-] ($1*(x)$)--($-1*(x)$); 
    \draw[gray, <-] ($1*(y)$)--($-1*(y)$); 
    \draw[gray, <-] ($1*(z)$)--($-1*(z)$); 
    \node at(x) [below]{$\bm{e}_1$}; 
    \node at(y) [below]{$\bm{e}_2$}; 
    \node at(z) [right]{$\bm{e}_3$}; 
    \def\gridwidth{0.1}
    \begin{scope}[gray, line width = \gridwidth]
    \foreach \c in {0,1,2} {
    \draw ($-\c*(y)$)--($2*(x)+ -\c*(y)$); 
    \draw ($\c*(x)$)--($-2*(y)+ \c*(x)$); 
    \draw ($-\c*(y)$)--($2*(z)+ -\c*(y)$); 
    \draw ($\c*(z)$)--($-2*(y)+ \c*(z)$); 
    \draw ($-2*(y)+\c*(z)$)--($-2*(y)+ 2*(x)+ \c*(z)$); 
    \draw ($-2*(y)+\c*(x)$)--($-2*(y)+ 2*(z)+ \c*(x)$); }
    \end{scope}
    \fill[black, opacity = 0.8] (0)circle(0.5mm);
    \coordinate(1) at($1*(x) + 0*(y) + 0*(z)$); 
    \coordinate(2) at($1*(x) + -1*(y) + 1*(z)$); 
    \coordinate(3) at($0*(x) + 0*(y) + 1*(z)$); 
    \coordinate(4) at($2*(x) + -1*(y) + 0*(z)$); 
    \coordinate(5) at($0*(x) + -1*(y) + 2*(z)$); 
    \coordinate(6) at($1*(x) + -1*(y) + 0*(z)$); 
    \coordinate(7) at($0*(x) + -1*(y) + 1*(z)$); 
    \coordinate(8) at($0*(x) + -1*(y) + 0*(z)$); 
\def\segwidth{0.5mm}
\def\positivecolor{green!30!red}
\def\negativecolor{green!30!orange}
\def\segcolor{green!30!blue}
    \foreach \s\t in {(6)/(7), (7)/(8), (8)/(6)} 
    {\draw[line width = 0.4mm, \negativecolor] \s--\t;}
    \foreach \s\t in {(2)/(6), (5)/(6), (5)/(7),(4)/(6)} 
    {\draw[line width = 0.4mm, \segcolor] \s--\t;}
    \foreach \s\t in {(1)/(2), (1)/(3), (2)/(3)} 
    {\draw[line width = \segwidth, \positivecolor] \s--\t;}
    \foreach \s\t in {(1)/(4), (5)/(3), (2)/(5), (2)/(4)} 
        {\draw[line width = \segwidth, \segcolor] \s--\t;} 
    \foreach \v in {(1),(2),(3),(4),(5),(6),(7),(8)} 
        {\fill[gray, opacity = 0.8] \v circle(1.1mm);} 
\end{tikzpicture}
\end{tabular}    
    \caption{}
    \label{fig:B4}
\end{figure}

If this is the case (a), then we have the following path $p$ around $\bm{u}:=\left[
\begin{smallmatrix}
    1 \\ -1 \\ 1
\end{smallmatrix}\right]$.
\begin{equation}\nonumber
    \underline{p}_{\bm{u}}\colon\ 
    \left( 
    \underline{\emax} = \left[
    \begin{smallmatrix}
        1 \\ 0 \\ 0
    \end{smallmatrix}\middle|
    \begin{smallmatrix}
        0 \\ 0 \\ 1
    \end{smallmatrix}
    \right] 
    \underset{1}{\xrightarrow{(\bullet,1)}}
    \left[
    \begin{smallmatrix}
        0 \\ -1 \\ 2
    \end{smallmatrix}\middle|
    \begin{smallmatrix}
        0 \\ 0 \\ 1
    \end{smallmatrix}
    \right]
    \underset{0}{\xrightarrow{(1,\bullet)}}
    \left[
    \begin{smallmatrix}
        0 \\ -1 \\ 2
    \end{smallmatrix}\middle|
    \begin{smallmatrix}
        0 \\ -1 \\ 1
    \end{smallmatrix}
    \right]
    \underset{2}{\xrightarrow{(\bullet,0)}}
    \left[
    \begin{smallmatrix}
        2 \\ -1 \\ 0
    \end{smallmatrix}\middle|
    \begin{smallmatrix}
        0 \\ -1 \\ 1
    \end{smallmatrix}
    \right]\right).
\end{equation}
However, it gives a contradiction by Proposition \ref{rank2}. 
In addition, by permuting a role of $\bm{e}_1$ and $\bm{e}_3$, (b) can not appear. 
Thus, the fan $\Sigma_{d(9)}$ only appears. 


\subsection{Case $d(10)$} \label{case:10}
We consider the case $d_{+-+}(A,e) = d(10) = ((1,2,1),(1,1,0))$. 
By Proposition \ref{min_max}, we have 
\begin{equation}
    \emax = 
    \left[
     \begin{smallmatrix}
        1 \\ 0 \\ 0 
    \end{smallmatrix}\middle|
    \begin{smallmatrix}
        1 \\ -1 \\ 0 
    \end{smallmatrix}\middle| 
    \begin{smallmatrix}
        0 \\ 0 \\ 1 
    \end{smallmatrix}
    \right], \  
    \emin = \left[
     \begin{smallmatrix}
        1 \\ -2 \\ 0 
    \end{smallmatrix}\middle|
    \begin{smallmatrix}
        0 \\ -1 \\ 0 
    \end{smallmatrix}\middle| 
    \begin{smallmatrix}
        0 \\ -1 \\ 1 
    \end{smallmatrix}
    \right], \ 
    \bm{v}_{\emax} = \begin{bsmallmatrix}
        1 \\ 0 \\ 1
    \end{bsmallmatrix} \ \text{and} \ 
    \bm{v}_{\emin} = \begin{bsmallmatrix}
        -1 \\ -1 \\ 0
    \end{bsmallmatrix}.
\end{equation}

For $X:=\{\emax,\emax\}$, the set $\PeX$ is given by the following system of inequalities in $\mathbb{R}_{+-+}^3$: 
\begin{equation}
    x_1 + x_3 \leq 1,\ -x_1-x_2 \leq 1,\ -x_1\le 0,\ x_2\le 0,\ -x_3 \le 0. 
\end{equation}
Then, we have 
\begin{equation}
    \PeXZ = \left\{
    \begin{bsmallmatrix}
        1 \\ 0 \\ 0
    \end{bsmallmatrix}, 
    \begin{bsmallmatrix}
        1 \\ -1 \\ 0
    \end{bsmallmatrix}, 
    \begin{bsmallmatrix}
        1 \\ -2 \\ 0
    \end{bsmallmatrix}, 
    \begin{bsmallmatrix}
        0 \\ -1 \\ 0
    \end{bsmallmatrix}, 
    \begin{bsmallmatrix}
        0 \\ -1 \\ 1
    \end{bsmallmatrix},
    \begin{bsmallmatrix}
        0 \\ 0 \\ -1
    \end{bsmallmatrix} 
    \right\} \subseteq \PeAe. 
\end{equation}
On the other hand, we have 
$\mathcal{L}(X) = \{
\begin{bsmallmatrix}
    1 \\ 0 \\ 1
\end{bsmallmatrix}, 
\begin{bsmallmatrix}
    -1 \\ -1 \\ 0
\end{bsmallmatrix}
\} \subseteq \mathcal{E}_1$. 
Thus, we have $\PeAe = \PeX$ by Proposition \ref{intermediate}. 
In this case, $\SigmaeAe$ is one of fans in Figure \ref{fig:C}, where the left (resp., right) one corresponds to the fan $\Sigma_{d(10),0}$ (resp., $\Sigma_{d(10),1}$). 
In the following, we show that the fan $\Sigma_{d(10),0}$ (resp., $\Sigma_{d(10),1}$) appears if and only if $h_{13}=0$ (resp., $h_{13}=1$).

\begin{figure}[ht]
    \begin{tabular}{ccccc}
    $\Sigma_{d(10),0}$ & $\Sigma_{d(10),1}$ \\  
\begin{tikzpicture}[baseline=0mm, scale=1]
    \coordinate(0) at(0:0); 
    \node(x) at(215:1) {}; 
    \node(y) at(0:1) {}; 
    \node(z) at(90:1) {}; 
    \draw[gray, <-] ($1*(x)$)--($-1*(x)$); 
    \draw[gray, <-] ($1*(y)$)--($-1*(y)$); 
    \draw[gray, <-] ($1*(z)$)--($-1*(z)$); 
    \node at(x) [below]{$\bm{e}_1$}; 
    \node at(y) [below]{$\bm{e}_2$}; 
    \node at(z) [right]{$\bm{e}_3$}; 
    \def\gridwidth{0.1}
    \begin{scope}[gray, line width = \gridwidth]
    \foreach \c in {0,1,2} {
    \draw ($-\c*(y)$)--($2*(x)+ -\c*(y)$); 
    \draw ($\c*(x)$)--($-2*(y)+ \c*(x)$); 
    \draw ($-\c*(y)$)--($2*(z)+ -\c*(y)$); 
    \draw ($\c*(z)$)--($-2*(y)+ \c*(z)$); 
    \draw ($-2*(y)+\c*(z)$)--($-2*(y)+ 2*(x)+ \c*(z)$); 
    \draw ($-2*(y)+\c*(x)$)--($-2*(y)+ 2*(z)+ \c*(x)$); }
    \end{scope}
    \fill[black, opacity = 0.8] (0)circle(0.5mm);
    \coordinate(1) at($1*(x) + 0*(y) + 0*(z)$); 
    \coordinate(2) at($1*(x) + -1*(y) + 0*(z)$); 
    \coordinate(3) at($0*(x) + 0*(y) + 1*(z)$); 
    \coordinate(4) at($0*(x) + -1*(y) + 1*(z)$); 
    \coordinate(5) at($1*(x) + -2*(y) + 0*(z)$); 
    \coordinate(6) at($0*(x) + -1*(y) + 0*(z)$); 
\def\segwidth{0.5mm}
\def\positivecolor{green!30!red}
\def\negativecolor{green!30!orange}
\def\segcolor{green!30!blue}
    \foreach \s\t in {(5)/(6), (6)/(4)} 
    {\draw[line width = 0.4mm, \negativecolor] \s--\t;}
    \draw[line width = \segwidth, \negativecolor] (4)--(5);
    \foreach \s\t in {(1)/(2), (1)/(3), (2)/(3)} 
    {\draw[line width = \segwidth, \positivecolor] \s--\t;}
    \foreach \s\t in {(3)/(4), (5)/(3), (2)/(5)} 
        {\draw[line width = \segwidth, \segcolor] \s--\t;} 
    \foreach \v in {(1),(2),(3),(4),(5),(6)} 
        {\fill[gray, opacity = 0.8] \v circle(1.1mm);} 
\end{tikzpicture}
&
\begin{tikzpicture}[baseline=0mm, scale=1]
    \coordinate(0) at(0:0); 
    \node(x) at(215:1) {}; 
    \node(y) at(0:1) {}; 
    \node(z) at(90:1) {}; 
    \draw[gray, <-] ($1*(x)$)--($-1*(x)$); 
    \draw[gray, <-] ($1*(y)$)--($-1*(y)$); 
    \draw[gray, <-] ($1*(z)$)--($-1*(z)$); 
    \node at(x) [below]{$\bm{e}_1$}; 
    \node at(y) [below]{$\bm{e}_2$}; 
    \node at(z) [right]{$\bm{e}_3$}; 
    \def\gridwidth{0.1}
    \begin{scope}[gray, line width = \gridwidth]
    \foreach \c in {0,1,2} {
    \draw ($-\c*(y)$)--($2*(x)+ -\c*(y)$); 
    \draw ($\c*(x)$)--($-2*(y)+ \c*(x)$); 
    \draw ($-\c*(y)$)--($2*(z)+ -\c*(y)$); 
    \draw ($\c*(z)$)--($-2*(y)+ \c*(z)$); 
    \draw ($-2*(y)+\c*(z)$)--($-2*(y)+ 2*(x)+ \c*(z)$); 
    \draw ($-2*(y)+\c*(x)$)--($-2*(y)+ 2*(z)+ \c*(x)$); }
    \end{scope}
    \fill[black, opacity = 0.8] (0)circle(0.5mm);
    \coordinate(1) at($1*(x) + 0*(y) + 0*(z)$); 
    \coordinate(2) at($1*(x) + -1*(y) + 0*(z)$); 
    \coordinate(3) at($0*(x) + 0*(y) + 1*(z)$); 
    \coordinate(4) at($0*(x) + -1*(y) + 1*(z)$); 
    \coordinate(5) at($1*(x) + -2*(y) + 0*(z)$); 
    \coordinate(6) at($0*(x) + -1*(y) + 0*(z)$); 
\def\segwidth{0.5mm}
\def\positivecolor{green!30!red}
\def\negativecolor{green!30!orange}
\def\segcolor{green!30!blue}
    \foreach \s\t in {(5)/(6), (6)/(4)} 
    {\draw[line width = 0.4mm, \negativecolor] \s--\t;}
    \draw[line width = \segwidth, \negativecolor] (4)--(5);
    \foreach \s\t in {(1)/(2), (1)/(3), (2)/(3)} 
    {\draw[line width = \segwidth, \positivecolor] \s--\t;}
    \foreach \s\t in {(3)/(4), (2)/(4), (2)/(5)} 
        {\draw[line width = \segwidth, \segcolor] \s--\t;} 
    \foreach \v in {(1),(2),(3),(4),(5),(6)} 
        {\fill[gray, opacity = 0.8] \v circle(1.1mm);} 
\end{tikzpicture}
\end{tabular}
    \caption{}
    \label{fig:C}
\end{figure}

Since $\Sigma(A,e)$ is ordered, there is an arrow 
\begin{equation}\label{C arrow}
    \tau := \left[ 
    \begin{smallmatrix}
        1 \\ -2 \\ 0
    \end{smallmatrix}\middle|
    \begin{smallmatrix}
        f_1 \\ -2f_1-f_2+1 \\ f_2 
    \end{smallmatrix}\middle|
    \begin{smallmatrix}
        0 \\ -1 \\ 1
    \end{smallmatrix}
    \right] \xrightarrow{(f_1,\bullet,f_2)}
    \emin = \left[ 
    \begin{smallmatrix}
        1 \\ -2 \\ 0
    \end{smallmatrix}\middle|
    \begin{smallmatrix}
        0 \\ -1 \\ 0 
    \end{smallmatrix}\middle|
    \begin{smallmatrix}
        0 \\ -1 \\ 1
    \end{smallmatrix}
    \right]
\end{equation}
ending at $\emin$ in the Hasse quiver of $(\SigmaeAe,\leq)$, where $f_1,f_2,f_1+f_2\in \{0,1,2\}$. 
If $f_1=f_2=0$, then $\left[
\begin{smallmatrix}
    1\\-2\\0  
\end{smallmatrix}\middle|
\begin{smallmatrix}
    0\\1\\0  
\end{smallmatrix}
\right]\in \SigmaeAe$, which is a contradiction to the sign-coherent property. 
If $(f_1,f_2)\in \{(1,1),(2,0),(0,2)\}$, then we have 
\begin{equation}
    \langle \begin{bsmallmatrix}
    f_1 \\ -2f_1-f_2+1 \\ f_2
\end{bsmallmatrix}, \bm{v}_{\emax}\rangle = f_1 + f_2 = 2\not\leq 1, 
\end{equation}
a contradiction.
Therefore, we have $(f_1,f_2)\in \{(0,1),(1,0)\}$. 
Applying Proposition \ref{path p1} with $(r_{12},r_{32})=(2,1)$ and $h_{12}=1$, the path $\underline{\alpha}_{(-\bm{e}_2)}$ in \eqref{path pq} is either \text{\bf (i)} or \text{\bf (iii)} as follows.  
\begin{eqnarray}
\text{\bf (i):} &&
\left(\underline{\emin} =
\left[ 
    \begin{smallmatrix}
        1 \\ -2 \\ 0 
    \end{smallmatrix}\middle|
    \begin{smallmatrix}
        0 \\ -1 \\ 1
    \end{smallmatrix}
    \right]
\underset{1}{\xrightarrow{(0,\bullet)}}
    \underline{\tau_1} \underset{2}{\xrightarrow{(\bullet,0)}} 
    \underline{\tau_2} = \underline{\sigma_-}\right), \\
\text{\bf (iii):} && 
\left(\underline{\emin} =
\left[ 
    \begin{smallmatrix}
        1 \\ -2 \\ 0 
    \end{smallmatrix}\middle|
    \begin{smallmatrix}
        0 \\ -1 \\ 1
    \end{smallmatrix}
    \right]
\underset{0}{\xrightarrow{(1,\bullet)}}
    \underline{\tau_1} \underset{1}{\xrightarrow{(\bullet,1)}} 
    \underline{\tau_2} 
    \underset{1}{\xrightarrow{(1,\bullet)}}
    \underline{\tau_3} = \underline{\sigma_-}\right), 
\end{eqnarray}
where $\tau_j$ can be computed from $\emin$ inductively. 
By Proposition \ref{path p1}, \text{\bf (i)} appears if and only if $h_{13}=0$. 
Thus, \text{\bf (iii)} appears if and only if $h_{13}=1$.

We first consider the case \text{\bf (i)}. 
Connecting with \eqref{C arrow}, we obtain the following path $p$ around $\bm{u} := \begin{bsmallmatrix}
    1 \\ -2 \\ 0 
\end{bsmallmatrix}$.  
\begin{equation}
    \underline{p}_{\bm{u}} \colon\  
    \left(
    \underline{\tau} = 
    \left[ 
    \begin{smallmatrix}
        f_1 \\ -2f_1-f_2+1 \\ f_2 
    \end{smallmatrix}\middle|
    \begin{smallmatrix}
        0 \\ -1 \\ 1
    \end{smallmatrix}
    \right] \underset{f_1}{\xrightarrow{(\bullet,f_2)}}
    \underline{\emin} = \left[ 
    \begin{smallmatrix}
        0 \\ -1 \\ 0 
    \end{smallmatrix}\middle|
    \begin{smallmatrix}
        0 \\ -1 \\ 1
    \end{smallmatrix}
    \right] \underset{0}{\xrightarrow{(1,\bullet)}}
    \underline{\tau_1} = \left[ 
    \begin{smallmatrix}
        0 \\ -1 \\ 0 
    \end{smallmatrix}\middle|
    \begin{smallmatrix}
        1 \\ -1 \\ -1
    \end{smallmatrix}
    \right]\right).  
\end{equation}
It implies that $f_2 \neq 0$ by Proposition \ref{rank2}. Thus, we have $(f_1,f_2)= (0,1)$ and $\tau$ is 
\begin{equation}
    \tau = \left[
    \begin{smallmatrix}
        1 \\ -2 \\ 0
    \end{smallmatrix}\middle|
    \begin{smallmatrix}
        0 \\ 0 \\ 1 
    \end{smallmatrix}\middle|
    \begin{smallmatrix}
        0 \\ -1 \\ 1
    \end{smallmatrix}
    \right] \in \SigmaeAe
\end{equation}
in this case. This is precisely the fan $\Sigma_{d(10),0}$ as desired.

Next, we consider the case \text{\bf (iii)}. 
By a similar discussion to the previous paragraph, 
we obtain the following path $p'$ around $\bm{u}:=\begin{bsmallmatrix}
    1 \\ -2 \\ 0
\end{bsmallmatrix}$.
\begin{equation}
    \underline{p'}_{\bm{u}} \colon\ 
    \left(\underline{\tau} =
    \left[ 
    \begin{smallmatrix}
        f_1 \\ -2f_1-f_2+1 \\ f_2 
    \end{smallmatrix}\middle|
    \begin{smallmatrix}
        0 \\ -1 \\ 1
    \end{smallmatrix}
    \right] \underset{f_1}{\xrightarrow{(\bullet,f_2)}}
    \underline{\emin} = \left[ 
    \begin{smallmatrix}
        0 \\ -1 \\ 0 
    \end{smallmatrix}\middle|
    \begin{smallmatrix}
        0 \\ -1 \\ 1
    \end{smallmatrix}
    \right] \underset{1}{\xrightarrow{(0,\bullet)}}
    \underline{\tau_1} = \left[ 
    \begin{smallmatrix}
        0 \\ -1 \\ 0 
    \end{smallmatrix}\middle|
    \begin{smallmatrix}
        1 \\ -1 \\ -1
    \end{smallmatrix}
    \right]\right). 
\end{equation}
It implies that $f_2=0$ by Proposition \ref{rank2}. 
Thus, we conclude that $(f_1,f_2) = (1,0)$ and $\tau$ is 
\begin{equation}
    \tau = \left[
    \begin{smallmatrix}
        1 \\ -2 \\ 0
    \end{smallmatrix}\middle|
    \begin{smallmatrix}
        1 \\ -1 \\ 0 
    \end{smallmatrix}\middle|
    \begin{smallmatrix}
        0 \\ -1 \\ 1
    \end{smallmatrix}
    \right] \in \SigmaeAe. 
\end{equation}
Therefore, we obtain the fan $\Sigma_{d(10),1}$ as desired.


\subsection{Case $d(11)$} \label{case:11}
Assume that $d_{+-+}(A,e) = d(11) = ((1,2,1),(1,2,0))$. 
By Proposition \ref{min_max}, we have 
\begin{equation}\label{D1 mm}
    \emax = 
    \left[
     \begin{smallmatrix}
        1 \\ 0 \\ 0 
    \end{smallmatrix}\middle|
    \begin{smallmatrix}
        1 \\ -1 \\ 0 
    \end{smallmatrix}\middle| 
    \begin{smallmatrix}
        0 \\ 0 \\ 1 
    \end{smallmatrix}
    \right], \
    \emin = \left[
     \begin{smallmatrix}
        1 \\ -2 \\ 0 
    \end{smallmatrix}\middle|
    \begin{smallmatrix}
        0 \\ -1 \\ 0 
    \end{smallmatrix}\middle| 
    \begin{smallmatrix}
        0 \\ -2 \\ 1 
    \end{smallmatrix}
    \right], \
   \bm{v}_{\emax} = \begin{bsmallmatrix}
        1 \\ 0 \\ 1
    \end{bsmallmatrix} \ \text{and} \ 
   \bm{v}_{\emin} = \begin{bsmallmatrix}
        -1 \\ -1 \\ -1
    \end{bsmallmatrix}. 
\end{equation}
For $X=\{\emax,\emin\}$, the set $\PeX$ is given by the following system of inequalities in $\mathbb{R}_{+-+}^3$: 
\begin{equation}
    x_1 + x_3 \leq 1,\ -x_1-x_2-x_3 \leq 1,\ -x_1\le 0,\ x_2\le 0,\ -x_3\le 0. 
\end{equation}
Then, we have 
\begin{equation}
    \PeXZ = \left\{
    \begin{bsmallmatrix}
        1 \\ 0 \\ 0
    \end{bsmallmatrix}, 
    \begin{bsmallmatrix}
        1 \\ -1 \\ 0
    \end{bsmallmatrix},
    \begin{bsmallmatrix}
        1 \\ -2 \\ 0
    \end{bsmallmatrix}, 
    \begin{bsmallmatrix}
        0 \\ -1 \\ 0
    \end{bsmallmatrix}, 
    \begin{bsmallmatrix}
        0 \\ -2 \\ 1
    \end{bsmallmatrix}, 
    \begin{bsmallmatrix}
        0 \\ -1 \\ 1
    \end{bsmallmatrix},
    \begin{bsmallmatrix}
        0 \\ 0 \\ 1
    \end{bsmallmatrix} 
    \right\}.   
\end{equation}
Since $(l_{32},r_{32}) = (1,2)$, we have $\begin{bsmallmatrix}
        0 \\ -1 \\ 1
\end{bsmallmatrix}\in \SigmaeAe$ by Proposition \ref{coord_plane}.
Thus, we obtain $\PeXZ \subset \PeAe$. 
On the other hand, we have $\mathcal{L}(X) = \{
\begin{bsmallmatrix}
        1 \\ 0 \\ 1
\end{bsmallmatrix}, 
\begin{bsmallmatrix}
        -1 \\ -1 \\ -1
\end{bsmallmatrix}
\} \subseteq \mathcal{E}_1$ by a direct calculation.
Then, we deduce that $\PeX = \PeAe$ from Proposition \ref{intermediate}. 
In this case, there are only $3$ possibility of $\SigmaeAe$ described in Figure \ref{fig:D1}, where the left most one is the fan $\Sigma_{d(11)}$. 

\begin{figure}[ht]
    \begin{tabular}{ccccc}
    $\Sigma_{d(11)}$ & (a)& (b) \\ 
\begin{tikzpicture}[baseline=0mm, scale=1]
    \coordinate(0) at(0:0); 
    \node(x) at(215:1) {}; 
    \node(y) at(0:1) {}; 
    \node(z) at(90:1) {}; 
    \draw[gray, <-] ($1*(x)$)--($-1*(x)$); 
    \draw[gray, <-] ($1*(y)$)--($-1*(y)$); 
    \draw[gray, <-] ($1*(z)$)--($-1*(z)$); 
    \node at(x) [below]{$\bm{e}_1$}; 
    \node at(y) [below]{$\bm{e}_2$}; 
    \node at(z) [right]{$\bm{e}_3$}; 
    \def\gridwidth{0.1}
    \begin{scope}[gray, line width = \gridwidth]
    \foreach \c in {0,1,2} {
    \draw ($-\c*(y)$)--($2*(x)+ -\c*(y)$); 
    \draw ($\c*(x)$)--($-2*(y)+ \c*(x)$); 
    \draw ($-\c*(y)$)--($2*(z)+ -\c*(y)$); 
    \draw ($\c*(z)$)--($-2*(y)+ \c*(z)$); 
    \draw ($-2*(y)+\c*(z)$)--($-2*(y)+ 2*(x)+ \c*(z)$); 
    \draw ($-2*(y)+\c*(x)$)--($-2*(y)+ 2*(z)+ \c*(x)$); }
    \end{scope}
    \fill[black, opacity = 0.8] (0)circle(0.5mm);
    \coordinate(1) at($1*(x) + 0*(y) + 0*(z)$); 
    \coordinate(2) at($1*(x) + -1*(y) + 0*(z)$); 
    \coordinate(3) at($0*(x) + 0*(y) + 1*(z)$); 
    \coordinate(4) at($0*(x) + -1*(y) + 1*(z)$); 
    \coordinate(5) at($0*(x) + -2*(y) + 1*(z)$); 
    \coordinate(6) at($1*(x) + -2*(y) + 0*(z)$); 
    \coordinate(7) at($0*(x) + -1*(y) + 0*(z)$); 
\def\segwidth{0.5mm}
\def\positivecolor{green!30!red}
\def\negativecolor{green!30!orange}
\def\segcolor{green!30!blue}
    \foreach \s\t in {(6)/(7), (7)/(5)} 
    {\draw[line width = 0.4mm, \negativecolor] \s--\t;}
    \draw[line width = \segwidth, \negativecolor] (5)--(6);
    \foreach \s\t in {(1)/(2), (1)/(3), (2)/(3)} 
    {\draw[line width = \segwidth, \positivecolor] \s--\t;}
    \foreach \s\t in {(3)/(4), (4)/(5), (2)/(4), (2)/(5), (2)/(6)} 
        {\draw[line width = \segwidth, \segcolor] \s--\t;} 
    \foreach \v in {(1),(2),(3),(4),(5),(6),(7)} 
        {\fill[gray, opacity = 0.8] \v circle(1.1mm);} 
\end{tikzpicture}
&
\begin{tikzpicture}[baseline=0mm, scale=1]
    \coordinate(0) at(0:0); 
    \node(x) at(215:1) {}; 
    \node(y) at(0:1) {}; 
    \node(z) at(90:1) {}; 
    \draw[gray, <-] ($1*(x)$)--($-1*(x)$); 
    \draw[gray, <-] ($1*(y)$)--($-1*(y)$); 
    \draw[gray, <-] ($1*(z)$)--($-1*(z)$); 
    \node at(x) [below]{$\bm{e}_1$}; 
    \node at(y) [below]{$\bm{e}_2$}; 
    \node at(z) [right]{$\bm{e}_3$}; 
    \def\gridwidth{0.1}
    \begin{scope}[gray, line width = \gridwidth]
    \foreach \c in {0,1,2} {
    \draw ($-\c*(y)$)--($2*(x)+ -\c*(y)$); 
    \draw ($\c*(x)$)--($-2*(y)+ \c*(x)$); 
    \draw ($-\c*(y)$)--($2*(z)+ -\c*(y)$); 
    \draw ($\c*(z)$)--($-2*(y)+ \c*(z)$); 
    \draw ($-2*(y)+\c*(z)$)--($-2*(y)+ 2*(x)+ \c*(z)$); 
    \draw ($-2*(y)+\c*(x)$)--($-2*(y)+ 2*(z)+ \c*(x)$); }
    \end{scope}
    \fill[black, opacity = 0.8] (0)circle(0.5mm);
    \coordinate(1) at($1*(x) + 0*(y) + 0*(z)$); 
    \coordinate(2) at($1*(x) + -1*(y) + 0*(z)$); 
    \coordinate(3) at($0*(x) + 0*(y) + 1*(z)$); 
    \coordinate(4) at($0*(x) + -1*(y) + 1*(z)$); 
    \coordinate(5) at($0*(x) + -2*(y) + 1*(z)$); 
    \coordinate(6) at($1*(x) + -2*(y) + 0*(z)$); 
    \coordinate(7) at($0*(x) + -1*(y) + 0*(z)$); 
\def\segwidth{0.5mm}
\def\positivecolor{green!30!red}
\def\negativecolor{green!30!orange}
\def\segcolor{green!30!blue}
    \foreach \s\t in {(6)/(7), (7)/(5)} 
    {\draw[line width = 0.4mm, \negativecolor] \s--\t;}
    \draw[line width = \segwidth, \negativecolor] (5)--(6);
    \foreach \s\t in {(1)/(2), (1)/(3), (2)/(3)} 
    {\draw[line width = \segwidth, \positivecolor] \s--\t;}
    \foreach \s\t in {(3)/(4), (4)/(5), (2)/(4), (4)/(6), (2)/(6)} 
        {\draw[line width = \segwidth, \segcolor] \s--\t;} 
    \foreach \v in {(1),(2),(3),(4),(5),(6),(7)} 
        {\fill[gray, opacity = 0.8] \v circle(1.1mm);} 
\end{tikzpicture}
& 
\begin{tikzpicture}[baseline=0mm, scale=1]
    \coordinate(0) at(0:0); 
    \node(x) at(215:1) {}; 
    \node(y) at(0:1) {}; 
    \node(z) at(90:1) {}; 
    \draw[gray, <-] ($1*(x)$)--($-1*(x)$); 
    \draw[gray, <-] ($1*(y)$)--($-1*(y)$); 
    \draw[gray, <-] ($1*(z)$)--($-1*(z)$); 
    \node at(x) [below]{$\bm{e}_1$}; 
    \node at(y) [below]{$\bm{e}_2$}; 
    \node at(z) [right]{$\bm{e}_3$}; 
    \def\gridwidth{0.1}
    \begin{scope}[gray, line width = \gridwidth]
    \foreach \c in {0,1,2} {
    \draw ($-\c*(y)$)--($2*(x)+ -\c*(y)$); 
    \draw ($\c*(x)$)--($-2*(y)+ \c*(x)$); 
    \draw ($-\c*(y)$)--($2*(z)+ -\c*(y)$); 
    \draw ($\c*(z)$)--($-2*(y)+ \c*(z)$); 
    \draw ($-2*(y)+\c*(z)$)--($-2*(y)+ 2*(x)+ \c*(z)$); 
    \draw ($-2*(y)+\c*(x)$)--($-2*(y)+ 2*(z)+ \c*(x)$); }
    \end{scope}
    \fill[black, opacity = 0.8] (0)circle(0.5mm);
    \coordinate(1) at($1*(x) + 0*(y) + 0*(z)$); 
    \coordinate(2) at($1*(x) + -1*(y) + 0*(z)$); 
    \coordinate(3) at($0*(x) + 0*(y) + 1*(z)$); 
    \coordinate(4) at($0*(x) + -1*(y) + 1*(z)$); 
    \coordinate(5) at($0*(x) + -2*(y) + 1*(z)$); 
    \coordinate(6) at($1*(x) + -2*(y) + 0*(z)$); 
    \coordinate(7) at($0*(x) + -1*(y) + 0*(z)$); 
\def\segwidth{0.5mm}
\def\positivecolor{green!30!red}
\def\negativecolor{green!30!orange}
\def\segcolor{green!30!blue}
    \foreach \s\t in {(6)/(7), (7)/(5)} 
    {\draw[line width = 0.4mm, \negativecolor] \s--\t;}
    \draw[line width = \segwidth, \negativecolor] (5)--(6);
    \foreach \s\t in {(1)/(2), (1)/(3), (2)/(3)} 
    {\draw[line width = \segwidth, \positivecolor] \s--\t;}
    \foreach \s\t in {(3)/(4), (4)/(5), (3)/(6), (4)/(6), (2)/(6)} 
        {\draw[line width = \segwidth, \segcolor] \s--\t;} 
    \foreach \v in {(1),(2),(3),(4),(5),(6),(7)} 
        {\fill[gray, opacity = 0.8] \v circle(1.1mm);} 
\end{tikzpicture}
    \end{tabular}
    \caption{}
    \label{fig:D1}
\end{figure}

To prove the assertion, it suffices to show that $\tau\in \SigmaeAe$, where 
\begin{equation}
    \tau = \left[ 
    \begin{smallmatrix}
        1 \\ -2 \\ 0
    \end{smallmatrix}\middle|
    \begin{smallmatrix}
        1 \\ -1 \\ 0 
    \end{smallmatrix}\middle|
    \begin{smallmatrix}
        0 \\ -1 \\ 2
    \end{smallmatrix}
    \right]. 
\end{equation}

We consider an arrow 
\begin{equation}\label{D1 arrow}
    \tau' := \left[ 
    \begin{smallmatrix}
        1 \\ -2 \\ 0
    \end{smallmatrix}\middle|
    \begin{smallmatrix}
        f_1 \\ -2(f_1+f_2)+1 \\ f_2 
    \end{smallmatrix}\middle|
    \begin{smallmatrix}
        0 \\ -2 \\ 1
    \end{smallmatrix}
    \right] \xrightarrow{(f_1,\bullet,f_2)}
    \emin = \left[ 
    \begin{smallmatrix}
        1 \\ -2 \\ 0
    \end{smallmatrix}\middle|
    \begin{smallmatrix}
        0 \\ -1 \\ 0 
    \end{smallmatrix}\middle|
    \begin{smallmatrix}
        0 \\ -2 \\ 1
    \end{smallmatrix}
    \right]
\end{equation}
ending at $\emin$, where $f_1,f_2,f_1+f_2\in \{0,1,2\}$. 
From now on, we show that $(f_1,f_2)=(1,0)$ by excluding all other cases. 
If $f_1=f_2=0$, then we have $\left[\begin{smallmatrix}
    1 \\ -2 \\ 0
\end{smallmatrix}\middle| 
\begin{smallmatrix}
    0 \\ 1 \\ 0
\end{smallmatrix}\right]\in \SigmaeAe$, which is a contradiction to the sign-coherent property. 
If $(f_1,f_{2})\in \{(1,1),(2,0),(0,2)\}$, then we have 
\begin{equation}
    \langle 
    \begin{bsmallmatrix}
        f_1 \\ -3 \\ f_2 
    \end{bsmallmatrix}, 
   \bm{v}_{\emax}
    \rangle = 
    f_1 + f_2 = 2 \not\leq 1, 
\end{equation}
a contradiction.
Thus, we have $(f_1,f_2)\in \{(1,0),(0,1)\}$. 
In addition, by applying Proposition \ref{path p3} with $(r_{12},r_{32})=(2,2)$ and $h_{32}=0$, 
the path $\underline{\beta}_{(-\bm{e}_2)}$ in \eqref{path pq} must be 
\begin{equation} 
    \text{\bf (v'):} \quad 
    \left(
    \underline{\emin} = 
    \left[ 
    \begin{smallmatrix}
        1 \\ -2 \\ 0 
    \end{smallmatrix}\middle|
    \begin{smallmatrix}
        0 \\ -2 \\ 1
    \end{smallmatrix}
    \right]
    \underset{1}{\xrightarrow{(\bullet,1)}}
    \underline{\tau_1'} \underset{0}{\xrightarrow{(2,\bullet)}}
    \underline{\tau_2'} \underset{1}{\xrightarrow{(\bullet,1)}}  
    \underline{\tau_3'} \underset{0}{\xrightarrow{(2,\bullet)}}  
    \underline{\tau_4'} 
    = \underline{\sigma_-}
    \right)
\end{equation}
where $\tau_j'$ can be computed from $\emin$ inductively. Connecting with \eqref{D1 arrow}, we obtain the following path $p$ around 
$\bm{u}:=\begin{bsmallmatrix}
    0 \\ -2 \\ 1
\end{bsmallmatrix}$. 
\begin{equation}
    \underline{p}_{\bm{u}} \colon \  
    \left( \underline{\tau}' = 
    \left[ 
    \begin{smallmatrix}
        1 \\ -2 \\ 0
    \end{smallmatrix}\middle|
    \begin{smallmatrix}
        f_1 \\ -2(f_1+f_2) +1 \\ f_2 
    \end{smallmatrix}
    \right] \underset{f_2}{\xrightarrow{(f_1,\bullet)}}
    \underline{\emin} = \left[ 
    \begin{smallmatrix}
        1 \\ -2 \\ 0 
    \end{smallmatrix}\middle|
    \begin{smallmatrix}
        0 \\ -1 \\ 0
    \end{smallmatrix}
    \right] \underset{1}{\xrightarrow{(\bullet,1)}}
    \underline{\tau_1'} = \left[ 
    \begin{smallmatrix}
        -1 \\ -1 \\ 1 
    \end{smallmatrix}\middle|
    \begin{smallmatrix}
        0 \\ -1 \\ 0
    \end{smallmatrix}
    \right]\right).  
\end{equation}
It implies that $f_1\neq 0$ by Proposition \ref{rank2}. Therefore, we conclude that $(f_1,f_2)=(1,0)$, and $\tau'$ is 
\begin{equation}
    \tau' = \tau = \left[ 
    \begin{smallmatrix}
        1 \\ -2 \\ 0
    \end{smallmatrix}\middle|
    \begin{smallmatrix}
        1 \\ -1 \\ 0 
    \end{smallmatrix}\middle|
    \begin{smallmatrix}
        0 \\ -1 \\ 2
    \end{smallmatrix}
    \right]\in \SigmaeAe 
\end{equation}
as desired.

\subsection{Case $d(12)$} \label{case:12}
Assume that $d_{+-+}(A,e) = d(12) = ((1,1,1),(1,2,1))$. 
By Proposition \ref{min_max}, we have 
\begin{equation}
    \emax = 
    \left[
     \begin{smallmatrix}
        1 \\ 0 \\ 0 
    \end{smallmatrix}\middle|
    \begin{smallmatrix}
        1 \\ -1 \\ 1 
    \end{smallmatrix}\middle| 
    \begin{smallmatrix}
        0 \\ 0 \\ 1 
    \end{smallmatrix}
    \right], \
    \emin = \left[
     \begin{smallmatrix}
        1 \\ -1 \\ 0 
    \end{smallmatrix}\middle|
    \begin{smallmatrix}
        0 \\ -1 \\ 0 
    \end{smallmatrix}\middle| 
    \begin{smallmatrix}
        0 \\ -2 \\ 1 
    \end{smallmatrix}
    \right], \ 
   \bm{v}_{\emax} = \begin{bsmallmatrix}
    1 \\ 1 \\ 1    
    \end{bsmallmatrix}\ \ \text{and} \ \
   \bm{v}_{\emin} = 
    \begin{bsmallmatrix}
        0\\ -1 \\ -1
    \end{bsmallmatrix}. 
\end{equation}

We claim that $\kappa,\varphi \in \SigmaeAe$, where 
\begin{equation}\label{D2 kappaphi}
    \kappa = 
    \left[
     \begin{smallmatrix}
        0 \\ -1 \\ 1 
    \end{smallmatrix}\middle|
    \begin{smallmatrix}
        1 \\ -1 \\ 1 
    \end{smallmatrix}\middle|
    \begin{smallmatrix}
        2 \\ -2 \\ 1 
    \end{smallmatrix}
    \right], \varphi = \left[ 
    \begin{smallmatrix}
        1 \\ -1 \\ 0 
    \end{smallmatrix}\middle|
    \begin{smallmatrix}
        1 \\ 0 \\ 0 
    \end{smallmatrix}\middle|
    \begin{smallmatrix}
        2 \\ -2 \\ 1
    \end{smallmatrix}
    \right] \ \ \text{with} \ \ 
   \bm{v}_{\kappa} = \begin{bsmallmatrix}
        0 \\ 0 \\ 1
    \end{bsmallmatrix}, 
   \bm{v}_{\varphi} = \begin{bsmallmatrix}
        1 \\ 0 \\ -1
    \end{bsmallmatrix}. 
\end{equation}
After that, we get the assertion by the following reason: 
Let $X:=\{\emax,\emin,\kappa,\varphi\}$. 
We recall from the definition that the set $\PeX$ is bounded by the following hyperplanes. 
\begin{equation} \label{hyperplanes} 
\langle \bm{x}, \bm{v}_{\emax} \rangle=  
\langle \bm{x}, \bm{v}_{\emin} \rangle=  
\langle \bm{x}, \bm{v}_{\kappa} \rangle= 
\langle \bm{x}, \bm{v}_{\varphi} \rangle= 1 \ \text{and}\ 
\langle \bm{x}, -\bm{e}_1 \rangle= \langle \bm{x}, \bm{e}_2 \rangle= \langle \bm{x}, -\bm{e}_3 \rangle= 0. 
\end{equation}
Or equivalently, it is given by the following system of inequalities:
\begin{equation}
    x_1 + x_2 + x_3 \leq 1,\ -x_2 - x_3 \leq 1, \ x_3 \leq 1,\ x_1-x_3 \leq 1,\ -x_1\le 0,\ x_2\le 0,\ -x_3 \le 0.
\end{equation}
From a direct calculation, we have 
\begin{equation}
    \PeXZ = \left\{
    \begin{bsmallmatrix}
        1 \\ 0 \\ 0 
    \end{bsmallmatrix}, 
    \begin{bsmallmatrix}
        1 \\ -1 \\ 0
    \end{bsmallmatrix}, 
    \begin{bsmallmatrix}
        0 \\ -1 \\ 0
    \end{bsmallmatrix}, 
    \begin{bsmallmatrix}
        0 \\ -2 \\ 1
    \end{bsmallmatrix}, 
    \begin{bsmallmatrix}
        0 \\ -1 \\ 1
    \end{bsmallmatrix}, 
    \begin{bsmallmatrix}
        0 \\ 0 \\ 1
    \end{bsmallmatrix}, 
    \begin{bsmallmatrix}
        1 \\ -1 \\ 1
    \end{bsmallmatrix}, 
    \begin{bsmallmatrix}
        1 \\ -2 \\ 1
    \end{bsmallmatrix}, 
    \begin{bsmallmatrix}
        2 \\ -2 \\ 1 
    \end{bsmallmatrix} 
    \right\}.  
\end{equation}
Now, we show that every vertex of $\PeX$ lies in $\PeX_{\mathbb{Z}}$. 
In fact, for each vertex $\bm{x}$ of $\PeX$, 
we take $3$ hyperplanes 
$\mathbb{H}_1,\mathbb{H}_2, \mathbb{H}_3$ from \eqref{hyperplanes} such that 
$\mathbb{H}_{1}\cap \mathbb{H}_{2}\cap \mathbb{H}_{3}=\{\bm{x}\}$. 
Letting $\bm{u}_1,\bm{u}_2,\bm{u}_3$ be their normal vectors, 
they are such that 
\[
\{\bm{u}_1,\bm{u}_2,\bm{u}_3\}\subseteq \{\bm{v}_{\emax},\bm{v}_{\emin},\bm{v}_{\kappa},\bm{v}_{\varphi},-\bm{e}_1,\bm{e}_2,-\bm{e}_3\}.
\]
and $\det [\bm{u}_1\,\bm{u}_2\,\bm{u}_3]\ne 0$. 
We can easily check that 
$\det [\bm{u}_1\,\bm{u}_2\,\bm{u}_3]\in \{\pm 1\}$ and hence $$\bm{x}\in \PeXZ$$ for all triples but 
$\{\bm{v}_{\emax}, \bm{v}_{\varphi},\bm{e}_2\}$.  
Indeed, $\det[\bm{v}_{\emax}, \bm{v}_{\varphi},\bm{e}_2] = 2$ holds. 
However, we also get the claim for this triple 
since $\bm{x}=\bm{e}_1 \in \PeXZ$ in this case. 
From the above discussion, we obtain $\PeX = \Conv_0(\PeXZ)$.

On the other hand, since $\begin{bsmallmatrix}
        2 \\ -2 \\ 1 
    \end{bsmallmatrix}\in \Conv_0(\{\emin,\kappa\})\subseteq \PeAe$, 
we have $\PeXZ\subseteq \PeAe$. 
Consequently, we have 
\[
\PeX = \Conv_{0}(\PeXZ) \subseteq \PeAe \subseteq \PeX,  
\]
showing that $\PeAe = \PeX$. 
From this description, maximal cones in $X$ give rise to a unique convex non-singular fan $\Sigma_{d(12)}$, which coincides with $\SigmaeAe$ as desired. 
See Figure \ref{fig:D2}. 

\begin{figure}[ht]
    \begin{tabular}{cccc}
    $\Sigma_{d(12)}$ \\ 
\begin{tikzpicture}[baseline=0mm, scale=1]
    \coordinate(0) at(0:0); 
    \node(x) at(215:1) {}; 
    \node(y) at(0:1) {}; 
    \node(z) at(90:1) {}; 
    \draw[gray, <-] ($1*(x)$)--($-1*(x)$); 
    \draw[gray, <-] ($1*(y)$)--($-1*(y)$); 
    \draw[gray, <-] ($1*(z)$)--($-1*(z)$); 
    \node at(x) [below]{$\bm{e}_1$}; 
    \node at(y) [below]{$\bm{e}_2$}; 
    \node at(z) [right]{$\bm{e}_3$}; 
    \def\gridwidth{0.1}
    \begin{scope}[gray, line width = \gridwidth]
    \foreach \c in {0,1,2} {
    \draw ($-\c*(y)$)--($2*(x)+ -\c*(y)$); 
    \draw ($\c*(x)$)--($-2*(y)+ \c*(x)$); 
    \draw ($-\c*(y)$)--($2*(z)+ -\c*(y)$); 
    \draw ($\c*(z)$)--($-2*(y)+ \c*(z)$); 
    \draw ($-2*(y)+\c*(z)$)--($-2*(y)+ 2*(x)+ \c*(z)$); 
    \draw ($-2*(y)+\c*(x)$)--($-2*(y)+ 2*(z)+ \c*(x)$); }
    \end{scope}
    \fill[black, opacity = 0.8] (0)circle(0.5mm);
    \coordinate(1) at($1*(x) + 0*(y) + 0*(z)$); 
    \coordinate(2) at($1*(x) + -1*(y) + 1*(z)$); 
    \coordinate(3) at($0*(x) + 0*(y) + 1*(z)$); 
    \coordinate(4) at($2*(x) + -2*(y) + 1*(z)$); 
    \coordinate(5) at($0*(x) + -1*(y) + 1*(z)$); 
    \coordinate(6) at($1*(x) + -1*(y) + 0*(z)$); 
    \coordinate(7) at($1*(x) + -2*(y) + 1*(z)$); 
    \coordinate(8) at($0*(x) + -2*(y) + 1*(z)$); 
    \coordinate(9) at($0*(x) + -1*(y) + 0*(z)$); 
\def\segwidth{0.5mm}
\def\positivecolor{green!30!red}
\def\negativecolor{green!30!orange}
\def\segcolor{green!30!blue}
    \foreach \s\t in {(8)/(9), (8)/(6), (6)/(9)} 
    {\draw[line width = 0.4mm, \negativecolor] \s--\t;}
    \foreach \s\t in {(7)/(6)} 
    {\draw[line width = 0.4mm, \segcolor] \s--\t;}
    \foreach \s\t in {(1)/(2), (1)/(3), (2)/(3)} 
    {\draw[line width = \segwidth, \positivecolor] \s--\t;}
    \foreach \s\t in {(4)/(2), (4)/(1), (2)/(5), (4)/(5), (4)/(7), (7)/(5), (4)/(6), (7)/(8), (1)/(6), (3)/(5), (5)/(8)} 
        {\draw[line width = \segwidth, \segcolor] \s--\t;} 
    \foreach \v in {(1),(2),(3),(4),(5),(6),(7),(8),(9)} 
        {\fill[gray, opacity = 0.8] \v circle(1.1mm);} 
\end{tikzpicture}
        \end{tabular}
    \caption{}
    \label{fig:D2}
\end{figure}

From now on, we show that $\kappa,\varphi$ lies in $\SigmaeAe$, where $\kappa,\varphi$ are cones in \eqref{D2 kappaphi}. 
Firstly, we study maximal cones of $\SigmaeAe$ adjacent to the maximum and minimum. 
We consider an arrow 
\begin{equation}\label{D2 arrow1}
    \tau := \left[ 
    \begin{smallmatrix}
        1 \\ -1 \\ 0
    \end{smallmatrix}\middle|
    \begin{smallmatrix}
        f_1 \\ -f_1-2f_2+1 \\ f_2 
    \end{smallmatrix}\middle|
    \begin{smallmatrix}
        0 \\ -2 \\ 1
    \end{smallmatrix}
    \right] \xrightarrow{(f_1,\bullet,f_2)}
    \emin = \left[ 
    \begin{smallmatrix}
        1 \\ -2 \\ 0
    \end{smallmatrix}\middle|
    \begin{smallmatrix}
        0 \\ -1 \\ 0 
    \end{smallmatrix}\middle|
    \begin{smallmatrix}
        0 \\ -2 \\ 1
    \end{smallmatrix}
    \right]
\end{equation}
ending at $\emin$, 
where $f_1,f_2,f_1+f_2\in \{0,1,2\}$. 
By Proposition \ref{path p1} with $(r_{12},r_{32})=(1,2)$ and $h_{12}=1$, 
the path $\underline{\alpha}_{(-\bm{e}_2)}$ in \eqref{path pq} is of the form 
\begin{equation} \label{D2 p1}
    \text{\bf (i):} \quad 
    \left(\underline{\emin} = 
    \left[ 
    \begin{smallmatrix}
        1 \\ -1 \\ 0 
    \end{smallmatrix}\middle|
    \begin{smallmatrix}
        0 \\ -2 \\ 1
    \end{smallmatrix}
    \right]
    \underset{2}{\xrightarrow{(0,\bullet)}}
    \underline{\tau_1} =
    \left[ 
    \begin{smallmatrix}
        1 \\ -1 \\ 0 
    \end{smallmatrix}\middle|
    \begin{smallmatrix}
        0 \\ 0 \\ -1
    \end{smallmatrix}
    \right]
    \underset{1}{\xrightarrow{(\bullet,0)}}
    \underline{\tau_2} = \underline{\sigma_-}\right),
\end{equation}
where $\tau_j$ can be computed from $\emin$ inductively. 
Connecting with \eqref{D2 arrow1}, we obtain the following path $p$ around $\bm{u}:= \begin{bsmallmatrix}
    1 \\ -1 \\ 0
\end{bsmallmatrix}$.
\begin{equation} \label{D2 path1}
    \underline{p}_{\bm{u}}\colon \  
    \left(
    \underline{\tau} = 
    \left[ 
    \begin{smallmatrix}
        f_1 \\ -f_1-2f_2+1 \\ f_2 
    \end{smallmatrix}\middle|
    \begin{smallmatrix}
        0 \\ -2 \\ 1
    \end{smallmatrix}
    \right] \underset{f_1}{\xrightarrow{(f_2,\bullet)}}
    \underline{\emin} = \left[ 
    \begin{smallmatrix}
        0 \\ -1 \\ 0 
    \end{smallmatrix}\middle|
    \begin{smallmatrix}
        0 \\ -2 \\ 1
    \end{smallmatrix}
    \right] \underset{0}{\xrightarrow{(\bullet,2)}}
    \underline{\tau_1} = \left[ 
    \begin{smallmatrix}
        0 \\ -1 \\ 0 
    \end{smallmatrix}\middle|
    \begin{smallmatrix}
        0 \\ 0 \\ -1
    \end{smallmatrix}
    \right]\right).  
\end{equation}
It implies that $f_2=1$ and $f_1\in \{0,1\}$ by Proposition \ref{rank2}.
If $f_1=0$, then we have $\tau = \left[
\begin{smallmatrix}
    1 \\ -1 \\ 0 
\end{smallmatrix}\middle| 
\begin{smallmatrix}
    0 \\ -1 \\ 1 
\end{smallmatrix}\middle| 
\begin{smallmatrix}
    0 \\ -2 \\ 1 
\end{smallmatrix}
\right]$ and $\bm{v}_{\tau}= \begin{bsmallmatrix}
    1 \\ 0 \\ 1
\end{bsmallmatrix}$. In this case, we have $\langle\bm{v}_{\tau}, \begin{bsmallmatrix}
    1 \\ -1 \\ 1
\end{bsmallmatrix}\rangle = 2 \not\leq 1$, which is a contradiction.
Thus, we have $(f_1,f_2)=(1,1)$ and 
\begin{equation}\label{D2 tau}
    \tau = \left[ 
    \begin{smallmatrix}
        1 \\ -1 \\ 0 
    \end{smallmatrix}\middle|
    \begin{smallmatrix}
        1 \\ -2 \\ 1
    \end{smallmatrix}\middle|
    \begin{smallmatrix}
        0 \\ -2 \\ 1
    \end{smallmatrix}
    \right] \in \SigmaeAe.
\end{equation}

Secondly, we focus on the path $\underline{\beta}_{(-\bm{e}_2)}$ in \eqref{path pq}. By Proposition \ref{path p3} with $(r_{12},r_{32})=(1,2)$ and $h_{32}=1$, 
it is either {\upshape {\bf (i')}} or {\upshape {\bf (iii')}} given as follows: 
\begin{align}
\text{\bf (i'):} \quad &
\left(\underline{\emin} =
\left[ 
    \begin{smallmatrix}
        1 \\ -1 \\ 0 
    \end{smallmatrix}\middle|
    \begin{smallmatrix}
        0 \\ -2 \\ 1
    \end{smallmatrix}
    \right]
\underset{1}{\xrightarrow{(\bullet,0)}}
    \underline{\tau_1'} = 
    \left[ 
    \begin{smallmatrix}
        -1 \\ 0 \\ 0 
    \end{smallmatrix}\middle|
    \begin{smallmatrix}
        0 \\ -2 \\ 1
    \end{smallmatrix}
    \right]
\underset{2}{\xrightarrow{(0,\bullet)}} 
    \underline{\tau_2'} = \underline{\sigma_-}\right), \label{D2 p3} 
\\    
\text{\bf (iii'):} \quad &
    \left(\underline{\emin} =
\left[ 
    \begin{smallmatrix}
        1 \\ -1 \\ 0 
    \end{smallmatrix}\middle|
    \begin{smallmatrix}
        0 \\ -2 \\ 1
    \end{smallmatrix}
    \right]
\underset{0}{\xrightarrow{(\bullet,1)}}
    \underline{\tau_1'} \underset{1}{\xrightarrow{(1,\bullet)}} 
    \underline{\tau_2'} = 
    \left[ 
    \begin{smallmatrix}
        -1 \\ -1 \\ 1 
    \end{smallmatrix}\middle|
    \begin{smallmatrix}
        -1 \\ 0 \\ 0
    \end{smallmatrix}
    \right]
    \underset{1}{\xrightarrow{(\bullet,1)}}
    \underline{\tau_3'} = \underline{\sigma_-}\right),\nonumber
\end{align}
where $\tau_j'$ can be computed from $\emin$ inductively. We show that the path {\upshape {\bf (i')}} only appears. 
In fact, the maximal cone $\tau$ in \eqref{D2 tau} provides the following path $q$ around $\bm{v}:=\begin{bsmallmatrix}
    0 \\ -2 \\ 1
\end{bsmallmatrix}$.  
\begin{equation}\label{D2 path2}
    \underline{q}_{\bm{v}} \colon\  
    \left(\underline{\tau} = \left[ 
    \begin{smallmatrix}
        1 \\ -1 \\ 0 
    \end{smallmatrix}\middle|
    \begin{smallmatrix}
        1 \\ -2 \\ 1
    \end{smallmatrix}
    \right] \underset{1}{\xrightarrow{(1,\bullet)}}
    \underline{\emin} = 
    \left[ 
    \begin{smallmatrix}
        1 \\ -1 \\ 0 
    \end{smallmatrix}\middle|
    \begin{smallmatrix}
        0 \\ -1 \\ 0
    \end{smallmatrix}
    \right] \underset{c_2}{\xrightarrow{(\bullet,c_1)}}
    \underline{\eta} := \left[ 
    \begin{smallmatrix}
        -1 \\ -c_1-2c_2+1 \\ c_2 
    \end{smallmatrix}\middle|
    \begin{smallmatrix}
        0 \\ -1 \\ 0
    \end{smallmatrix}
    \right]\right), 
\end{equation}
where $\eta$ is a maximal cone obtained from $\sigma_{\emin}$ by exchanging $\left[\begin{smallmatrix}
        1 \\ -1 \\ 0 
\end{smallmatrix}\right]$ with $c_1,c_2,c_1+c_2\in \{0,1,2\}$.  
That is, $\eta$ coincides with the cone $\tau_1'$ lying in the path $\underline{\beta}_{(-\bm{e}_2)}$. 
Therefore, the path \text{\bf (i')} (resp., \text{\bf (iii')}) appears if and only if $(c_1,c_2)=(1,0)$ (resp., $(c_1,c_2)=(0,1)$). 
However, by Proposition \ref{rank2}, we have $c_{1}>0$ and hence we have {\upshape {\bf (i')}}. 
Then, by Proposition \ref{rank2}, 
we can extend the above path \eqref{D2 path2} to the following path $\tilde{q}$ around $\bm{v}$.
\begin{equation}
    \underline{\tilde{q}}_{\bm{v}} \colon \ 
    \left(\underline{\xi} := \left[ 
    \begin{smallmatrix}
        0 \\ -2b-1 \\ b+1 
    \end{smallmatrix}\middle|
    \begin{smallmatrix}
        1 \\ -2 \\ 1
    \end{smallmatrix}
    \right] \underset{b}{\xrightarrow{(\bullet,1)}}
    \underline{\tau} = \left[ 
    \begin{smallmatrix}
        1 \\ -1 \\ 0 
    \end{smallmatrix}\middle|
    \begin{smallmatrix}
        1 \\ -2 \\ 1
    \end{smallmatrix}
    \right] \underset{1}{\xrightarrow{(1,\bullet)}}
    \underline{\emin} = 
    \left[ 
    \begin{smallmatrix}
        1 \\ -1 \\ 0 
    \end{smallmatrix}\middle|
    \begin{smallmatrix}
        0 \\ -1 \\ 0
    \end{smallmatrix}
    \right] \underset{0}{\xrightarrow{(\bullet,1)}}
    \underline{\tau_1'} = 
    \left[ 
    \begin{smallmatrix}
        -1 \\ 0 \\ 0 
    \end{smallmatrix}\middle|
    \begin{smallmatrix}
        0 \\ -1 \\ 0
    \end{smallmatrix}
    \right]\right), 
\end{equation}
where $b\in \{0,1\}$. 
If $b=1$, then we have $\begin{bsmallmatrix}
    0 \\ -3 \\ 2
\end{bsmallmatrix}\in \SigmaAe$, which is a contradiction to Proposition \ref{coord_plane}. 
Thus, we have $b=0$ and 
\begin{equation}\label{D2 xi}
    \xi = 
    \left[ 
    \begin{smallmatrix}
        0 \\ -1 \\ 1 
    \end{smallmatrix}\middle|
    \begin{smallmatrix}
        1 \\ -2 \\ 1
    \end{smallmatrix}\middle|
    \begin{smallmatrix}
        0 \\ -2 \\ 1 
    \end{smallmatrix}
    \right] \in \SigmaeAe \ \ \text{with} \ \ \bm{v}_{\xi} = \begin{bsmallmatrix}
        0 \\ 0 \\ 1
    \end{bsmallmatrix}.  
\end{equation}

Thirdly, we consider an arrow 
\begin{equation}\label{D2 arrow2}
    \emax = 
    \left[
     \begin{smallmatrix}
        1 \\ 0 \\ 0 
    \end{smallmatrix}\middle|
    \begin{smallmatrix}
        1 \\ -1 \\ 1 
    \end{smallmatrix}\middle| 
    \begin{smallmatrix}
        0 \\ 0 \\ 1 
    \end{smallmatrix}
    \right] \xrightarrow{(\bullet,a_1,a_2)}
    \rho := 
    \left[
     \begin{smallmatrix}
        a_1-1 \\ -a_1 \\ a_1+a_2 
    \end{smallmatrix}\middle|
    \begin{smallmatrix}
        1 \\ -1 \\ 1 
    \end{smallmatrix}\middle| 
    \begin{smallmatrix}
        0 \\ 0 \\ 1 
    \end{smallmatrix}
    \right]
\end{equation}
starting at $\emax$, where $a_1,a_2,a_1+a_2\in \{0,1,2\}$. If $a_1=0$, then we have $\left[\begin{smallmatrix}
    -1 \\ 0 \\ a_2
\end{smallmatrix}\middle| 
\begin{smallmatrix}
    1 \\ -1 \\ 1
\end{smallmatrix}\right]\in \SigmaeAe$, which is a contradiction to the sign-coherent property. 
In addition, if $(a_1,a_2)\in \{(1,1),(2,0)\}$, then we have   
\begin{equation}
    \begin{bsmallmatrix}
    a_1 - 1 \\ -a_1 \\ 2 
\end{bsmallmatrix}\in \SigmaeAe \quad \text{and}\quad
\langle\bm{v}_{\xi}, \begin{bsmallmatrix}
    a_1 - 1 \\ -a_1 \\ 2 
\end{bsmallmatrix}\rangle = 2 \not\leq 1, 
\end{equation}
a contradiction.
Therefore, we have $(a_1,a_2) = (1,0)$ and 
\begin{equation}\label{D2 rho}
    \rho = \left[
     \begin{smallmatrix}
        0 \\ -1 \\ 1 
    \end{smallmatrix}\middle|
    \begin{smallmatrix}
        1 \\ -1 \\ 1 
    \end{smallmatrix}\middle| 
    \begin{smallmatrix}
        0 \\ 0 \\ 1 
    \end{smallmatrix}
    \right]\in \SigmaeAe. 
\end{equation}

Now, we show that $\kappa$ in \eqref{D2 kappaphi} lies in $\SigmaeAe$. One can extend \eqref{D2 arrow2} to the following path $s$ around $\bm{w}= \begin{bsmallmatrix}
    1 \\ -1 \\ 1
\end{bsmallmatrix}$. 
\begin{equation*}
    \underline{s}_{\bm{w}} \colon \ 
    \left(
    \underline{\emax} = 
    \left[
     \begin{smallmatrix}
        1 \\ 0 \\ 0 
    \end{smallmatrix}\middle|
    \begin{smallmatrix}
        0 \\ 0 \\ 1 
    \end{smallmatrix}
    \right] \underset{1}{\xrightarrow{(\bullet,0)}} 
    \underline{\rho} = 
    \left[
     \begin{smallmatrix}
        0 \\ -1 \\ 1 
    \end{smallmatrix}\middle|
    \begin{smallmatrix}
        0 \\ 0 \\ 1 
    \end{smallmatrix}
    \right] \underset{p}{\xrightarrow{(0,\bullet)}} 
    \underline{\kappa}' := 
    \left[
    \begin{smallmatrix}
        0 \\ -1 \\ 1 
    \end{smallmatrix}\middle|
    \begin{smallmatrix}
        p \\ -p \\ p-1 
    \end{smallmatrix}
    \right]\right),
\end{equation*}
where $p\in \{0,1,2\}$. 
If $p=0$, then we have $\left[\begin{smallmatrix}
    0 \\ -1 \\ 1
\end{smallmatrix}\middle|
\begin{smallmatrix}
    0 \\ 0 \\ -1
\end{smallmatrix}\right]\in \SigmaeAe$, which is a contradiction to the sign-coherent property. In addition, if $p=1$, then we have 
$\kappa' = \left[\begin{smallmatrix}
    0 \\ -1 \\ 1
\end{smallmatrix}\middle|
\begin{smallmatrix}
    1 \\ -1 \\ 1
\end{smallmatrix}\middle| 
\begin{smallmatrix}
    1 \\ -1 \\ 0
\end{smallmatrix}\right]$ and $\bm{v}_{\kappa'}= \begin{bsmallmatrix}
    0 \\ -1 \\ 0
\end{bsmallmatrix}$. In this case, we have $\langle\bm{v}_{\kappa'}, \begin{bsmallmatrix}
    0 \\ -2 \\ 1
\end{bsmallmatrix}\rangle = 2 \not\leq 1$, which is a contradiction.
Therefore, we have $p=2$ and 
\begin{equation}\label{D2 kappa}
    \kappa' = \kappa = 
    \left[
    \begin{smallmatrix}
        0 \\ -1 \\ 1 
    \end{smallmatrix}\middle|
     \begin{smallmatrix}
        1 \\ -1 \\ 1 
    \end{smallmatrix}\middle|
    \begin{smallmatrix}
        2 \\ -2 \\ 1 
    \end{smallmatrix}
    \right]\in \SigmaeAe 
\end{equation}
as desired. 

Finally, we show that $\varphi$ in \eqref{D2 kappaphi} lies in $\SigmaeAe$. 
We can extend the path \eqref{D2 path1} to the following path $\tilde{p}$  around $\bm{u}=\begin{bsmallmatrix}
    1 \\ -1 \\ 0
\end{bsmallmatrix}$. 
\begin{equation} \label{D2 path3}
    \underline{\tilde{p}}_{\bm{u}}\colon \ 
    \left(
    \underline{\varphi}' :=
    \left[ 
    \begin{smallmatrix}
        t+1 \\ -t \\ 0 
    \end{smallmatrix}\middle|
    \begin{smallmatrix}
        2 \\ -2 \\ 1
    \end{smallmatrix}
    \right]
    \underset{t}{\xrightarrow{(\bullet,1)}}
    \left[ 
    \begin{smallmatrix}
        1 \\ -2 \\ 1 
    \end{smallmatrix}\middle|
    \begin{smallmatrix}
        2 \\ -2 \\ 1
    \end{smallmatrix}
    \right] 
    \underset{0}{\xrightarrow{(2,\bullet)}}
    \underline{\tau} = 
    \left[ 
    \begin{smallmatrix}
        1 \\ -2 \\ 1 
    \end{smallmatrix}\middle|
    \begin{smallmatrix}
        0 \\ -2 \\ 1
    \end{smallmatrix}
    \right] \underset{1}{\xrightarrow{(1,\bullet)}}
    \underline{\emin} = \left[ 
    \begin{smallmatrix}
        0 \\ -1 \\ 0 
    \end{smallmatrix}\middle|
    \begin{smallmatrix}
        0 \\ -2 \\ 1
    \end{smallmatrix}
    \right] \underset{0}{\xrightarrow{(\bullet,2)}}
    \underline{\tau_1} = \left[ 
    \begin{smallmatrix}
        0 \\ -1 \\ 0 
    \end{smallmatrix}\middle|
    \begin{smallmatrix}
        0 \\ 0 \\ -1
    \end{smallmatrix}
    \right]\right),  
\end{equation}
where $t\in \{0,1\}$. 
If $t=1$, then we have $\begin{bsmallmatrix}
    2 \\ -1 \\ 0
\end{bsmallmatrix}\in \SigmaeAe$, 
a contradiction to our assumption that $l_{12}=1$ by Proposition \ref{coord_plane}. 
Thus, we have $t=0$ and
\begin{equation}
    \varphi' = \varphi = \left[ 
    \begin{smallmatrix}
        1 \\ -1 \\ 0 
    \end{smallmatrix}\middle|
    \begin{smallmatrix}
        1 \\ 0 \\ 0 
    \end{smallmatrix}\middle|
    \begin{smallmatrix}
        2 \\ -2 \\ 1
    \end{smallmatrix}
    \right]\in \SigmaeAe 
\end{equation}
as desired. 
It completes a proof.

\subsection{Case $d(13)$} \label{case:13}
Assume that $d_{+-+}(A,e) = d(13) = ((2,1,1),(1,2,1))$. 
Our calculation is similar to the previous case $d(12)$. 
By Proposition \ref{min_max}, $\emax$ and $\emin$ are the same as \eqref{D1 mm}. 
In addition, since $(r_{12},r_{32})=(1,2)$ and $h_{12}=h_{32}=1$, one can apply the same discussion to obtain the paths $\alpha$, $\beta$ in \eqref{D2 p1}, \eqref{D2 p3} respectively and also maximal cones $\tau$, $\xi$, $\rho$, $\kappa$ in \eqref{D2 tau}, \eqref{D2 xi}, \eqref{D2 rho} and \eqref{D2 kappa} respectively.
Moreover, we obtain the sequence \eqref{D2 path3} with $t=1$ and 
\begin{equation}
    \varphi' = \left[ 
    \begin{smallmatrix}
        1 \\ -1 \\ 0 
    \end{smallmatrix}\middle|
    \begin{smallmatrix}
        2 \\ -1 \\ 0 
    \end{smallmatrix}\middle|
    \begin{smallmatrix}
        2 \\ -2 \\ 1
    \end{smallmatrix}
    \right]\in \SigmaeAe. 
\end{equation}
In fact, if $t=0$, then we have $\left[\begin{smallmatrix}
    1 \\ 0 \\ 0
\end{smallmatrix}\middle| 
\begin{smallmatrix}
    1 \\ -1 \\ 0
\end{smallmatrix}\right]\in \SigmaeAe$, which is a contradiction to our assumption that $l_{12}=2$. 
Therefore, these maximal cones give rise to one of fans described in Figure \ref{fig:D2}, where the left one is the fan $\Sigma_{d(13)}$. 

\begin{figure}[ht]
\begin{tabular}{ccccc}
    $\Sigma_{d(13)}$ & (a) \\ 
\begin{tikzpicture}[baseline=0mm, scale=1]
    \coordinate(0) at(0:0); 
    \node(x) at(215:1) {}; 
    \node(y) at(0:1) {}; 
    \node(z) at(90:1) {}; 
    \draw[gray, <-] ($1*(x)$)--($-1*(x)$); 
    \draw[gray, <-] ($1*(y)$)--($-1*(y)$); 
    \draw[gray, <-] ($1*(z)$)--($-1*(z)$); 
    \node at(x) [below]{$\bm{e}_1$}; 
    \node at(y) [below]{$\bm{e}_2$}; 
    \node at(z) [right]{$\bm{e}_3$}; 
    \def\gridwidth{0.1}
    \begin{scope}[gray, line width = \gridwidth]
    \foreach \c in {0,1,2} {
    \draw ($-\c*(y)$)--($2*(x)+ -\c*(y)$); 
    \draw ($\c*(x)$)--($-2*(y)+ \c*(x)$); 
    \draw ($-\c*(y)$)--($2*(z)+ -\c*(y)$); 
    \draw ($\c*(z)$)--($-2*(y)+ \c*(z)$); 
    \draw ($-2*(y)+\c*(z)$)--($-2*(y)+ 2*(x)+ \c*(z)$); 
    \draw ($-2*(y)+\c*(x)$)--($-2*(y)+ 2*(z)+ \c*(x)$); }
    \end{scope}
    \fill[black, opacity = 0.8] (0)circle(0.5mm);
    \coordinate(1) at($1*(x) + 0*(y) + 0*(z)$); 
    \coordinate(2) at($1*(x) + -1*(y) + 1*(z)$); 
    \coordinate(3) at($0*(x) + 0*(y) + 1*(z)$); 
    \coordinate(4) at($0*(x) + -1*(y) + 1*(z)$); 
    \coordinate(5) at($2*(x) + -2*(y) + 1*(z)$); 
    \coordinate(6) at($2*(x) + -1*(y) + 0*(z)$); 
    \coordinate(7) at($1*(x) + -2*(y) + 1*(z)$); 
    \coordinate(8) at($1*(x) + -1*(y) + 0*(z)$); 
    \coordinate(9) at($0*(x) + -2*(y) + 1*(z)$); 
    \coordinate(10) at($0*(x) + -1*(y) + 0*(z)$); 
\def\segwidth{0.5mm}
\def\positivecolor{green!30!red}
\def\negativecolor{green!30!orange}
\def\segcolor{green!30!blue}
    \foreach \s\t in {(8)/(9), (8)/(10), (10)/(9)} 
    {\draw[line width = 0.4mm, \negativecolor] \s--\t;}
    \foreach \s\t in {(8)/(6), (8)/(5), (8)/(7)} 
    {\draw[line width = 0.4mm, \segcolor] \s--\t;}
    \foreach \s\t in {(1)/(2), (1)/(3), (2)/(3)} 
    {\draw[line width = \segwidth, \positivecolor] \s--\t;}
    \foreach \s\t in {(4)/(2), (4)/(3), (1)/(5), (2)/(5), (4)/(5), (1)/(6), (6)/(5), (4)/(7), (7)/(5), (4)/(9), (7)/(9)} 
        {\draw[line width = \segwidth, \segcolor] \s--\t;} 
    \foreach \v in {(1),(2),(3),(4),(5),(6),(7),(8),(9),(10)} 
        {\fill[gray, opacity = 0.8] \v circle(1.1mm);} 
\end{tikzpicture}
& 
\begin{tikzpicture}[baseline=0mm, scale=1]
    \coordinate(0) at(0:0); 
    \node(x) at(215:1) {}; 
    \node(y) at(0:1) {}; 
    \node(z) at(90:1) {}; 
    \draw[gray, <-] ($1*(x)$)--($-1*(x)$); 
    \draw[gray, <-] ($1*(y)$)--($-1*(y)$); 
    \draw[gray, <-] ($1*(z)$)--($-1*(z)$); 
    \node at(x) [below]{$\bm{e}_1$}; 
    \node at(y) [below]{$\bm{e}_2$}; 
    \node at(z) [right]{$\bm{e}_3$}; 
    \def\gridwidth{0.1}
    \begin{scope}[gray, line width = \gridwidth]
    \foreach \c in {0,1,2} {
    \draw ($-\c*(y)$)--($2*(x)+ -\c*(y)$); 
    \draw ($\c*(x)$)--($-2*(y)+ \c*(x)$); 
    \draw ($-\c*(y)$)--($2*(z)+ -\c*(y)$); 
    \draw ($\c*(z)$)--($-2*(y)+ \c*(z)$); 
    \draw ($-2*(y)+\c*(z)$)--($-2*(y)+ 2*(x)+ \c*(z)$); 
    \draw ($-2*(y)+\c*(x)$)--($-2*(y)+ 2*(z)+ \c*(x)$); }
    \end{scope}
    \fill[black, opacity = 0.8] (0)circle(0.5mm);
    \coordinate(1) at($1*(x) + 0*(y) + 0*(z)$); 
    \coordinate(2) at($1*(x) + -1*(y) + 1*(z)$); 
    \coordinate(3) at($0*(x) + 0*(y) + 1*(z)$); 
    \coordinate(4) at($0*(x) + -1*(y) + 1*(z)$); 
    \coordinate(5) at($2*(x) + -2*(y) + 1*(z)$); 
    \coordinate(6) at($2*(x) + -1*(y) + 0*(z)$); 
    \coordinate(7) at($1*(x) + -2*(y) + 1*(z)$); 
    \coordinate(8) at($1*(x) + -1*(y) + 0*(z)$); 
    \coordinate(9) at($0*(x) + -2*(y) + 1*(z)$); 
    \coordinate(10) at($0*(x) + -1*(y) + 0*(z)$); 
\def\segwidth{0.5mm}
\def\positivecolor{green!30!red}
\def\negativecolor{green!30!orange}
\def\segcolor{green!30!blue}
    \foreach \s\t in {(8)/(9), (8)/(10), (10)/(9)} 
    {\draw[line width = 0.4mm, \negativecolor] \s--\t;}
    \foreach \s\t in {(8)/(6), (8)/(5), (8)/(7)} 
    {\draw[line width = 0.4mm, \segcolor] \s--\t;}
    \foreach \s\t in {(1)/(2), (1)/(3), (2)/(3)} 
    {\draw[line width = \segwidth, \positivecolor] \s--\t;}
    \foreach \s\t in {(4)/(2), (4)/(3), (2)/(6), (2)/(5), (4)/(5), (1)/(6), (6)/(5), (4)/(7), (7)/(5), (4)/(9), (7)/(9)} 
        {\draw[line width = \segwidth, \segcolor] \s--\t;} 
    \foreach \v in {(1),(2),(3),(4),(5),(6),(7),(8),(9),(10)} 
        {\fill[gray, opacity = 0.8] \v circle(1.1mm);} 
\end{tikzpicture}
\end{tabular}
\caption{}
\label{fig:D3}
\end{figure}

If this is (a), then $\SigmaAe$ contains the following maximal cones.
\[
    \gamma := 
    \left[
    \begin{smallmatrix}
        2 \\ -1 \\ 0
    \end{smallmatrix}\middle| 
    \begin{smallmatrix}
        1 \\ -1 \\ 1
    \end{smallmatrix}\middle| 
    \begin{smallmatrix}
        2 \\ -2 \\ 1
    \end{smallmatrix}
    \right], \ 
    \kappa = \left[
    \begin{smallmatrix}
        0 \\ -1 \\ 1
    \end{smallmatrix}\middle| 
    \begin{smallmatrix}
        1 \\ -1 \\ 1
    \end{smallmatrix}\middle| 
    \begin{smallmatrix}
        2 \\ -2 \\ 1
    \end{smallmatrix}
    \right], \ 
    \delta := \left[
    \begin{smallmatrix}
        0 \\ -1 \\ 1
    \end{smallmatrix}\middle| 
    \begin{smallmatrix}
        1 \\ -2 \\ 1
    \end{smallmatrix}\middle| 
    \begin{smallmatrix}
        2 \\ -2 \\ 1
    \end{smallmatrix}
    \right]. 
\]
By Theorem \ref{theorem:g-fan is ordered}, one can find that there is the following path $p'$ around $\bm{v}:=\begin{bsmallmatrix}
    2 \\ -2 \\ 1
\end{bsmallmatrix}$. 
\begin{equation}\label{D3 path1}
    \underline{p}'_{\bm{v}} = \left(
    \underline{\gamma} = 
    \left[
    \begin{smallmatrix}
        2 \\ -1 \\ 0
    \end{smallmatrix}\middle| 
    \begin{smallmatrix}
        1 \\ -1 \\ 1
    \end{smallmatrix}  
    \right]  \underset{1}{\xrightarrow{(\bullet,0)}}
    \underline{\kappa} = \left[
    \begin{smallmatrix}
        0 \\ -1 \\ 1
    \end{smallmatrix}\middle| 
    \begin{smallmatrix}
        1 \\ -1 \\ 1
    \end{smallmatrix}  
    \right] \underset{1}{\xrightarrow{(1,\bullet)}} 
    \underline{\delta} = 
    \left[
    \begin{smallmatrix}
        0 \\ -1 \\ 1
    \end{smallmatrix}\middle| 
    \begin{smallmatrix}
        1 \\ -2 \\ 1
    \end{smallmatrix}  
    \right]\right).    
\end{equation}
However, it contradicts to Proposition \ref{rank2}. 
Therefore, we conclude that $\SigmaeAe = \Sigma_{d(13)}$ in this case. 

\bigskip


Now, we are ready to prove Theorem \ref{thm:orthant}. 

\begin{proof}[Proof of Theorem \ref{thm:orthant}]
We get the assertion (2) by consequences of Subsections \ref{case:0}-\ref{case:13}. 

In the following, we prove (1). We need to exclude the cases indicated by "$-$" in Table \ref{tab:15list}. 

Firstly, if $h_{12} = h_{32} = 0$, then $d_{+-+}(A,e)=((0,0,0),(0,0,0))$ holds, 
as in the proof in Subsection \ref{case:0}. 
So, this case has been done. 

Secondly, we assume that 
$d_{+-+}(A,e) \in \{((1,1,1),(2,1,0)), ((1,2,1),(2,1,0))\}$. 
In this case, we have 
\begin{equation*}
    \emax = 
    \left[
    \begin{smallmatrix}
        1 \\ 0 \\ 0
    \end{smallmatrix}\middle| 
    \begin{smallmatrix}
        1 \\ -1 \\ 0
    \end{smallmatrix}\middle|
    \begin{smallmatrix}
        0 \\ 0 \\ 1
    \end{smallmatrix}
    \right]\in \SigmaeAe \ \ \text{with} \ \ \bm{v}_{\emax} = 
    \begin{bsmallmatrix}
    1 \\ 0 \\ 1    
    \end{bsmallmatrix}. 
\end{equation*}
By $l_{32}=2$, 
we have $
\begin{bsmallmatrix}
        0 \\ -1 \\ 2
\end{bsmallmatrix}\in \SigmaeAe$. 
However, we have $\langle \bm{v}_{\emax},  \begin{bsmallmatrix}
        0 \\ -1 \\ 2
    \end{bsmallmatrix}\rangle = 2\not \leq 1$, a contraction.

Thirdly, we assume that $d_{+-+}(A,e) \in \{((2,1,1),(1,1,0)), ((2,1,1),(1,2,0))\}$. 
By Proposition \ref{min_max}, we have 
\begin{equation*}
    \emax = 
    \left[
    \begin{smallmatrix}
        1 \\ 0 \\ 0
    \end{smallmatrix}\middle| 
    \begin{smallmatrix}
        2 \\ -1 \\ 0
    \end{smallmatrix}\middle|
    \begin{smallmatrix}
        0 \\ 0 \\ 1
    \end{smallmatrix}
    \right]\in \SigmaeAe  \ \  \text{with} \ \  \bm{v}_{\emax} = 
    \begin{bsmallmatrix}
    1 \\ 1 \\ 1    
    \end{bsmallmatrix}. 
\end{equation*}
Starting from the positive cone $\sigma_+$, we have the following path $p$ around $\bm{e}_3$. 
\begin{equation}\nonumber
    \underline{p}_{\bm{e}_3} \colon\ 
    \left(\underline{\sigma_{+}} = \left[
    \begin{smallmatrix}
        1 \\ 0 \\ 0
    \end{smallmatrix}\middle|
    \begin{smallmatrix}
        0 \\ 1 \\ 0
    \end{smallmatrix}
    \right] 
    \underset{0}{\xrightarrow{(\bullet,2)}}
    \underline{\emax} = \left[
    \begin{smallmatrix}
        1 \\ 0 \\ 0
    \end{smallmatrix}\middle|
    \begin{smallmatrix}
        2 \\ -1 \\ 0
    \end{smallmatrix}
    \right] 
    \underset{a_1}{\xrightarrow{(\bullet,1)}}
    \left[
    \begin{smallmatrix}
        1 \\ -1 \\ a_1
    \end{smallmatrix}\middle|
    \begin{smallmatrix}
        2 \\ -1 \\ 0
    \end{smallmatrix}
    \right]
    \underset{0}{\xrightarrow{(2,\bullet)}}
    \left[
    \begin{smallmatrix}
        1 \\ -1 \\ a_1
    \end{smallmatrix}\middle|
    \begin{smallmatrix}
        0 \\ -1 \\ 2a_1
    \end{smallmatrix}
    \right]
    \underset{a_2}{\xrightarrow{(\bullet,1)}}
    \left[
    \begin{smallmatrix}
        -1 \\ 0 \\ 2a_1+a_2
    \end{smallmatrix}\middle|
    \begin{smallmatrix}
        0 \\ -1 \\ 2a_1
    \end{smallmatrix}
    \right]\right), 
\end{equation}
where $a_1,a_2\in \{0,1\}$. Since $a_{1}\in \{0,1\}$, 
we have either 
    $\left[\begin{smallmatrix}
        0 \\ -1 \\ 0
    \end{smallmatrix}\middle|
    \begin{smallmatrix}
        0 \\ 0 \\ 1
    \end{smallmatrix} \right] \in \SigmaAe$ or 
    $\left[\begin{smallmatrix}
        0 \\ -1 \\ 2
    \end{smallmatrix}\middle|
    \begin{smallmatrix}
        0 \\ 0 \\ 1
    \end{smallmatrix}\right] \in \SigmaAe$, in other words, $l_{32}=0$ or $l_{32}=2$ respectively. 
    However, this contradicts our assumption that $l_{32}=1$. 

Finally, we consider the case 
$d_{+-+}(A,e) = ((1,2,1),(1,2,1))$. 
By Proposition \ref{min_max}, we have 
\begin{equation}
    \emax = 
    \left[
     \begin{smallmatrix}
        1 \\ 0 \\ 0 
    \end{smallmatrix}\middle|
    \begin{smallmatrix}
        1 \\ -1 \\ 1 
    \end{smallmatrix}\middle| 
    \begin{smallmatrix}
        0 \\ 0 \\ 1 
    \end{smallmatrix}
    \right] \quad \text{and} \quad 
    \emin = \left[
     \begin{smallmatrix}
        1 \\ -2 \\ 0 
    \end{smallmatrix}\middle|
    \begin{smallmatrix}
        0 \\ -1 \\ 0 
    \end{smallmatrix}\middle| 
    \begin{smallmatrix}
        0 \\ -2 \\ 1 
    \end{smallmatrix}
    \right]. 
\end{equation}
Staring at the positive cone $\sigma_+$, there is a path 
\begin{equation*}
    \sigma_+ = 
    \left[
    \begin{smallmatrix}
        1 \\ 0 \\ 0
    \end{smallmatrix}\middle|
    \begin{smallmatrix}
        0 \\ 1 \\ 0
    \end{smallmatrix}\middle|
    \begin{smallmatrix}
        0 \\ 0 \\ 1
    \end{smallmatrix}
    \right]\xrightarrow{(1,\bullet,1)} 
    \emax = 
    \left[
    \begin{smallmatrix}
        1 \\ 0 \\ 0
    \end{smallmatrix}\middle|
    \begin{smallmatrix}
        1 \\ -1 \\ 1
    \end{smallmatrix}\middle|
    \begin{smallmatrix}
        0 \\ 0 \\ 1
    \end{smallmatrix}
    \right]\xrightarrow{(a_1,a_2,\bullet)}
    \left[
    \begin{smallmatrix}
        1 \\ 0 \\ 0
    \end{smallmatrix}\middle|
    \begin{smallmatrix}
        1 \\ -1 \\ 1
    \end{smallmatrix}\middle|
    \begin{smallmatrix}
        a_1+a_2 \\ -a_2 \\ a_2-1
    \end{smallmatrix}
    \right],
\end{equation*}
where $(a_1,a_2) \in \{(0,1),(1,1),(0,2)\}$ since $a_2>0$ by Proposition \ref{rank2}. 
Similarly, we have a path 
\begin{equation*}
    \sigma_+ = 
    \left[
    \begin{smallmatrix}
        1 \\ 0 \\ 0
    \end{smallmatrix}\middle|
    \begin{smallmatrix}
        0 \\ 1 \\ 0
    \end{smallmatrix}\middle|
    \begin{smallmatrix}
        0 \\ 0 \\ 1
    \end{smallmatrix}
    \right]\xrightarrow{(1,\bullet,1)} 
    \emax = 
    \left[
    \begin{smallmatrix}
        1 \\ 0 \\ 0
    \end{smallmatrix}\middle|
    \begin{smallmatrix}
        1 \\ -1 \\ 1
    \end{smallmatrix}\middle|
    \begin{smallmatrix}
        0 \\ 0 \\ 1
    \end{smallmatrix}
    \right]\xrightarrow{(\bullet,b_2,b_1)}
    \left[
    \begin{smallmatrix}
        b_2-1 \\ -b_2 \\ b_1+b_2
    \end{smallmatrix}\middle|
    \begin{smallmatrix}
        1 \\ -1 \\ 1
    \end{smallmatrix}\middle|
    \begin{smallmatrix}
        0 \\ 0 \\ 1
    \end{smallmatrix}
    \right]
\end{equation*}
with $(b_1,b_2)\in \{(0,1),(1,1),(0,2)\}$. Now, let $\bm{u}:=\left[
    \begin{smallmatrix}
        1 \\ -1 \\ 1
    \end{smallmatrix}\right]\in \SigmaeAe$. 

\begin{enumerate}[\rm (i)]
    \item If $(a_1,a_2)=(1,1)$, then we have $\left[
    \begin{smallmatrix}
        1 \\ 0 \\ 0
    \end{smallmatrix}\middle| 
    \begin{smallmatrix}
        2 \\ -1 \\ 0
    \end{smallmatrix}
    \right]\in \SigmaAe$, that is, $l_{12}=2$. It contradicts to our assumption that $l_{12}=1$. Similarly, we have a contradiction if $(b_1,b_2) = (1,1)$. 
    \item Next, we consider the case $(a_1,a_2) = (0,2)$. 
    Starting from $\emax$, there is the following path $p$ around $\bm{u}$. 
    \begin{equation}
    \underline{p}_{\bm{u}} = \left(\underline{\emax} = 
    \left[
    \begin{smallmatrix}
        1 \\ 0 \\ 0
    \end{smallmatrix}\middle|
    \begin{smallmatrix}
        0 \\ 0 \\ 1
    \end{smallmatrix}
    \right]\underset{2}{\xrightarrow{(0,\bullet)}}
    \left[
    \begin{smallmatrix}
        1 \\ 0 \\ 0
    \end{smallmatrix}\middle|
    \begin{smallmatrix}
        2 \\ -2 \\ 1
    \end{smallmatrix}
    \right] \underset{c}{\xrightarrow{(\bullet,0)}}
    \underline{\tau} := \left[
    \begin{smallmatrix}
        c-1 \\ -c \\ c
    \end{smallmatrix}\middle|
    \begin{smallmatrix}
        2 \\ -2 \\ 1
    \end{smallmatrix}
    \right]\right),
    \end{equation}
where $c\in \{0,1,2\}$. If $c=0$, then 
we have $\left[
    \begin{smallmatrix}
        -1 \\ 0 \\ 0
    \end{smallmatrix}\middle|
    \begin{smallmatrix}
        2 \\ -2 \\ 1
    \end{smallmatrix}
    \right]\in \SigmaAe$, which is a contradiction to the sign-coherent property. 
    If $c=1$, then we have 
\begin{equation*}
    \tau = \left[
    \begin{smallmatrix}
        0 \\ -1 \\ 1
    \end{smallmatrix}\middle|
    \begin{smallmatrix}
        1 \\ -1 \\ 1
    \end{smallmatrix}\middle|
    \begin{smallmatrix}
        2 \\ -2 \\ 1
    \end{smallmatrix}
    \right]\in \SigmaeAe \ \ \text{with} \ \ \bm{v}_{\tau} = 
    \begin{bsmallmatrix}
    0 \\ 0 \\ 1    
    \end{bsmallmatrix}. 
\end{equation*}
If moreover $(b_1,b_2) = (0,2)$, then 
we have $\bm{w} := \begin{bsmallmatrix}
    1 \\ -2 \\ 2 
\end{bsmallmatrix}\in \SigmaeAe$. However, we have $\langle \bm{v}_{\tau}, \bm{w}\rangle = 2 \not\leq1$, which is a contradiction.
On the other hand, if $(b_1,b_2) = (0,1)$, then we have the following path $q$ around $\bm{u}$.
\begin{equation}
    \underline{q}_{\bm{u}} = \left( 
    \left[
    \begin{smallmatrix}
        1 \\ 0 \\ 0
    \end{smallmatrix}\middle|
    \begin{smallmatrix}
        0 \\ 0 \\ 1
    \end{smallmatrix}
    \right] \underset{1}{\xrightarrow{(\bullet,1)}}
    \left[
    \begin{smallmatrix}
        0 \\ -1 \\ 1
    \end{smallmatrix}\middle|
    \begin{smallmatrix}
        0 \\ 0 \\ 1
    \end{smallmatrix}
    \right] \underset{2}{\xrightarrow{(0,\bullet)}}
    \underline{\tau} = \left[
    \begin{smallmatrix}
        0 \\ -1 \\ 1
    \end{smallmatrix}\middle|
    \begin{smallmatrix}
        2 \\ -2 \\ 1
    \end{smallmatrix}
    \right]\right).
\end{equation}
It is a contradiction to the result in Proposition \ref{rank2}. 

Finally, we assume that $c=2$. In this case, we have 
\begin{equation}
     \tau = \left[
    \begin{smallmatrix}
        1 \\ -2 \\ 2
    \end{smallmatrix}\middle|
    \begin{smallmatrix}
        1 \\ -1 \\ 1
    \end{smallmatrix}\middle|
    \begin{smallmatrix}
        2 \\ -2 \\ 1
    \end{smallmatrix}
    \right]\in \SigmaeAe.
\end{equation}
We consider an arrow 
\begin{equation}
    \tau' := 
    \left[
    \begin{smallmatrix}
        1 \\ -2 \\ 0
    \end{smallmatrix}\middle|
    \begin{smallmatrix}
        f_1 \\ -2(f_1+f_2)+1 \\ f_2
    \end{smallmatrix}\middle|
    \begin{smallmatrix}
        0 \\ -2 \\ 1
    \end{smallmatrix}
    \right] \xrightarrow{(f_1,\bullet,f_2)} 
    \emin = \left[
    \begin{smallmatrix}
        1 \\ -2 \\ 0
    \end{smallmatrix}\middle|
    \begin{smallmatrix}
        0 \\ -1 \\ 0
    \end{smallmatrix}\middle|
    \begin{smallmatrix}
        0 \\ -2 \\ 1
    \end{smallmatrix}
    \right] 
\end{equation}
ending at $\emin$, where 
$(f_1,f_2) \in \{(0,0),(0,1),(1,0),(1,1),(0,2),(2,0)\}$. If $f_1=f_2=0$, then we have $\left[\begin{smallmatrix}
        0 \\ 1 \\ 0
    \end{smallmatrix}\middle|
    \begin{smallmatrix}
        0 \\ -2 \\ 1
    \end{smallmatrix}
    \right]\in \SigmaeAe$, which is a contradiction to the sign-coherent property. 
    If $(f_1,f_2) \in \{(0,1),(1,0)\}$, 
    then we have $\bm{v}_{\tau'} = \begin{bsmallmatrix}
        1 \\ 0 \\ 1
    \end{bsmallmatrix}$. 
    However, we have $\langle \bm{v}_{\tau'}, \bm{u}\rangle = 2\not\leq 1$, which is a contradiction.
    If $(f_1,f_2)\in \{(0,2),(2,0)\}$, then we have either 
$\begin{bsmallmatrix}
    0 \\ -3 \\ 2
\end{bsmallmatrix}\in \SigmaeAe$ or $\begin{bsmallmatrix}
    2 \\ -3 \\ 0
\end{bsmallmatrix}\in \SigmaeAe$, which is a contradiction to Proposition \ref{rank2}.
Lastly, we assume $(f_1,f_2) = (1,1)$. In this case, we have  
\begin{equation}
    \tau' = 
    \left[
    \begin{smallmatrix}
        1 \\ -2 \\ 0
    \end{smallmatrix}\middle|
    \begin{smallmatrix}
        1 \\ -3 \\ 1
    \end{smallmatrix}\middle|
    \begin{smallmatrix}
        0 \\ -2 \\ 1
    \end{smallmatrix}
    \right] \in \SigmaAe. 
\end{equation}
However, this implies that $\PeAe$ contains a non-zero integer vector 
\begin{equation}
    \bm{x}:=\begin{bsmallmatrix}
        1 \\ -2 \\ 1
    \end{bsmallmatrix}= \dfrac{1}{5}\begin{bsmallmatrix}
        2 \\ -2 \\ 1
    \end{bsmallmatrix}+\dfrac{1}{5}\begin{bsmallmatrix}
        1 \\ -2 \\ 2
    \end{bsmallmatrix}+\dfrac{2}{5}\begin{bsmallmatrix}
        1 \\ -3 \\ 1
    \end{bsmallmatrix}
    \in \PeAe
\end{equation}
as its interior point. 
This is a contradiction. 
Thus, the case $(a_1,a_2) = (0,2)$ does not appear. 

By permuting a role of $\bm{e}_1$ and $\bm{e}_3$, we also have a contradiction when $(b_1,b_2)= (0,2)$. 
\item The remaining case is $(a_1,a_2) = (b_1,b_2)=(0,1)$. 
In this case, the Hasse quiver of the interval $(\Sigma_{\bm{u}}(A,e),\leq)$ is given by 
\begin{equation}
    \xymatrix{
    & \hspace{-15mm} \emax = \text{$\left[
    \begin{smallmatrix}
        1 \\ 0 \\ 0 
    \end{smallmatrix}
    \middle| 
    \begin{smallmatrix}
        1 \\ -1 \\ 1 
    \end{smallmatrix} 
    \middle| 
    \begin{smallmatrix}
        0 \\ 0 \\ 1 
    \end{smallmatrix} 
    \right]$} 
    \ar[r]^-{(0,1,\bullet)} 
    \ar[d]^-{(\bullet,1,0)}  
    & 
    \text{$\left[
    \begin{smallmatrix}
        1 \\ 0 \\ 0 
    \end{smallmatrix}
    \middle| 
    \begin{smallmatrix}
        1 \\ -1 \\ 1 
    \end{smallmatrix} 
    \middle| 
    \begin{smallmatrix}
        1 \\ -1 \\ 0 
    \end{smallmatrix} 
    \right]$} \ar[d]^-{(\bullet,1,0)} & \\  
    & \text{$\left[
    \begin{smallmatrix}
        0 \\ -1 \\ 1 
    \end{smallmatrix}
    \middle| 
    \begin{smallmatrix}
        1 \\ -1 \\ 1 
    \end{smallmatrix} 
    \middle| \begin{smallmatrix}
        0 \\ 0 \\ 1 
    \end{smallmatrix} 
    \right]$} \ar[r]^-{(0,1,\bullet)} 
    & 
    \text{$\left[
    \begin{smallmatrix}
        0 \\ -1 \\ 1 
    \end{smallmatrix}
    \middle| \begin{smallmatrix}
        1 \\ -1 \\ 1 
    \end{smallmatrix} 
    \middle| 
    \begin{smallmatrix}
        1 \\ -1 \\ 0 
    \end{smallmatrix} 
    \right]=: \underline{\rho}$}\hspace{-9mm}
    }
\end{equation}
In this case, we have $\bm{v}_{\rho} = \begin{bsmallmatrix}
    0 \\ -1 \\ 0
\end{bsmallmatrix}$ and $\langle \bm{v}_{\rho}, \begin{bsmallmatrix}
    1 \\ -2 \\ 0
\end{bsmallmatrix} \rangle = 2\not\leq 1$.
This is a contradiction.  
\end{enumerate}
Since all the cases have been considered in the above discussion, we finish the proof.
\end{proof}

\section{Proof of Theorem \ref{main theorem}}
\label{section:Narrowing down}

In this section, we give a proof of our main result (Theorem \ref{main theorem}).

\subsection{Proof of Theorem \ref{main theorem}(1)}

Firstly, we deduce Theorem \ref{main theorem} (1) as a consequence of Theorem \ref{thm:orthant}. 
Let $G=\mathfrak{S}_3\times \{\pm1\}$. 

\begin{proof}[Proof of Theorem \ref{main theorem}(1)]
Let $(A,e = (e_1,e_2,e_3))$ be a $g$-convex algebra of rank $3$. 
Since the subfans $\Sigma_{+++}(A,e)$ and $\Sigma_{---}(A,e)$ are trivial by Proposition \ref{prop:positive_cone}, 
it suffices to determine $\Sigma_{\epsilon}(A,e)$ for each $\epsilon\in \{\pm\}^3\setminus\{(+++),(---)\}$. 
We take an element $g=(s,z)\in G$ such that $\epsilon = g^{-1}(+-+)$. 
Then, it gives an isomorphism $\Sigma_{\epsilon}(A,e)\simeq \Sigma_{+-+}(A^g,e^g)$. 
Since $\Sigma_{+-+}(A^g,e^g)$ is determined by $d(A,e)$ by Theorem \ref{thm:orthant}, so is $\Sigma_{+-+}(A,e)$. 
Thus, $\Sigma(A,e)$ is determined by $d(A,e)$. 
Finally, the fans in Table \ref{fig:15poly} are pairwise distinct. 
\end{proof}

\subsection{Proof of Theorem \ref{main theorem}(2)}
In this section, we give a proof of Theorem \ref{main theorem}(2). To do this, we need additional constraints on our numerical data.

Let $(A,e=(e_1,e_2,e_3))$ be a $g$-convex algebra of rank $3$ and $d(A,e)$ the data on the minimal number of generators.  
In addition, let $d_{+-+} := d_{+-+}(A,e)$. 

\begin{proposition} \label{Gen A}
The following statements hold. 
\begin{enumerate}
    \item[\rm $(0)$] 
    If $d_{+-+} \neq d(0)$, then we have $h_{12} = 1$ or $h_{32} = 1$. 
    \item[\rm $(1)$] 
        If $d_{+-+} = d(1)$, 
        then we have $d_{31}=(0,0,0)$ or $h_{31}=1$.
    \item[\rm $(2)$] 
        If $d_{+-+} = d(2)$, 
    then we have $d_{31}=(0,0,0)$ and $h_{13} = 0$.
    \item[\rm $(3)$] 
        If $d_{+-+} = d(3)$, 
        then we have $d_{31}=(0,0,0)$ or $h_{31}=1$. 
    \item[\rm $(4)$] 
        If $d_{+-+} = d(4)$, 
        then we have $d_{31}\in \{(2,1,1),(1,1,1)\}$.
    \item[\rm $(5)$]
        If $d_{+-+} = d(5)$, 
        then we have $d_{31}=(1,2,1)$ and $h_{13}=0$.
    \item[\rm $(6)$] 
        If $d_{+-+} = d(6)$, 
        then we have $d_{31}\in \{(2,1,1),(1,1,1)\}$.
    \item[\rm $(7)$] 
        If $d_{+-+} = d(7)$, 
        then we have $d_{31},d_{13} \not\in \{(2,1,1),(2,1,0)\}$.
    \item[\rm $(8)$] 
        If $d_{+-+} = d(8)$, 
        then we have $d_{31}\not\in \{(2,1,1),(2,1,0)\}$ and $d_{13}\neq (0,0,0)$.
    \item[\rm $(9)$]
        If $d_{+-+} = d(9)$, 
        then we have $d_{31},d_{13} \neq (0,0,0)$.
    \item[\rm $(10)$] 
        If $d_{+-+} = d(10)$ and
        \begin{itemize}
            \item $h_{13} = 0$, then we have $d_{31}=(2,1,1)$. 
            \item $h_{13} = 1$, then we have $d_{31}\in \{(2,1,1), (1,1,1)\}$.
        \end{itemize}        
    \item[\rm $(11)$] 
        If $d_{+-+} = d(11)$, 
        then we have $d_{31}=(1,2,1)$ and $h_{13}=1$.
    \item[\rm $(12)$] 
        If $d_{+-+} = d(12)$, 
        then we have $d_{31}\in \{(1,1,0),(1,2,0),(0,0,0)\}$ and $d_{13}=(0,0,0)$.
    \item[\rm $(13)$] 
        If $d_{+-+} = d(13)$, 
        then we have $d_{31}\in \{(1,1,0),(1,2,0),(0,0,0)\}$ and $d_{13}=(2,1,0)$.
    \end{enumerate}
\end{proposition}

On the other hand, we exclude the following cases. 

\begin{proposition}\label{Gen X}
The following (i) and (ii) can not appear.
\begin{enumerate}
    \item[\rm (i)] 
    $d_{12} = (2,1,1)$, $d_{13} = (2,1,0)$, $d_{23} = (2,1,1)$ and $d_{32} = (1,1,1)$.
    \item[\rm (ii)] 
    $d_{12} = (2,1,1)$, $d_{13} = (1,1,0)$, $d_{21} = (1,1,1)$, $d_{23} = (1,2,1)$ and $h_{32}=0$.
\end{enumerate}
\end{proposition}

Using these results, we prove Theorem \ref{main theorem}(2). 

\begin{proof}[Proof of Theorem \ref{main theorem}(2)]
    Let $\mathcal{S}$ be the set of numerical data $d = (d_{ij})_{1\leq i\neq j \leq 3}$ with 
    \begin{equation}
        d_{ij} = (l_{ij},r_{ij},h_{ij}) \in \{(0,0,0),(1,1,0),(2,1,0),(1,2,0),(1,1,1),(2,1,1),(1,2,1)\}. 
    \end{equation}
    In addition, let $d_{+-+} := (d_{12},d_{32})$ for each $d\in \mathcal{S}$.  
    
    The group $G$ acts on the set $\mathcal{S}$ in the following way: 
    For $g = (s,z)\in G$ and $d \in \mathcal{S}$, 
    we define $d^g \in \mathcal{S}$ by 
    \begin{equation}
        (d^g)_{ij} := 
        \begin{cases}
             (l_{s^{-1}(i)s^{-1}(j)}, r_{s^{-1}(i)s^{-1}(j)}, h_{s^{-1}(i)s^{-1}(j)}) 
             &\text{if $z=1$,} \\ 
             (r_{s^{-1}(j)s^{-1}(i)}, l_{s^{-1}(j)s^{-1}(i)}, h_{s^{-1}(j)s^{-1}(i)}) 
             &\text{if $z=-1$.} 
        \end{cases}
    \end{equation}

    Theorem \ref{main theorem}(1) asserts that $d(A,e)$ belongs to $\mathcal{S}$ for any finite dimensional $k$-algebra $(A,e)$ of rank $3$ which is $g$-convex. 
    Clearly, we have $d^g(A,e) = d(A^g,e^g)$ for $g\in G$. 
    In particular, the above group action restricts to that on 
    the subset $\mathcal{S}_{\rm alg}$ consisting of $d = d(A,e)\in \mathcal{S}$ for some $g$-convex algebra $(A,e)$ of rank $3$. 

    According to Theorem \ref{main theorem}(1), 
    the claim (2) is accomplished by giving 
    a complete set of representatives of $d(A,e)$'s in $\mathcal{S}_{\rm alg}$ up to the group action of $G$. 
    To do this, we narrow down $d\in \mathcal{S}$ with satisfying the following conditions for all $g\in G$. 
    \begin{itemize}
        \item $(d^g)_{+-+} = d(m)$ for some $m\in \{0,\ldots,13\}$, 
        \item $d^g$ satisfies all constraints in Proposition \ref{Gen A}. 
    \end{itemize}
    It yields the $66$ elements up to the group action of $G$, 
    where $61$ of them are listed in Table \ref{fig:61alg&poly} 
    and the remaining $5$ elements are given by 
    \begin{equation}\label{eq:dada X} 
        {\small
        \begin{bmatrix}  
        - & 211 & 110 \\  
        111 &  -  & 121 \\
        000 & 000 &  -  
        \end{bmatrix}
        }, 
        {\small
        \begin{bmatrix}  
        - & 211 & 210 \\  
        000 &  -  & 211 \\
        000 & 111 &  -  
        \end{bmatrix}
        }, 
        {\small
        \begin{bmatrix}  
        - & 211 & 110 \\  
        111 &  -  & 121 \\
        111 & 000 &  -  
        \end{bmatrix}
        },
        {\small
        \begin{bmatrix}  
        - & 211 & 110 \\  
        111 &  -  & 121 \\
        211 & 210 &  -  
        \end{bmatrix}
        }, 
        {\small
        \begin{bmatrix}  
        - & 211 & 210 \\  
        111 &  -  & 211 \\
        000 & 111 &  -  
        \end{bmatrix}
        }. 
    \end{equation}
    Indeed, we may use a computer to obtain this list. 
    By Propositions \ref{Gen A}, 
    every $d = d(A,e) \in \mathcal{S}_{\rm alg}$ appears in this list up to the group action of $G$. 
    Conversely, for each $d$ in Table \ref{fig:61alg&poly}, 
    we can find a $g$-convex algebra $(A,e)$ of rank $3$ such that $d = d(A,e)$, as in Table \ref{fig:61alg&poly}. 
    On the other hand, each datum in \eqref{eq:dada X} determines a sign-coherent fan appearing in Table \ref{tab:non-gfans} and does not belong to $\mathcal{S}_{\rm alg}$ by Proposition \ref{Gen X}. 
    
    Consequently, Table \ref{fig:61alg&poly} gives a complete set of representatives of $\mathcal{S}_{\rm alg}$ up to group action of $G$, as desired. 
    This finishes the proof of Theorem \ref{main theorem}(2). 
\end{proof}

\subsection{Proof of Proposition \ref{Gen A}}

In this section, we prove a series of claims of Proposition \ref{Gen A}. 
We notice that assumption of each claim (1)-(13) determines the fan $\Sigma_{+-+}(A,e)$ as in Theorem \ref{thm:orthant}(2).

\vspace{3mm} \noindent (0) 
This follows from an observation in Subsection \ref{case:0} that $d_{12} = d_{32} = (0,0,0)$ if and only if $h_{12} = h_{32}=0$.

\vspace{3mm} \noindent (1) 
By assumption, we have $d_{32}=(0,0,0)$. 
Then, the assertion follows from the previous case (0). 

\vspace{3mm} \noindent (2)
In this case, we have 
\begin{equation}
    \emin = \left[
    \begin{smallmatrix}
        1 \\ -2 \\ 0 
    \end{smallmatrix}\middle|
    \begin{smallmatrix}
        0 \\ -1 \\ 0 
    \end{smallmatrix}\middle|
    \begin{smallmatrix}
        0 \\ 0 \\ 1 
    \end{smallmatrix}
    \right]\in \Sigma_{+-+}(A,e). 
\end{equation}
By Propositions \ref{path p1} and \ref{path p3} with $(r_{12}, r_{32})=(2,0)$, 
the paths $\underline{\alpha}_{(-\bm{e}_2)}$ and $\underline{\beta}_{(-\bm{e}_2)}$ in \eqref{path pq} must be
\begin{eqnarray} 
    \text{\bf (i):} \quad &&
    \left(\underline{\emin} = 
    \left[ 
    \begin{smallmatrix}
        1 \\ -2 \\ 0 
    \end{smallmatrix}\middle|
    \begin{smallmatrix}
        0 \\ 0 \\ 1
    \end{smallmatrix}
    \right]
    \underset{0}{\xrightarrow{(0,\bullet)}}
    \underline{\tau_1} =
    \left[ 
    \begin{smallmatrix}
        1 \\ -2 \\ 0 
    \end{smallmatrix}\middle|
    \begin{smallmatrix}
        0 \\ 0 \\ -1
    \end{smallmatrix}
    \right]
    \underset{2}{\xrightarrow{(\bullet,0)}}
    \underline{\tau_2} = \underline{\sigma_-}\right), 
\\
    \text{\bf (i'):} \quad &&
    \left(\underline{\emin} = 
    \left[ 
    \begin{smallmatrix}
        1 \\ -2 \\ 0 
    \end{smallmatrix}\middle|
    \begin{smallmatrix}
        0 \\ 0 \\ 1
    \end{smallmatrix}
    \right]
    \underset{2}{\xrightarrow{(\bullet,0)}}
    \underline{\tau_1'} = 
    \left[ 
    \begin{smallmatrix}
        -1 \\ 0 \\ 0 
    \end{smallmatrix}\middle|
    \begin{smallmatrix}
        0 \\ 0 \\ 1
    \end{smallmatrix}
    \right]
    \underset{0}{\xrightarrow{(0,\bullet)}}
    \underline{\tau_2'} = \underline{\sigma_-}\right),
\end{eqnarray} 
respectively. 
In fact, it follows from Proposition \ref{path p1} (2)-(5) that the cases {\bf (ii)}, {\bf (iii)}, {\bf (iv)}, and {\bf (v)} can not appear for $\alpha$ by $(r_{12},r_{32}) = (2,0)$. 
Similarly, we have the case {\bf (i')} for $\beta$.

By Propositions \ref{path p1} (1) and \ref{path p3} (1'), we have $h_{13} = h_{31} = 0$. 
Then, $d_{32}=(0,0,0)$ and $h_{31}=0$ gives $d_{31}=(0,0,0)$ by the claim (0). 
Thus, we get the assertion.

\vspace{3mm} \noindent (3) 
This is done by the same argument as the case (1) since we have $d_{32}=0$. 

\vspace{3mm} \noindent (4) 
By Proposition \ref{path p3} with $(r_{12},r_{32})=(1,1)$ and $h_{32}=0$, the path $\underline{\beta}_{(-\bm{e}_2)}$ in \eqref{path pq} is of the form 
\begin{equation} 
    \text{\bf (iii'):} \quad 
    \left(\underline{\emin} = 
    \left[ 
    \begin{smallmatrix}
        1 \\ -1 \\ 0 
    \end{smallmatrix}\middle|
    \begin{smallmatrix}
        0 \\ -1 \\ 1
    \end{smallmatrix}
    \right]
    \underset{0}{\xrightarrow{(\bullet,1)}}
    \underline{\tau_1} \underset{1}{\xrightarrow{(1,\bullet)}}
    \underline{\tau_2} = 
    \left[ 
    \begin{smallmatrix}
        -1 \\ 0 \\ 1 
    \end{smallmatrix}\middle|
    \begin{smallmatrix}
        -1 \\ 0 \\ 0
    \end{smallmatrix}
    \right]
    \underset{0}{\xrightarrow{(\bullet,1)}}
    \underline{\tau_3} = \underline{\sigma_-}\right). 
\end{equation}
Then, $\tau_2 \in \SigmaAe$ implies that $h_{31}=1$ and $r_{31}=1$. Thus, we have $d_{31}\in \{(2,1,1),(1,1,1)\}$. 

\vspace{3mm} \noindent (5) 
By Proposition \ref{path p1} with $(r_{12},r_{32})=(1,2)$ and $h_{12}=1$, the path $\underline{\alpha}_{(-\bm{e}_2)}$ in \eqref{path pq} is given by 
\begin{equation} 
    \text{\bf (i):} \quad 
    \left(\underline{\emin} = 
    \left[ 
    \begin{smallmatrix}
        1 \\ -1 \\ 0 
    \end{smallmatrix}\middle|
    \begin{smallmatrix}
        0 \\ -2 \\ 1
    \end{smallmatrix}
    \right]
    \underset{2}{\xrightarrow{(0,\bullet)}}
    \underline{\tau_1} =
    \left[ 
    \begin{smallmatrix}
        1 \\ -1 \\ 0 
    \end{smallmatrix}\middle|
    \begin{smallmatrix}
        0 \\ 0 \\ -1
    \end{smallmatrix}
    \right]
    \underset{1}{\xrightarrow{(\bullet,0)}}
    \underline{\tau_2} = \underline{\sigma_-}\right)
\end{equation}
with $h_{13}=0$. 
On the other hand, by Proposition \ref{path p3} with $r_{12}=1$, $r_{32}=2$ and $h_{32}=0$, 
the path $\underline{\beta}_{(-\bm{e}_2)}$ in \eqref{path pq} is of the form 
\begin{equation} 
    \text{\bf (v'):} \quad 
    \left(\underline{\emin} = 
    \left[ 
    \begin{smallmatrix}
        1 \\ -1 \\ 0 
    \end{smallmatrix}\middle|
    \begin{smallmatrix}
        0 \\ -2 \\ 1
    \end{smallmatrix}
    \right]
    \underset{0}{\xrightarrow{(\bullet,1)}}
    \underline{\tau_1'} \underset{0}{\xrightarrow{(2,\bullet)}}  
    \underline{\tau_2'} \underset{1}{\xrightarrow{(\bullet,1)}}  
    \underline{\tau_3'} =
    \left[ 
    \begin{smallmatrix}
        -1 \\ 0 \\ 0 
    \end{smallmatrix}\middle|
    \begin{smallmatrix}
        -2 \\ 0 \\ 1
    \end{smallmatrix}
    \right]
    \underset{0}{\xrightarrow{(2,\bullet)}}  
    \underline{\tau_4'}
    = \underline{\sigma_-}\right).
\end{equation}
Then, $\tau_3'\in \SigmaAe$ implies that $h_{31}=1$ and $r_{31}=2$. That is, $d_{31}=(1,2,1)$ as desired. 

\vspace{3mm}\noindent (6)
We get the assertion by the same argument as (4).

\vspace{3mm}\noindent (7)
In this case, we have $\tau, \tau'\in \SigmaeAe$ in \eqref{B2 tautau'} with $\bm{v}_{\tau}=\begin{bsmallmatrix}
    0 \\ 0 \\ 1
\end{bsmallmatrix}$ and $\bm{v}_{\tau'}= \begin{bsmallmatrix}
    1 \\ 0 \\ 0
\end{bsmallmatrix}$. 
They imply that 
$\begin{bsmallmatrix}
    2 \\ 0 \\ -1
\end{bsmallmatrix}, 
\begin{bsmallmatrix}
    -1 \\ 0 \\ 2
\end{bsmallmatrix}\not\in \Sigma(A,e)$
since 
$\langle \bm{v}_{\tau'}, \begin{bsmallmatrix}
    2 \\ 0 \\ -1
\end{bsmallmatrix}\rangle=\langle \bm{v}_{\tau}, \begin{bsmallmatrix}
    -1 \\ 0 \\ 2
\end{bsmallmatrix}\rangle = 2 \not \leq 1$. 
Thus, we have $l_{13}\neq 2$ and $l_{31}\neq 2$. That is, $d_{13},d_{31}\not\in \{(2,1,1),(2,1,0)\}$ as desired. 

\vspace{3mm}\noindent (8)
In this case, we have $\tau \in \SigmaeAe$ in \eqref{B2 tautau'} with $\bm{v}_{\tau}=\begin{bsmallmatrix}
    0 \\ 0 \\ 1
\end{bsmallmatrix}$. 
From the same discussion as in the previous case (7), we have 
$\begin{bsmallmatrix}
    -1 \\ 0 \\ 2
\end{bsmallmatrix} \not\in\SigmaAe$ and $d_{31}\not\in \{(2,1,1),(2,1,0)\}$. 
On the other hand, 
from a description of $\SigmaeAe=\Sigma_{d(8)}$, 
we have the following path $p$ around $\bm{e}_1=\begin{bsmallmatrix}
    1 \\ 0 \\ 0
\end{bsmallmatrix}$. 
\begin{equation} \label{path pe1}
    \underline{p}_{\bm{e}_1} \colon \ 
    \left(\underline{\sigma_{+}} = 
    \left[ 
    \begin{smallmatrix}
        0 \\ 1 \\ 0 
    \end{smallmatrix}\middle|
    \begin{smallmatrix}
        0 \\ 0 \\ 1
    \end{smallmatrix}
    \right]
    \underset{1}{\xrightarrow{(\bullet,1)}}
    \underline{\emax} = 
    \left[ 
    \begin{smallmatrix}
        1 \\ -1 \\ 1 
    \end{smallmatrix}\middle|
    \begin{smallmatrix}
        0 \\ 0 \\ 1
    \end{smallmatrix}
    \right]
    \underset{1}{\xrightarrow{(1,\bullet)}}
    \left[ 
    \begin{smallmatrix}
        1 \\ -1 \\ 1 
    \end{smallmatrix}\middle|
    \begin{smallmatrix}
        2 \\ -1 \\ 0
    \end{smallmatrix}
    \right]
    \underset{t}{\xrightarrow{(\bullet,1)}}
    \underline{\rho} := 
    \left[ 
    \begin{smallmatrix}
        t+1 \\ 0 \\ -1 
    \end{smallmatrix}\middle|
    \begin{smallmatrix}
        2 \\ -1 \\ 0
    \end{smallmatrix}
    \right]\right), 
\end{equation}
where $t\in \{0,1\}$. 
In particular, we have either $\begin{bsmallmatrix}
    1 \\ 0 \\ -1
\end{bsmallmatrix}\in \SigmaAe$ or $\begin{bsmallmatrix}
    2 \\ 0 \\ -1
\end{bsmallmatrix}\in \SigmaAe$. 
In both cases, we have $d_{13}\neq (0,0,0)$.

\vspace{3mm}\noindent (9)
In this case, we have a path \eqref{path pe1} in the fan $\SigmaeAe=\Sigma_{d(9)}$. 
Thus, we have $d_{13}\neq (0,0,0)$ by the same argument as a proof of (8) above. 
By permuting a role of $\bm{e}_1$ and $\bm{e}_3$, we also have $d_{31}\neq (0,0,0)$.

\vspace{3mm}\noindent (10)
Assume that $d_{12} = (1,2,1)$ and $d_{32}=(1,1,0)$. 
By Proposition \ref{path p3} with $(r_{12},r_{32})=(2,1)$ and $h_{32}=0$, we have $h_{31}=1$ and the path $\underline{\beta}_{(-\bm{e}_2)}$ in \eqref{path pq} is either \text{\bf (iii')} or \text{\bf (iv')} given as follows. 
\begin{eqnarray}
\text{\bf (iii'):} && 
\left(\underline{\emin} =
\left[ 
    \begin{smallmatrix}
        1 \\ -2 \\ 0 
    \end{smallmatrix}\middle|
    \begin{smallmatrix}
        0 \\ -1 \\ 1
    \end{smallmatrix}
    \right]\underset{1}{\xrightarrow{(\bullet,1)}}
    \underline{\tau_1'} \underset{1}{\xrightarrow{(1,\bullet)}} 
    \underline{\tau_2'} = 
    \left[ 
    \begin{smallmatrix}
        -1 \\ 0 \\ 1 
    \end{smallmatrix}\middle|
    \begin{smallmatrix}
        -1 \\ 0 \\ 0
    \end{smallmatrix}
    \right]
    \underset{0}{\xrightarrow{(\bullet,1)}} 
    \underline{\tau_3'} = \underline{\sigma_-} \right).   \\ 
\text{\bf (iv'):} && 
\left(\underline{\emin} =
    \left[ 
    \begin{smallmatrix}
        1 \\ -2 \\ 0 
    \end{smallmatrix}\middle|
    \begin{smallmatrix}
        0 \\ -1 \\ 1
    \end{smallmatrix}
    \right]
\underset{0}{\xrightarrow{(\bullet,2)}}
    \underline{\tau_1'} \underset{1}{\xrightarrow{(1,\bullet)}} 
    \underline{\tau_2'} \underset{0}{\xrightarrow{(\bullet,2)}}
    \underline{\tau_3'} = 
    \left[
    \begin{smallmatrix}
        -1 \\ 0 \\ 0 
    \end{smallmatrix}\middle|
    \begin{smallmatrix}
        -1 \\ 0 \\ 1
    \end{smallmatrix}
    \right]
    \underset{0}{\xrightarrow{(1,\bullet)}}
    \underline{\tau_4'} = \underline{\sigma_-}\right). 
\end{eqnarray}
In both cases, we have 
$\left[ 
    \begin{smallmatrix}
        -1 \\ 0 \\ 1 
    \end{smallmatrix}\middle|
    \begin{smallmatrix}
        -1 \\ 0 \\ 0
    \end{smallmatrix}
    \right]\in \SigmaAe$. 
This implies that $r_{31}=1$. 
Thus, we have $d_{31}\in \{(2,1,1),(1,1,1)\}$. 
In particular $h_{31} = 1$. 
Further, we show that $d_{31}=(2,1,1)$ if $h_{13}=0$, as follows. 
From a description of the fan $\SigmaeAe = \Sigma_{d(10),0}$, there is the following path $p$ around 
$\bm{u} = \begin{bsmallmatrix}
        0 \\ -1 \\ 1
    \end{bsmallmatrix}$. 
    \begin{equation*}
    \underline{p}_{\bm{u}} = 
    \left(\underline{\rho} = 
    \left[ 
    \begin{smallmatrix}
        1 \\ -2 \\ 0 
    \end{smallmatrix}\middle|
    \begin{smallmatrix}
        0 \\ 0 \\ 1
    \end{smallmatrix}
    \right] \underset{1}{\xrightarrow{(0,\bullet)}}
    \underline{\emin} = 
    \left[ 
    \begin{smallmatrix}
        1 \\ -2 \\ 0 
    \end{smallmatrix}\middle|
    \begin{smallmatrix}
        0 \\ -1 \\ 0
    \end{smallmatrix}
    \right] \underset{c_2}{\xrightarrow{(\bullet,c_1)}}
    \underline{\tau_1'} = 
    \left[ 
    \begin{smallmatrix}
        -1 \\ -c_1-c_2+2 \\ c_2 
    \end{smallmatrix}\middle|
    \begin{smallmatrix}
        0 \\ -1 \\ 0
    \end{smallmatrix}
    \right]\right). 
    \end{equation*}
    By Proposition \ref{rank2}, we have $c_1=0$. 
    If $c_2 \in \{0,1\}$, then we have 
    $\left[\begin{smallmatrix}
        -1 \\ 1 \\ 1
    \end{smallmatrix}\middle|
    \begin{smallmatrix}
        0 \\ -1 \\ 0
    \end{smallmatrix}
    \right]\in \Sigma(A,e)$, which is a contradiction by the sign-coherent property. 
    Thus, we have $c_2 =2$ and 
    \begin{equation}
    \tau_1' = 
    \left[ 
    \begin{smallmatrix}
        -1 \\ 0 \\ 2 
    \end{smallmatrix}\middle|
    \begin{smallmatrix}
        0 \\ -1 \\ 0
    \end{smallmatrix}\middle|
    \begin{smallmatrix}
        0 \\ -1 \\ 1 
    \end{smallmatrix}
    \right]\in \Sigma(A,e).
    \end{equation}
    In particular, we have $l_{31}=2$. 
    Consequently, we have $d_{31}=(2,1,1)$ as desired.

\vspace{3mm}\noindent (11)
By Proposition \ref{path p3} with $(r_{12},r_{32})=(2,2)$ and $h_{32}=0$, we have $h_{31}=1$ and the path $\underline{\beta}_{(-\bm{e}_2)}$ in \eqref{path pq} is given by  
\begin{eqnarray}
\text{\bf (v'):} &&
\left(\underline{\emin} =
    \left[ 
    \begin{smallmatrix}
        1 \\ -1 \\ 0 
    \end{smallmatrix}\middle|
    \begin{smallmatrix}
        0 \\ -1 \\ 1
    \end{smallmatrix}
    \right]
\underset{1}{\xrightarrow{(\bullet,1)}}
    \underline{\tau_1'} \underset{0}{\xrightarrow{(2,\bullet)}} 
    \underline{\tau_2'} \underset{1}{\xrightarrow{(\bullet,1)}}
    \underline{\tau_3'} = 
    \left[
    \begin{smallmatrix}
        -1 \\ 0 \\ 0 
    \end{smallmatrix}\middle|
    \begin{smallmatrix}
        -2 \\ 0 \\ 1
    \end{smallmatrix}
    \right]
    \underset{0}{\xrightarrow{(2,\bullet)}}
    \underline{\tau_4'} = \underline{\sigma_-}\right). 
\end{eqnarray}
Then, $\tau'_3\in \SigmaAe$ implies $r_{31}=2$. That is, $d_{31}=(1,2,1)$. 
Similarly, by Proposition \ref{path p1}, since $(r_{12},r_{32})=(2,2)$ and $h_{12}=1\neq 0$, the maximal path $\underline{\alpha}_{(-\bm{e}_2)}$ is one of {\bf (i)}, {\bf (iii)} or {\bf (iv)} as follows. 
\begin{eqnarray}    
\text{\bf (i):} && 
\left(\underline{\emin} =
\left[ 
    \begin{smallmatrix}
        1 \\ -2 \\ 0 
    \end{smallmatrix}\middle|
    \begin{smallmatrix}
        0 \\ -2 \\ 1
    \end{smallmatrix}
    \right]\underset{2}{\xrightarrow{(0,\bullet)}}
    \underline{\tau_1} = 
    \left[ 
    \begin{smallmatrix}
        1 \\ -2 \\ 0 
    \end{smallmatrix}\middle|
    \begin{smallmatrix}
        0 \\ 0 \\ -1
    \end{smallmatrix}
    \right]
    \underset{2}{\xrightarrow{(\bullet,0)}} 
    \underline{\tau_2} = \underline{\sigma_-}\right). \\ 
\text{\bf (iii):} && 
\left(\underline{\emin} =
\left[ 
    \begin{smallmatrix}
        1 \\ -2 \\ 0 
    \end{smallmatrix}\middle|
    \begin{smallmatrix}
        0 \\ -2 \\ 1
    \end{smallmatrix}
    \right]\underset{1}{\xrightarrow{(1,\bullet)}}
    \underline{\tau_1} = 
    \left[ 
    \begin{smallmatrix}
        1 \\ -2 \\ 0 
    \end{smallmatrix}\middle|
    \begin{smallmatrix}
        1 \\ -1 \\ -1
    \end{smallmatrix}
    \right]
    \underset{1}{\xrightarrow{(\bullet,1)}} 
    \underline{\tau_2} 
    \underset{1}{\xrightarrow{(1,\bullet)}} 
    \underline{\tau_3} = \underline{\sigma_-}\right).  \\ 
\text{\bf (iv):} && 
\left(\underline{\emin} =
    \left[ 
    \begin{smallmatrix}
        1 \\ -2 \\ 0 
    \end{smallmatrix}\middle|
    \begin{smallmatrix}
        0 \\ -2 \\ 1
    \end{smallmatrix}
    \right]
\underset{0}{\xrightarrow{(2,\bullet)}}
    \underline{\tau_1} = 
    \left[ 
    \begin{smallmatrix}
        1 \\ -2 \\ 0 
    \end{smallmatrix}\middle|
    \begin{smallmatrix}
        2 \\ -2 \\ -1
    \end{smallmatrix}
    \right]
    \underset{1}{\xrightarrow{(\bullet,1)}} 
    \underline{\tau_2} \underset{0}{\xrightarrow{(2,\bullet)}}
    \underline{\tau_3} 
    \underset{1}{\xrightarrow{(\bullet,1)}}
    \underline{\tau_4} = \underline{\sigma_-}\right). 
\end{eqnarray}
On the other hand, from the description of $\SigmaeAe=\Sigma_{d(11)}$, there is the following path $q$ around $\bm{u}:=\begin{bsmallmatrix}
    1 \\ -2 \\ 0
\end{bsmallmatrix}$.
\begin{equation*}
    \underline{q}_{\bm{u}} \colon\ 
    \left(
    \underline{\rho} = 
    \left[ 
    \begin{smallmatrix}
        1 \\ -1 \\ 0 
    \end{smallmatrix}\middle|
    \begin{smallmatrix}
        0 \\ -2 \\ 1
    \end{smallmatrix}
    \right] \underset{1}{\xrightarrow{(\bullet,0)}}
    \underline{\emin} = 
    \left[ 
    \begin{smallmatrix}
        0 \\ -1 \\ 0 
    \end{smallmatrix}\middle|
    \begin{smallmatrix}
        0 \\ -2 \\ 1
    \end{smallmatrix}
    \right] \underset{b_2}{\xrightarrow{(b_1,\bullet)}}
    \underline{\eta} := 
    \left[ 
    \begin{smallmatrix}
        0 \\ -1 \\ 0 
    \end{smallmatrix}\middle|
    \begin{smallmatrix}
        b_2 \\ -b_1-2b_2+2 \\ -1
    \end{smallmatrix}
    \right]\right). 
    \end{equation*}
where $\eta$ coincides with $\tau_1$.
Since $b_1=0$ by Proposition \ref{rank2}, 
we conclude that the path {\bf (iv)} appears with $(b_1,b_2)=(0,2)$.
Furthermore, we have
\[
\tau_2=\left[ 
    \begin{smallmatrix}
        1 \\ -1 \\ -1 
    \end{smallmatrix}\middle|
    \begin{smallmatrix}
        2 \\ -2 \\ -1
    \end{smallmatrix}\middle|
    \begin{smallmatrix}
        0 \\ -1 \\ 0
    \end{smallmatrix}
    \right],\ \tau_3=\left[ 
    \begin{smallmatrix}
        1 \\ -1 \\ -1 
    \end{smallmatrix}\middle|
    \begin{smallmatrix}
        0 \\ 0 \\ -1
    \end{smallmatrix}\middle|
    \begin{smallmatrix}
        0 \\ -1 \\ 0
    \end{smallmatrix}
    \right].\ 
\]
This shows $h_{13}=1$.


\vspace{3mm}\noindent (12) and (13) 
In our proofs in \ref{case:12} and \ref{case:13} in Section \ref{sec:proof_Thm2}, we have already shown that 
the paths $\underline{\alpha}_{(-\bm{e}_2)}$ and $\underline{\beta}_{(-\bm{e}_2)}$ in \eqref{path pq} are given by \eqref{D2 p1} and \eqref{D2 p3} respectively. 
In particular, they give $h_{13}=h_{31}=0$ by Proposition\;\ref{path p1}(1) and Proposition\;\ref{path p3}(1') respectively. 
On the other hand, we have a maximal cone $\kappa\in \SigmaAe$ in \eqref{D2 kappaphi} with $\bm{v}_{\kappa} = \begin{bsmallmatrix}
    0 \\ 0 \\ 1
\end{bsmallmatrix}$ in both cases.
Then, we have $\begin{bsmallmatrix}
    -1 \\ 0 \\ 2
\end{bsmallmatrix}\not\in \SigmaAe$ because $\langle v_{\kappa}, \begin{bsmallmatrix}
    -1 \\ 0 \\ 2
\end{bsmallmatrix} \rangle = 2 \not \leq 1$. 
Therefore, we have $l_{31}\neq 2$ and $d_{31}\in \{(1,1,0),(1,2,0),(0,0,0)\}$. 

If this is the case (12), 
then the fan $\SigmaeAe=\Sigma_{d(12)}$ has the following path $p$ around $\bm{e}_1=\begin{bsmallmatrix}
    1 \\ 0 \\ 0
\end{bsmallmatrix}$. 
\begin{eqnarray}
    \underline{p}_{\bm{e}_1} \colon\ \left(
    \underline{\sigma_{+}} = 
\left[ 
    \begin{smallmatrix}
        0 \\ 1 \\ 0 
    \end{smallmatrix}\middle|
    \begin{smallmatrix}
        0 \\ 0 \\ 1
    \end{smallmatrix}
    \right] 
\underset{1}{\xrightarrow{(\bullet,1)}}
\underline{\emax} =
\left[ 
    \begin{smallmatrix}
        1 \\ -1 \\ 1 
    \end{smallmatrix}\middle|
    \begin{smallmatrix}
        0 \\ 0 \\ 1
    \end{smallmatrix}
    \right] \underset{0}{\xrightarrow{(2,\bullet)}}
    \underline{\psi_1'} \underset{0}{\xrightarrow{(\bullet,1)}} 
    \underline{\psi_2'}
    \underset{0}{\xrightarrow{(\bullet,2)}} 
    \underline{\psi_3'} = 
    \left[
    \begin{smallmatrix}
        1 \\ -1 \\ 0 
    \end{smallmatrix}\middle|
    \begin{smallmatrix}
        0 \\ 0 \\ -1
    \end{smallmatrix}
    \right]\right). 
\end{eqnarray}
Then, $\left[
    \begin{smallmatrix}
        1 \\ 0 \\ 0 
    \end{smallmatrix}\middle|
    \begin{smallmatrix}
        0 \\ 0 \\ -1
    \end{smallmatrix}
    \right]\in \SigmaAe$, and it implies $d_{13}=0$ as desired.

On the other hand, if this is the case (13), 
then the fan $\SigmaeAe=\Sigma_{d(13)}$ has the following path $q$ around $\bm{e}_1$. 
\begin{eqnarray} 
\underline{q}_{\bm{e}_1} \colon\ 
\left(\underline{\sigma_{+}} = 
\left[ 
    \begin{smallmatrix}
        0 \\ 1 \\ 0 
    \end{smallmatrix}\middle|
    \begin{smallmatrix}
        0 \\ 0 \\ 1
    \end{smallmatrix}
    \right] 
\underset{1}{\xrightarrow{(\bullet,1)}}
\underline{\emax} =
\left[ 
    \begin{smallmatrix}
        1 \\ -1 \\ 1 
    \end{smallmatrix}\middle|
    \begin{smallmatrix}
        0 \\ 0 \\ 1
    \end{smallmatrix}
    \right] \underset{0}{\xrightarrow{(2,\bullet)}}
    \underline{\psi_1} \underset{1}{\xrightarrow{(\bullet,1)}} 
    \underline{\psi_2}
    \underset{0}{\xrightarrow{(\bullet,2)}} 
    \underline{\psi_3} = \left[
    \begin{smallmatrix}
        2 \\ -1 \\ 0 
    \end{smallmatrix}\middle|
    \begin{smallmatrix}
        2 \\ 0 \\ -1
    \end{smallmatrix}
    \right] \right)
\end{eqnarray}
Then, $\begin{bsmallmatrix}
    2 \\ 0 \\ -1
\end{bsmallmatrix}\in \SigmaAe$ implies $l_{13}=2$. 
Since $h_{13}=0$ as we shown before, we have $d_{13}=(2,1,0)$ in this case. 
We finish a proof.

\subsection{Proof of Proposition \ref{Gen X}}
Finally, we prove Proposition \ref{Gen X}. 
In Section \ref{proof X1} (resp., Section \ref{proof X2}), we will show that contradictions arise when we assume that case (i) (resp., (ii)) occurs.
Our proof will employ a ring-theoretic approach.

Let $(A,e=(e_1,e_2,e_3))$ be a $g$-convex algebra of rank $3$. 
We begin with preparation.
Fix a pair $(i,j)$ of integers with $1\leq i\neq j \leq 3$. 
Let $\mathbb{U}_{ij}:=\cone\{\bm{e}_i,-\bm{e}_j\}$. 
In Proposition\;\ref{coord_plane}, we give rays of $\SigmaeAe$ lying in $\mathbb{U}_{ij}$. 
Here, we give more detailed information on this proposition.

We denote by $M_{s,t}(e_i A e_j)$ the set of $s\times t$ matrices whose entries are elements of $e_i A e_j$. 
To each $\bm{x}\in M_{s,t}(e_i A e_j)$, we associate a two-term complex   
\[
P_{\bm{x}} := [P_j^{\oplus t}\xrightarrow{\bm{x}\cdot-} P_i^{\oplus s}]\in \Kb(\proj A). 
\]

On the other hand, we recall from \cite[Section 5]{AHIKM2} that each element $x\in e_iAe_j$ defines a subalgebra $L_x$ of $e_i A e_i$ and a subalgebra $R_x$ of $e_jAe_j$ as follows:
\[
\begin{array}{lll}
  L_x   & := & \{a \in e_i A e_i\mid ax \in x A e_j\}, \\
  R_x   & := & \{a \in e_j A e_j\mid xa \in e_i A x\}. \\
\end{array}
\]
We say that $x\in e_i A e_j$ is a \emph{left generator} (resp. \emph{right generator}) of $e_i A e_j$ 
if $e_iAx = e_iAe_j$ (resp. $xAe_j = e_iAe_j$). 
Notice that if $(l_{ij},r_{ij})=(1,1)$, then
each left (resp. right) generator is also a right (resp. left) generator.
Conversely, if there exists a left-right generator of $e_iAe_j$, then we have $(l_{ij},r_{ij})=(1,1)$.

We can deduce the following statements from \cite[Proposition 5.6]{AHIKM2} (and its proof). 

\begin{proposition}
\label{prop: lr=12}
Assume that $(l_{12},r_{12})=(1,2)$. 
For a left generator $x\in e_1Ae_2$, there is $u\in e_1Ae_1\setminus L_x$ satisfying the following equivalent conditions {\upshape (a)}, {\upshape (b)} and {\upshape (c)}.
       \begin{enumerate}[\rm (a)] 
       \item $P_x\oplus P_{[x\ ux]}$ is presilting.
       \item The following equation holds.
       \[
        e_1 A e_1= L_x+L_xu.
       \]
       \item The following equation holds.
       \[
        M_{1,2}(e_1 A e_2) = M_{1,1}(e_1 Ae_1)[x\ ux]+M_{1,2}(xAe_2).
       \]
       \end{enumerate}
\end{proposition}

\begin{proposition}
	\label{dual of prop: lr=12}
Assume that $(l_{12},r_{12})=(2,1)$. 
For a right generator $x\in e_1Ae_2$, there is $u\in e_2Ae_2\setminus R_x$ satisfying the following equivalent conditions {\upshape (a)}, {\upshape (b)} and {\upshape (c)}.
       \begin{enumerate}[\rm (a)] 
       \item $P_x\oplus P_{\begin{bsmallmatrix}
           x \\ xu
       \end{bsmallmatrix}}$ is presilting.
       \item The following equation holds.
       \[
        e_2 A e_2 = R_x+uR_x.
       \]
       \item The following equation holds.
       \[
        M_{2,1}(e_1 A e_2) = \begin{bsmallmatrix}
           x \\ xu
       \end{bsmallmatrix} M_{1,1}(e_2Ae_2) + M_{2,1}(e_1Ax).
       \]
       \end{enumerate}
\end{proposition}

\begin{proof}[Proof of Propositions \ref{prop: lr=12} and \ref{dual of prop: lr=12}] 
We prove Proposition \ref{prop: lr=12}. 
Since $(l_{12},r_{12})=(1,2)$, we have 
$\left[\begin{smallmatrix}
    1 \\ -1 \\ 0 
\end{smallmatrix}\middle|\begin{smallmatrix}
    1 \\ -2 \\ 0 
\end{smallmatrix}\right] \in \Sigma(A,e)$.
It follows from \cite[Proposition 5.6]{AHIKM2} that there exists an element $u\in e_1Ae_1\setminus L_x$ satisfying the condition (c). Therefore, it is sufficient to check that
 conditions (a), (b) and (c) are equivalent.

 
(b)$\Rightarrow$(c). This implication holds by \cite[Proposition 5.6]{AHIKM2}.

(c)$\Rightarrow$(b). For each $a\in e_1Ae_1$, there exist $a'\in e_1 A e_1$ and $b,b'\in e_2 A e_2$ such that 
	\[
 	[0\ ax]= a'[x\ ux]+x[b\ b'].
 	\]
 	Then $a'$ and $a-a'u$ are in $L_x$, and hence $a=a'u+(a-a'u)\in L_xu+L_x$. 
  Thus the equation $e_1 A e_1 = L_x+L_x u$ holds as desired. 

(a)$\Leftrightarrow$(c). 
By \cite[Proposition\;3.9]{AHIKM2}, the following assertions hold.
\begin{itemize}
\item $P_x$ is presilting if and only if (i) $e_1Ax+xAe_2=e_1Ae_2$.
\item $P_{[x\ ux]}$ is presilting if and only if (ii) $e_1A[x\ ux]+[x\ ux]M_{2,2}(e_2Ae_2)=M_{1,2}(e_1 A e_2)$.
\item $\Hom_{\Kb(\proj A)}(P_x, P_{[x\ ux]}[1])=0$ if and only if (iii) $e_1 A x+[x\ ux]M_{2,1}(e_2Ae_2)=e_1 A e_2$.
\item $\Hom_{\Kb(\proj A)}(P_{[x\ ux]}, P_x[1])=0$ if and only if (iv) $e_1 A [x\ ux]+xM_{1,2}(e_2Ae_2)=M_{1,2}(e_1 A e_2)$.
\end{itemize}
It is clear that (iv) implies (ii), and (i) implies (iii). By looking at the first entry of the row vector, (iv) implies (i).
Thus the condition (a) is equivalent to the condition (c).
\end{proof}

\subsubsection{Proof of Proposition \ref{Gen X}(i)} \label{proof X1}
We prove Proposition \ref{Gen X}~(i). 
We recall our assumption  
\begin{enumerate}
    \item[\rm (i)] $d_{12} = (2,1,1)$, $d_{13} = (2,1,0)$, 
    $d_{23} = (2,1,1)$ and $d_{32} = (1,1,1)$ 
\end{enumerate}
in the statement. 
Since $r_{12}=r_{23} = r_{13} = 1$, 
we have right generators $x\in e_1Ae_2$, $y\in e_2Ae_3$ and $z\in e_1Ae_3$. In particular, we can choose $z = xy$ as a right generator of $e_1Ae_3$. 
By $(l_{12},r_{12})=(2,1)$, we apply Proposition \ref{dual of prop: lr=12} to a right generator $x$ of $e_1Ae_2$ and obtain $u\in e_2 A e_2\setminus R_x$ such that 
$P_x\oplus P_{\left[\begin{smallmatrix}
    x \\ xu  
\end{smallmatrix}\right]}$ 
is presilting. Similarly, by $(l_{13},r_{13})=(2,1)$, we apply Proposition \ref{dual of prop: lr=12} to a right generator $z=xy$ of $e_1Ae_3$ and obtain $v\in e_3Ae_3$ such that $P_{xy}\oplus P_{\left[\begin{smallmatrix}
    xy \\ xyv  
\end{smallmatrix}\right]}$ is presilting. 
These indecomposable two-term presilting complexes are given by 

\begin{equation*}
    P_x = \left[P_2\xrightarrow{x}P_1 \right], \ 
    P_{\left[\begin{smallmatrix}
    x \\ xu  
\end{smallmatrix}\right]} = \left[P_2\xrightarrow{\left[\begin{smallmatrix}
    x \\ xu  
\end{smallmatrix}\right]} P_1^{\oplus 2} \right] \ 
\text{and} \ 
    P_{\left[\begin{smallmatrix}
    xy \\ xyv  
\end{smallmatrix}\right]} = \left[P_3\xrightarrow{\left[\begin{smallmatrix}
    xy \\ xyv 
\end{smallmatrix}\right]} P_1^{\oplus 2} \right]
\end{equation*}
with $g$-vectors 
\begin{equation}    
g(P_x) = \left[\begin{smallmatrix}
    1 \\ -1 \\ 0
\end{smallmatrix}\right], \quad
g(P_{\left[\begin{smallmatrix}
    x \\ xu  
\end{smallmatrix}\right]}) = \left[\begin{smallmatrix}
    2 \\ -1 \\ 0
\end{smallmatrix}\right]\quad  \text{and} \quad  
g(P_{\left[\begin{smallmatrix}
    xy \\ xyv  
\end{smallmatrix}\right]}) = \left[\begin{smallmatrix}
    2 \\ 0 \\ -1
\end{smallmatrix}\right]. 
\end{equation}

\begin{lemma}\label{some presilt}
In the above,  
$P_x \oplus P_{\left[\begin{smallmatrix}
    xy \\ xyv  
\end{smallmatrix}\right]}$ is presilting. In particular, we have  
\begin{equation}\label{M21 i}
    M_{2,1}(e_1Ae_2) = 
    \begin{bmatrix}
    xy \\ xyv
\end{bmatrix}M_{1,1}(e_3Ae_2) + M_{2,1}(e_1Ax).
\end{equation}

\end{lemma}

\begin{proof}
We focus on the signature $(+--)\in \{\pm\}^3$. 
An element $g=(s,z) = ((12),-1)\in G$ satisfies $g \cdot (+--) = (+-+)$. 
By our assumption $d_{12} = (2,1,1)$ and $d_{13} = (2,1,0)$, we have  
$(d(A^g,e^g))_{+-+} = d(11)$ and $\Sigma_{+--}(A,e) = g^{-1} \cdot \Sigma_{+-+}(A^g,e^g) = g^{-1}\cdot \Sigma_{d(11)}$. 
Since 
$\left[
    \begin{smallmatrix}
        1 \\ -1 \\ 0 
    \end{smallmatrix}\middle|
    \begin{smallmatrix}
        0 \\ -2 \\ 1
    \end{smallmatrix}
    \right]\in\Sigma_{d(11)}$, 
we obtain a cone 
\begin{equation}
    \left[
    \begin{smallmatrix}
        1 \\ -1 \\ 0 
    \end{smallmatrix}\middle|
    \begin{smallmatrix}
        2 \\ 0 \\ -1
    \end{smallmatrix}
    \right]\in \Sigma_{+-+}(A,e). 
\end{equation}
The corresponding two-term presilting complex for $A$ must be isomorphic to $P_x\oplus P_{\left[\begin{smallmatrix}
    xy \\ xyv  
\end{smallmatrix}\right]}$ since the map taking its $g$-vector is injective for two-term presilting complexes (see \cite[Theorem\;6.5]{DIJ}).
\end{proof}

We reach to a contradiction by the next claim.

\begin{lemma}
	\label{lemma:remaining case (a)}
Under the above settings, we obtain the following contradictory claims. 
\begin{enumerate}[\rm (a)]
	\item Let $y^*$ be a left-right generator of $e_3 A e_2$. Then there is an invertible element $\varepsilon$ of $e_2 A e_2$ satisfying
	\[
	yy^*\varepsilon\in R_x.
	\]
 In this case, $y^*\varepsilon$ is a left-right generator of $e_3Ae_2$. 
	\item For each left-right generator $y^*$ of $e_3 A e_2$, we have 
	\[
	yy^*\not \in R_x. 
	\]
\end{enumerate}
\end{lemma}

In fact, they cause a contradiction by the following reason. 
Since $l_{32}=r_{32}=1$, there exists a left-right generator, say $y^*$, of $e_3Ae_2$. 
Then, $yy^*\varepsilon \in R_x$ by (a) but $yy^*\varepsilon\not \in R_x$ by (b) for some invertible element $\varepsilon$.

\begin{proof}[Proof of Lemma \ref{lemma:remaining case (a)}]
Let $x,y,z,u,v$ be as above. 
In addition, let $y^*$ be a left-right generator of $e_3 A e_2$. 
As a right generator $y^*$, we have $e_3Ae_2 = y^*Ae_2$.
In this case, there is $v'\in e_2Ae_2$ such that $vy^*=y^*v'$. By \eqref{M21 i}, for an element $\begin{bsmallmatrix}
     0 \\ xu
 \end{bsmallmatrix}\in M_{2,1}(e_1Ae_2)$, 
 	there exists $\ell\in e_2 A e_2$ such that
 	\[
 	\begin{bmatrix}
 	    0 \\ xu
 	\end{bmatrix} - 
  \begin{bmatrix}
 	    xyy^*\ell \\ xyvy^*\ell
 	\end{bmatrix} = 
  \begin{bmatrix}
 	    x(yy^*\ell) \\ x(u-yy^*w) 
 \end{bmatrix} 
 \in M_{2,1}(e_1Ax), 
 	\]
where we set $w:=v'\ell \in e_2Ae_2$. 
By the definition of $R_x$, it implies that 
 $yy^*\ell \in R_x$ and $u-yy^*w \in R_x$. 
Now, we have 
\begin{equation}\label{Mxxu}
    M_{2,1}(e_1 A e_2) = \begin{bmatrix}
        x \\ xu 
    \end{bmatrix} M_{1,1}(e_2 A e_2) + M_{2,1}(e_1 A x)
\end{equation}
by Proposition \ref{prop: lr=12}. By $u-yy^*w\in R_x$, we can write \eqref{Mxxu} as   
        \begin{equation}
           M_{2,1}(e_1 A e_2) = \begin{bmatrix}
               x \\xyy^*w
           \end{bmatrix} M_{1,1}(e_2 A e_2) + M_{2,1}(e_1 A x).  
        \end{equation}
By Proposition \ref{prop: lr=12} again, 
it implies that 
\begin{equation} \label{eq:xyyw}
P_{\left[\begin{smallmatrix}
    x \\ xyy^*w  
\end{smallmatrix}\right]} = 
\left[P_2\xrightarrow{\left[\begin{smallmatrix}
    x \\ xyy^*w 
\end{smallmatrix}\right]} P_1^{\oplus 2} \right]
\end{equation}
is a presilting complex and 
\begin{equation}\label{e2Ae2}
    e_2 A e_2 = R_x + yy^*w R_x. 
\end{equation}
By comparing their $g$-vectors, we find that the complex \eqref{eq:xyyw} is isomorphic to 
$P_{\left[\begin{smallmatrix}
    x \\ xu 
\end{smallmatrix}\right]}
$.

On the other hand, by $l_{12}=2$, we have
$
\left[\begin{smallmatrix}
1 \\ 0 \\ 0 
    \end{smallmatrix}\middle| \begin{smallmatrix}
2 \\ -1 \\ 0 
    \end{smallmatrix}\right] \in \Sigma(A,e)$ and 
$e_1A \oplus P_{\left[\begin{smallmatrix}
    x \\ xyy^*w 
\end{smallmatrix}\right]} 
$ is a two-term presilting complex for $A$. 
In particular, the following equation holds by \cite[Proposition\;3.9]{AHIKM2}: 
\begin{equation}\label{M11 i}
    M_{1,1}(e_1Ae_2) = 
    M_{1,1}(e_1Ax)  + M_{1,1}(e_1Ae_1)xyy^*w.
\end{equation}
By \eqref{M11 i}, for an element $xyy^*\in e_1Ae_2$, there exists $\ell_1\in e_1 A e_1$ such that 
\begin{equation} \label{e1Ax}
xyy^*-\ell_1 xyy^*w \in e_1 A x.
\end{equation}
Since $xy$ is a right generator of $e_1 A e_3$ and $y^*$ is a (left-)right generator of $e_3 A e_2$, we can take an elements $\ell_2\in e_3 A e_3$ such that
\[
\ell_1 xyy^*=xyy^* \ell_2.
\]
By \eqref{e1Ax}, we have
\begin{equation}
    x(yy^* - yy^*\ell_2w) = 
    xyy^* - xyy^*\ell_2w = 
    xyy^*-\ell_1 xyy^*w \in e_1Ax
\end{equation}	
and therefore 
\begin{equation}
yy^*(e_2 - \ell_2 w)\in R_x.
\end{equation}

In the above discussion, we have already shown that  
        \begin{itemize}
        \item $yy^*\ell \in R_x$, and 
        \item $yy^*(e_2 - \ell_2 w)\in R_x$. 
        \end{itemize}  
In the statement (a), we can take $\ell$ as an invertible element $\varepsilon$ if $\ell \not \in \rad(e_2 A e_2)$. Otherwise, we can take $e_2 - \ell_2 w= e_2 - \ell_2 v'\ell \not\in \rad(e_2Ae_2)$ as $\varepsilon$. Thus, the assertion (a) holds true. 

To prove (b), we assume that $yy^*\in R_x$. 
By \eqref{e2Ae2}, we can write $w\in e_2Ae_2$ as 
 	\[
 	w = p + yy^* w q 
 	\]
for some $p,q\in R_x$. 
By the induction on $k$, we can check that 
 	\begin{equation}\label{w ind}
 	    w-(yy^*)^k w q^k\in R_x
 	\end{equation}
hold for all positive integers $k>0$. 
Since $(yy^*)^N = 0$ for some $N>0$, we deduce that $w\in R_x$. 
Then, \eqref{e2Ae2} gives  
 	\begin{equation}
 	e_2 A e_2=R_x+yy^*w R_x = R_x.   
 	\end{equation}
It implies 
   	\[
 	e_1 A e_2 = x A e_2 = xR_x \subseteq e_1 A x \subseteq e_1 A e_2 \quad \text{and hence} \quad e_1Ae_2 = e_1Ax. 
 	\]
Thus, $x$ is also a left generator of $e_1 A e_2$. 
In particular, we have $l_{12} = 1$. 
However, this contradicts to our assumption that $l_{12} = 2$.
We finish a proof of (b). 
\end{proof}

\subsubsection{Proof of Proposition \ref{Gen X}(ii)}
\label{proof X2}
Finally, we prove Proposition \ref{Gen X}(ii). 
To do this, we assume that 
\begin{enumerate}
    \item[\rm (ii)] $d_{12} = (2,1,1)$, $d_{13} = (1,1,0)$, $d_{21} = (1,1,1)$, $d_{23} = (1,2,1)$ and $h_{32}=0$.
\end{enumerate}
Since $r_{12}=1$, we have a right generator $x$ of $e_1 A e_2$. By $(l_{12},r_{12})=(2,1)$, we can apply Proposition \ref{dual of prop: lr=12} to a right generator $x$ of $e_1Ae_2$ and obtain $u\in e_2 A e_2\setminus R_x$ such that 
$P_x\oplus P_{\left[\begin{smallmatrix}
    x \\ xu  
\end{smallmatrix}\right]}$ 
is presilting. These indecomposable two-term presilting complexes are given by 
\begin{equation*}
    P_x = \left[P_2\xrightarrow{x}P_1 \right]  
    \ \text{and} \
    P_{\left[\begin{smallmatrix}
    x \\ xu  
\end{smallmatrix}\right]} = \left[P_2\xrightarrow{\left[\begin{smallmatrix}
    x \\ xu  
\end{smallmatrix}\right]} P_1^{\oplus 2} \right]
\quad \text{with} \ \ 
g(P_x) = \left[\begin{smallmatrix}
    1 \\ -1 \\ 0  
\end{smallmatrix}\right] 
\ \text{and} \ 
g(P_{\left[\begin{smallmatrix}
    x \\ xu  
\end{smallmatrix}\right]}) = \left[\begin{smallmatrix}
    2 \\ -1 \\ 0  
\end{smallmatrix}\right] 
\end{equation*}

We can reach to a contradiction by the next claim.

\begin{lemma}
	\label{lemma:remaining case (b)}
	Under the above setting, we obtain the following contradictory claims.
	\begin{enumerate}[\rm (a)]
		\item Let $y$ be a left generator of $e_2 A e_3$. There exists an invertible element $\varepsilon$ in $e_2 A e_2$ such that 
  \begin{equation}
      x\varepsilon y =0. 
  \end{equation}
  In this case, $\varepsilon y$ is a left generator of $e_2Ae_3$. 
		\item For each left generator $y$ of $e_2 A e_3$,  we have 
  \begin{equation}
      xy \neq 0.  
  \end{equation} 
  \end{enumerate}
\end{lemma}

In fact, they cause a contradiction by the following reason. 
Since $l_{23}=1$, there exists a left generator, say $y$, of $e_2Ae_3$. 
Then, $x\varepsilon y =0$ by (a) but $x\varepsilon y\neq 0$ by (b) for some invertible element $\varepsilon$. 

To prove this, we need the following lemmas. 

\begin{lemma}\label{some presilt 2}
    Let $x$ and $u$ be as above. 
    Then, $P_{\left[\begin{smallmatrix}x \\ xu\end{smallmatrix}\right]}\oplus P_3[1]$ is presilting. In particular, we have 
    \begin{equation}\label{M21}
        M_{2,1}(e_1A_3) = \begin{bmatrix}
            x \\ xu
        \end{bmatrix} M_{1,1}(e_2Ae_3).
    \end{equation}
\end{lemma}

\begin{proof}
Let $g=(s,z) = ((12),-1)\in G$. 
Then, it satisfies $g \cdot (+--) = (+-+)$. 
By our assumption $d_{12} = (2,1,1)$ and $d_{13} = (1,1,0)$ with $h_{32}=0$, 
we obtain $(d(A^g,e^g))_{+-+} = d(10)$ and 
$\Sigma_{+--}(A,e) = g^{-1} \cdot \Sigma_{+-+}(A^g,e^g) = g^{-1} \cdot \Sigma_{d(10),0}$. 
Since 
$\left[
    \begin{smallmatrix}
        1 \\ -2 \\ 0 
    \end{smallmatrix}\middle|
    \begin{smallmatrix}
        0 \\ 0 \\ 1
    \end{smallmatrix}
    \right]\in \Sigma_{d(10),0}$, 
there is a cone 
\begin{equation}
    \left[
    \begin{smallmatrix}
        2 \\ -1 \\ 0 
    \end{smallmatrix}\middle|
    \begin{smallmatrix}
        0 \\ 0 \\ -1
    \end{smallmatrix}
    \right]\in \Sigma(A,e). 
\end{equation}
The corresponding two-term presilting complex for $A$ must be isomorphic to $P_{\left[\begin{smallmatrix}
    xy \\ xyv  
\end{smallmatrix}\right]}\oplus P_3[1]$.
\end{proof}

Next, let $y$ be a left generator of $e_2Ae_3$. 
Since $d_{13}=(1,1,0)$, we have $e_1Ae_3 = e_1Ae_2Ae_3 = xAy$. In particular, a left-right generator $z^*\in e_1Ae_3$ can be written as $z^*=xty$ for some $t\in e_2Ae_2$. 
Notice that $P_y$ and $P_{z^*}=P_{xty}$ are indecomposable two-term presilting complexes given by 
\begin{equation*}
    P_y := \left[P_3\xrightarrow{y}P_2 \right] 
    \ \text{and} \ 
    P_{z^*} := \left[P_3\xrightarrow{z^*}P_1 \right] 
    \quad \text{with} \ \ 
g(P_y) = \left[\begin{smallmatrix}
    0 \\ 1 \\ -1  
\end{smallmatrix}\right] 
\ \text{and} \ 
g(P_{z^*}) = \left[\begin{smallmatrix}
    1 \\ 0 \\ -1  
\end{smallmatrix}\right] 
\end{equation*}

\begin{lemma}\label{some presilt 3}
    Let $x,y,t,z^*$ be as above. 
    Then, $P_y \oplus P_{z^*}$ is presilting. In particular, we have 
    \begin{equation}\label{M11 ii}
        M_{1,1}(e_2Ae_3) = M_{1,1}(yAe_3) + M_{1,1}(e_2Ae_1)z^*. 
    \end{equation}
\end{lemma}

\begin{proof}
For an element $g=(s,z) = ((321),1)\in G$, we have $g\cdot (++-) = (+-+)$. 
By our assumption $d_{23}=(1,2,1)$ and $d_{13} = (1,1,0)$ with $h_{21}=1$, we have $(d(A^g,e^g))_{+-+} = d(10)$ and $\Sigma_{++-}(A,e) = g^{-1}\cdot \Sigma_{+-+}(A^g,e^g) = g^{-1}\cdot \Sigma_{d(10),1}$.
Since 
$\left[
    \begin{smallmatrix}
        1 \\ -1 \\ 0 
    \end{smallmatrix}\middle|
    \begin{smallmatrix}
        0  \\ -1 \\ 1
    \end{smallmatrix}
\right]\in \Sigma_{d(10),1}$, 
we have a cone 
\begin{equation}
    \left[
    \begin{smallmatrix}
        0 \\ 1 \\ -1 
    \end{smallmatrix}\middle|
    \begin{smallmatrix}
        1  \\ 0 \\ -1
    \end{smallmatrix}
    \right]\in \Sigma(A,e). 
\end{equation}
The corresponding two-term presilting complex for $A$ must be isomorphic to $P_{y}\oplus P_{z^*}$. 
\end{proof}

Now, we take a left-right generator $x^*$ of $e_2Ae_1$ (It exists by $r_{21}=l_{21}=1$). 

\begin{lemma}\label{lem:ell}
Let $x,y,u,t,x^*$ and $z^*=xty$ be as above. Then, there exists $\ell_2\in e_2 A e_2$ satisfying 
\begin{equation}\label{e2 and ell2}
        e_2 A e_2=L_y+L_y x^*x\ell_2.    
\end{equation}
\end{lemma}

\begin{proof}
By $(l_{23},r_{23})=(1,2)$, we can apply Proposition \ref{prop: lr=12} to a left generator $y$ of $e_2Ae_3$ and obtain $v\in e_2 A e_2\setminus L_y$ such that 
\begin{equation}\label{e2 to v ii}
    e_2 A e_2=L_y+L_y v. 
\end{equation} 
By \eqref{M11 ii}, for an element $vy\in e_2Ae_3$, there exists $w\in e_2 A e_1$ such that
	\begin{equation}
	(v-wxt)y = vy-w z^*\in y A e_3,    
	\end{equation}
where we use that $z^* = xty$. 
By the definition of $L_y$, it implies that $v-wxt \in L_y$. 
Using this, we can write \eqref{e2 to v ii} as 
\begin{equation}\label{e2 wxt}
    e_2 A e_2=L_y+L_y wxt.    
\end{equation}
To prove the claim, it suffices to show that there is $\ell_2 \in e_2Ae_2$ such that $wxt = x^*x\ell_2$. 
However, it is immediate from 
\begin{equation} \label{eq:x*x_right}
    wtx \in e_2 A e_1 A e_2 = x^*x Ae_2. 
\end{equation} 
 Indeed, we have the above equality by 
\begin{equation}
    e_2 A e_1 A e_2 = x^*Ae_1Ae_2 = x^*e_1(Ae_1A)e_2\subseteq x^* e_1Ae_2 =  x^*x Ae_2 \subseteq e_2 A e_1 A e_2, 
\end{equation}
where we use facts that $x$ and $x^*$ are right generators of $e_1Ae_2$ and $e_2Ae_1$ respectively.

It finishes a proof.
\end{proof}

Now, we are ready to prove Lemma \ref{lemma:remaining case (b)}. 

\begin{proof}[Proof of Lemma \ref{lemma:remaining case (b)}]
Let $x,y,u,t,x^*, z^*$ be as above. 
As a left generator $y$ of $e_2Ae_3$, we have $e_2Ae_3 = e_2Ay$. 
Then, the equation \eqref{M21} can be written as 
\begin{equation}
    M_{2,1}(e_1Ae_3) = \begin{bmatrix}
        x \\ xu
    \end{bmatrix} M_{1,1}(e_2Ay). 
\end{equation}
In particular, for $\begin{bsmallmatrix}
    0 \\ z^*
\end{bsmallmatrix}\in M_{2,1}(e_1Ae_3)$, there exists $\varepsilon\in e_2 A e_2$ such that 
\begin{equation}\label{z star}
    \begin{bmatrix}
        0 \\ z^*
    \end{bmatrix} = \begin{bmatrix}
        x\varepsilon y \\ xu \varepsilon y
    \end{bmatrix}.
\end{equation}
In this situation, $x\varepsilon y = 0$ and $z^* = xu\varepsilon y$ holds. 

To prove the assertions (a) and (b) in the statement, it is enough to check that $\varepsilon$ is invertible in $e_2Ae_2$ and 
$xy\ne 0$. 
Now, let $\ell_2\in e_2Ae_2$ be an element satisfying 
\begin{equation}\label{e2Ae2 ii}
    e_2 A e_2=L_y + L_y x^*x \ell_2
\end{equation} given in Lemma \ref{lem:ell}. 
Using this, for $u\varepsilon \in e_2Ae_2$, there are $p, q\in L_y$ such that
	\begin{equation}\label{uep}
	u \varepsilon = p +q x^*x \ell_2.
	\end{equation}
	By definition, $p\in L_y$ implies that there is an element $p'\in e_2 A e_2$ such that
	\begin{equation}\label{py}
	p y=yp'.
	\end{equation}
	In addition, since $z^*$ is a right generator of $e_1 A e_3$, for $x\ell_{2}y, xq x^* z^* \in e_1Ae_3$, we can take $\ell_3,\ell_3'\in e_3 A e_3$ such that
	\begin{equation}\label{left z}
	x\ell_2 y = z^* \ell_3, \quad  
    xq x^* z^* = z^*\ell_3'
	\end{equation}
    respectively.
	Then, we have
	\[
	\begin{array}{lll}
		z^* &\overset{\eqref{z star}}{=}& xu\varepsilon y\\
		&\overset{\eqref{uep}}{=}&xp y+xqx^*x\ell_2 y\\
		&\overset{\eqref{py}}{=}&xyp' + xqx^*x\ell_2 y\\
        &\overset{\eqref{left z}}{=}&xyp' + xqx^*z^*\ell_3 \\
		&\overset{\eqref{left z}}{=}&xyp' + z^*\ell'_3\ell_3,  
	\end{array}
    \]
	and therefore
 \begin{equation}
     z^*(e_3-\ell'_3\ell_3) = xyp'. 
 \end{equation}
 However, since $e_3-\ell'_3\ell_3$ is invertible in $e_3 A e_3$ by $\ell'_3\ell_3\in \rad(e_3 A e_3)$, we can write 
	\begin{equation}\label{z*}
	z^*=xyp'p'' \quad \text{with} \quad p'':= (e_3-\ell'_3\ell_3)^{-1}. 
	\end{equation}
	In particular, we obtain $xy\ne 0$. 
    In addition, \eqref{z*} gives 
    \begin{equation}
        z^* = xyp'p'' = z^*\ell_3'' p'p'' 
    \end{equation}
    for some $\ell_3''\in e_3Ae_3$ since $xy\in e_1Ae_3 = z^*Ae_3$ for a right generator $z^*$ of $e_1Ae_3$. 
    It implies that $p'\not \in \rad(e_3 A e_3)$, and $p\not \in \rad(e_2 A e_2)$ by $py = yp'$. 
    From \eqref{uep}, we conclude that 
	\begin{equation}
	u \varepsilon = p +q x^*x \ell_2\not \in \rad(e_2 A e_2).
	\end{equation}
	Therefore, $\varepsilon\not\in \rad(e_2 A e_2)$ and it is invertible in $e_2 A e_2$.
\end{proof}


\section*{Acknowledgments} 
T.A is supported by JSPS Grants-in-Aid for Scientific Research JP19J11408 and JSPS Grant-in-Aid for Transformative Research Areas (A) 22H05105. 
A.H is supported by JSPS Grant-in-Aid for Scientists Research (C) 20K03513.
O.I is supported by JSPS Grant-in-Aid for Scientific Research (B) 16H03923, (C) 18K3209 and (S) 15H05738. 
R.K is supported by JSPS Grant-in-Aid for Young Scientists (B) 17K14169. 
Y.M is supported by 
Grant-in-Aid for Scientific Research (C) 20K03539.

\end{document}